\documentclass[11pt]{amsart}

\usepackage[hmargin=3.5cm,top=3cm,bottom=3cm,includehead]{geometry}

\usepackage{amsmath, amssymb, amsthm, mathrsfs}
\usepackage{microtype}
\usepackage{graphicx, xcolor, enumerate} 
\usepackage[pdfborder={0 0 0}]{hyperref}

\usepackage[all,cmtip]{xy} 

\newcommand\E{{\mathbb E}}
\newcommand\N{{\mathbb N}}
\newcommand\R{{\mathbb R}}

\newcommand\Sp{{\mathbb S}}

\newcommand\PPP{\mathbf P}
\newcommand\RRR{\mathbf R}

\DeclareMathOperator{\dist}{dist}
\DeclareMathOperator{\supp}{supp}
\DeclareMathOperator{\Div}{div}

\def\BB{{\mathcal B}}

\def\EE{{\mathcal E}}
\def\FF{{\mathcal F}}
\def\GG{{\mathcal G}}
\def\HH{{\mathcal H}}
\def\II{{\mathcal I}}

\def\LL{{\mathcal L}}
\def\MM{{\mathcal M}}

\def\SS{{\mathcal S}}

\def\UU{{\mathcal U}}
\def\VV{{\mathcal V}}
\def\WW{{\mathcal W}}

\def\eps{{\varepsilon}}

\newcommand{\Black}{\color{black}}

\newcommand{\ps}[2]{\langle#1,#2\rangle}
\newcommand{\Ps}[2]{\left\langle#1,#2\right\rangle}
\newcommand{\norm}[1]{\lVert#1\rVert}
\newcommand{\Norm}[1]{\left\lVert#1\right\rVert}
\newcommand{\abs}[1]{\lvert#1\rvert}
\newcommand{\Abs}[1]{\left\lvert#1\right\rvert}
\newcommand{\Nt}{|\hskip-0.04cm|\hskip-0.04cm|}

\newcommand{\w}{\omega}

\newcommand{\wto}{\rightharpoonup}

\newcommand{\Psym}{\PPP_{\!\mathrm{sym}}}

\newcommand{\beqn}{\begin{equation}}
\newcommand{\eeqn}{\end{equation}}
\newcommand{\beqns}{\begin{equation*}}
\newcommand{\eeqns}{\end{equation*}}
\newcommand{\bear}{\begin{eqnarray}}
\newcommand{\eear}{\end{eqnarray}}
\newcommand{\bean}{\begin{eqnarray*}}
\newcommand{\eean}{\end{eqnarray*}}

\newcommand{\bal}{\begin{aligned}}
\newcommand{\eal}{\end{aligned}}

\newtheorem{thm}{Theorem}[section]

\newtheorem{lemma}[thm]{Lemma}

\theoremstyle{definition}
\newtheorem{definition}[thm]{Definition}
\theoremstyle{remark}
\newtheorem{ex}[thm]{Example}
\newtheorem{rem}[thm]{Remark}

\newtheoremstyle{a}
  {10pt}
  {10pt}
  {\normalfont}
  {}
  {\bfseries}
  {.}
  { }
  {\thmname{#1} (A\thmnumber{#2}) (\normalfont\thmnote{#3})\bfseries}
  
\theoremstyle{a}  
\newtheorem{assump}{Assumption} 

\numberwithin{equation}{section}
\numberwithin{figure}{section}

\author{Kleber Carrapatoso}
\title[Propagation of chaos for the Landau equation]{Propagation of chaos for the spatially homogeneous Landau equation for Maxwellian molecules}
\address{CEREMADE (UMR 7534), Universit\'e Paris Dauphine, Place du Mar\'echal De Lattre De Tassigny, 75775 Paris, France.}
\curraddr{CMLA (UMR 8536), \'Ecole Normale Sup\'erieure de Cachan, 61 av.\ du Pr\'esident Wilson, 94235 Cachan cedex, France.}
\email{carrapatoso@cmla.ens-cachan.fr}



\begin{document}

\begin{abstract}
We prove a quantitative propagation of chaos and entropic chaos, uniformly in time, for the spatially homogeneous Landau equation in the case of Maxwellian molecules. We improve the results of Fontbona, Gu\'erin and M\'el\'eard~\cite{FonGueMe} and Fournier~\cite{Fournier} where the propagation of chaos is proved for finite time. Moreover, we prove a quantitative estimate on the rate of convergence to equilibrium uniformly in the number of particles.
\end{abstract}

\maketitle

\tableofcontents

\section{Introduction}\label{sec:intro}

An important open problem in kinetic theory is to derive the Boltzmann equation from a many-particle system undergoing Newton's laws of dynamics. The correct scaling limit for this is the so-called Boltzmann-Grad or low density limit, see Grad \cite{Grad}. The best result is due to Lanford \cite{Lanford} who proved the limit for short times (see also Illner and Pulvirenti 
\cite{IllnerPulvirenti} and Gallagher, Saint-Raymond and Texier \cite{GST}). 

Kac \cite{Kac1956} proposed to derive the spatially homogeneous Boltzmann equation from a many-particle Markov process, performing a \emph{mean-field limit}. 
The program set by Kac in \cite{Kac1956} was then to investigate the behaviour of solutions of the mean-field equation in terms of the behaviour of the solutions of the \emph{master equation}, i.e.\ the equation for the law of the many-particle process. 
We refer to Mischler and Mouhot \cite{MMchaos} for a detailed introduction on Kac's program and for recent important results.


\medskip

In the same way, we would like to derive rigorously another equation from kinetic theory, the Landau equation, from a many-particle system described by Newton's laws. It is an open problem, but the correct scaling to this is also known, the weak-coupling limit, and we refer to Bobylev, Pulvirenti and Saffirio \cite{BPS} and the references therein for more information on this topic and partial results.  We do not pursue this problem here.

Instead, in this work, we shall use the approach described above introduced by Kac \cite{Kac1956}. Hence, we shall introduce a $N$-particle Markov process (see Section \ref{ssec:master}) from which we derive the spatially homogeneous Landau equation in the mean-field limit. 
The $N$-particle process used here is obtained by means of the grazing collisions limit applied to the $N$-particle master equation for the Boltzmann model. We should mention that the $N$-particle master equation introduced here was, in fact, originally proposed by Balescu and Prigogine in the 1950's (see \cite{KL} and references therein); and it is also studied by Kiessling and Lancelloti \cite{KL} and Miot, Pulvirenti and Saffirio \cite{MPS} (both in the Coulomb case).

\medskip

Let us briefly explain how we can prove the \emph{mean-field limit} with the approach proposed by Kac. Consider the probability density $F^N_t$ associated to the Landau $N$-particle system and its evolution equation, i.e.\ the master equation (Section~\ref{ssec:master}). Integrating this equation over all variables but the first, we obtain an evolution equation for the first marginal $\Pi_1(F^N_t)$ that depends on the second marginal $\Pi_2(F^N_t)$. If the second marginal of the probability density was the $2$-fold tensor product of a one-particle probability $f_t$, then $f_t$ would satisfy the Landau equation (Section~\ref{ssec:limit}). However, even if at initial time we start with an $N$-fold tensor probability 
$F^N(0)=f(0)^{\otimes N}$, this property can not be satisfied at later time because there are interactions between the particles. Kac suggested then that the chaos property (see definition below \eqref{def:chaos}), which is weaker than the tensor product property, could be propagated in time, which in turns would prove the mean-field limit.

%
%

\subsection{Known results}
Before giving our main results let us present known results concerning the propagation of chaos for the Landau equation for Maxwellian molecules.

The work of Fontbona, Gu\'erin and M\'el\'eard \cite{FonGueMe} consider a nonlinear diffusion processes driven by a white noise that have an interpretation in terms of PDEs corresponding to the Landau equation. 
They construct an $N$-particle system that converges, in the limit $N\to\infty$ and in finite time, to the nonlinear process and, moreover, obtain quantitative convergence rates in Wasserstein distance $W_2$. Then Fournier \cite{Fournier}, with the same probabilistic interpretation, improves the rate of convergence of \cite{FonGueMe}.

We should mention that the Landau master equation introduced in this work (Section \ref{ssec:master}) differs from the master equation associated to the $N$-particle process in \cite{FonGueMe,Fournier}, see Remark~\ref{rem:master}, indeed the master equation studied in this paper is conservative whereas that in \cite{FonGueMe,Fournier} is not.

\subsection{Main results}\label{ssec:results}

Consider a Polish space $E$, we shall denote by $\PPP(E)$ the space of probability measures on $E$. We denote by $\Psym(E^N)$ the space of symmetric probabilities on $E^N$. We say that a symmetric probability $F^N \in \Psym(E^N)$ is $f$-chaotic (or Kac chaotic), for some probability $f\in\PPP(E)$, if for all $\ell\in\N^*$ we have
\beqn\label{def:chaos}
F^N_\ell \wto f^{\otimes\ell}\quad\text{when}\quad N\to\infty,
\eeqn
where $F^N_\ell = \Pi_\ell(F^N)$ is the $\ell$-th marginal of $F^N$ and the convergence has to be understood in weak sense on $\PPP(E^\ell)$, i.e. the convergence of integral against continuous and bounded functions $\varphi \in C_b(E^\ell)$. In this paper we are interested in quantitative rates of convergence, more precisely we shall investigate estimates of the type, for any $\varphi\in\FF^{\otimes\ell}$ with $\FF\subset C_b(E)$ and $\|\varphi\|_{\FF^{\otimes\ell}}\le 1$,
$$
\left|  \left\langle   F^N_\ell - f^{\otimes\ell},\varphi \right\rangle\right| 
\le C(\ell) \,\eps(N),
$$
with a constant $C(\ell)$ possibly depending on $\ell$ and a function $\eps(N)\to 0$ when $N\to\infty$. Another possibility is to replace the left-hand side of the last 
equation by some distance on the space of probabilities, as for example the Wasserstein distance, $W_1(F^N_\ell, f^{\otimes\ell})$.

\medskip

The many-particle process can be considered in $\R^{dN}$ and then its law $F^N$ is a symmetric probability measure on $\R^{dN}$, however, thanks to the conservation laws, the process can be restricted to some sub-manifold of $\R^{dN}$. In our case, the dynamics of the many-particle process conserves momentum and energy (see Section~\ref{sec:landau} for details), which implies that the process can be restricted to 
\beqn\label{eq:sphere}
\SS^N(\EE,\MM) := \left\{ V=(v_1,\dots,v_N)\in \R^{dN}\;;\; \frac1N\, \sum_{i=1}^N |v_i-\MM|^2 = \EE,\;  \frac1N\, \sum_{i=1}^N v_i = \MM \right\}
\eeqn
where $\EE\ge 0$ and $\MM \in \R^d$. We consider through the paper, without loss of generality, the case $\MM=0$, we denote $\SS^N(\EE) := \SS^N(\EE,0)$ and call these sub-manifolds Boltzmann's spheres.

\subsubsection*{Initial data}
Considering the process in $\SS^N(\EE)$, we shall need an initial data $F^N_0 \in \Psym(\SS^N(\EE))$ that is $f_0$-chaotic for some $f_0\in \PPP(\R^d)$. This problem was studied in \cite{KC}, where it is proven that for some (regular enough) probability measure $f \in \PPP(\R^d)$, with zero momentum $\MM=\int v f=0$ and energy $\EE = \int |v|^2 f$, we can construct a probability measure $F^N \in \Psym(\SS^N(\EE))$ that is $f$-chaotic (and also entropically $f$-chaotic, see Section~\ref{sec:ent} for the definition), by taking the tensor product of $f$ and then restricting it to the Boltzmann's sphere $\SS^N(\EE)$. We shall denote this probability measure by
\beqn\label{eq:FN_SB}
F^N = \left[  f^{\otimes N} \right]_{\SS^N(\EE)} := \frac{f^{\otimes N}}{\int_{\SS^N(\EE)} f^{\otimes N} \, d\gamma^N} \, \gamma^N ,
\eeqn
where $\gamma^N$ is the uniform probability measure on $\SS^N(\EE)$.

\medskip

We can now state our main results in a simplified version.

\begin{thm}\label{thm:intro}
Consider $f_0\in  \PPP(\R^d)$, with zero momentum and energy $\EE$, and also $F^N_0 = [f_0^{\otimes N}]_{\SS^N(\EE)} \in \Psym(\SS^N(\EE))$. 
Let $f_t$ be the solution of the Landau equation for Maxwellian molecules (see \eqref{eq:landau}) with initial data $f_0$ and $F^N_t$ the solution of the Landau master equation for Maxwellian molecules (see \eqref{Lmaster2}) with initial data $F^N_0$.

\begin{enumerate}[(1)]

\item (Theorems~\ref{thm:chaos} and \ref{thm:Wchaos}) Then, for all $N\in\N^*$ we have
\begin{equation*}
\sup_{t\ge 0} \frac{W_1\left(F^N_t, f_t^{\otimes N} \right)}{N} \leq \eps_1(N),
\end{equation*}
where $\eps_1$ is a polynomial function and $\eps_1(N)\to 0$ as $N\to\infty$. 

Moreover, for all $t\ge 0$, for all $N\in \N^*$, we have
$$
\frac{W_1(F^N_t,\gamma^N)}{N} \le p(t),
$$
for a polynomial rate $p(t)\to 0$ as $t\to\infty$ and where $\gamma^N$ is the uniform probability measure on $\SS^N(\EE)$.

%
%

\medskip

\item (Theorem~\ref{thm:PropEntChaos}) Then, for all $N\in \N^*$ we have
$$
\sup_{t\ge0} \left| \frac1N\, H(F^N_t | \gamma^N) - H(f_t|\gamma)\right|
\le \eps_2(N),
$$
where $\eps_2$ is a polynomial function $\eps_2(N)\to 0$ as $N\to\infty$, $H(f|\gamma)$ denotes the relative entropy of $f_t$ with respect to $\gamma$, the centred Gaussian probability measure in $\R^d$ with energy $\EE$, and $H(F^N_t | \gamma^N)$ denotes the relative entropy of $F^N_t$ with respect to $\gamma^N$ (see Section~\ref{sec:ent}).

Moreover, for all $t\ge 0$, for all $N\in \N^*$, we have
$$
\frac1N \, H(F^N_t | \gamma^N) \le p(t),
$$
for some polynomial function $p(t)\to 0$ as $t\to\infty$.
\end{enumerate}

\end{thm}

\subsection{Strategy}

Since the work of Kac \cite{Kac1956}, propagation of chaos has been investigated by many authors and for different models, and we refer the reader to \cite{MMchaos} and the references therein for a discussion of several results and different methods related to this problem. 

This work is based on an abstract method recent developed by Mischler and Mouhot~\cite{MMchaos} to prove quantitative and uniform in time propagation of chaos for Boltzmann models, and also by Mischler, Mouhot and Wennberg~\cite{MMWchaos} to prove quantitative propagation of chaos for drift, diffusion and jump processes. We shall first generalise the method of \cite{MMchaos}, that we call ``consistency-stability method'', to the case of the Landau equation. This method reduces the propagation of chaos to some consistency and stability estimates for the semigroups and generators associated to the $N$-particle system and to the limit mean-field equation (see Section~\ref{sec:abs-thm} for the details). Then we shall investigate these consistency and stability estimates for the Landau equation in order to be able to apply the ``consistency-stability method'' built before and obtain a quantitative and uniform in time propagation of chaos. Finally, we investigate the propagation of entropic chaos, as a consequence of the previous estimate together with some results on chaotic probability measures from Hauray and Mischler~\cite{HaurayMischler} and from the author~\cite{KC}.

\subsection{Organisation of the paper} 
Section \ref{sec:landau} is devoted to deduce a $N$-particle stochastic process modelling the Landau dynamics and to present the limit Landau equation.
In Section \ref{sec:abs-thm} we adapt the consistency-stability method of \cite{MMchaos} (see also \cite{MMWchaos}), and we state an abstract theorem (Theorem~\ref{thm:abstract}). In Section~\ref{sec:app} we apply the method presented before to the Landau model in order to prove the propagation of chaos with quantitative rate and uniformly in time (Theorems~\ref{thm:chaos} and \ref{thm:Wchaos}). Finally, in Section~\ref{sec:ent} we prove a quantitative propagation of entropic chaos uniformly in time (Theorem~\ref{thm:PropEntChaos}).

\bigskip\noindent
{\bf Acknowledgements.} The author would like to thank St\'ephane Mischler and Cl\'ement Mouhot for suggesting the subject, their encouragement and fruitful discussions.


\section{The Landau model}\label{sec:landau}
Our aim in this section is to present the $N$-particle system and the limit mean-field equation corresponding to the Landau model. The limit Landau equation is well known and we shall present it in the Section~\ref{ssec:limit}. Furthermore, in Section~\ref{ssec:master} we deduce a master equation for the $N$-particle system corresponding to the Landau dynamics.

First of all, we present the Boltzmann model, with its master equation and limit equation, which will be very useful in the sequel since Boltzmann and Landau equations are linked by the asymptotic of grazing collisions that we shall explain in details later.

\subsection{The Boltzmann model}\label{ssec:boltzmann}
We present here the Boltzmann model, with the limit mean field equation and the master equation. The spatially homogeneous Boltzmann equation \cite{Villani-BoltzmannBook,MMchaos} is given by, for $f=f(t,v)$,
\beqn\label{eq:Boltzmann}
\partial_t f =Q(f,f),
\eeqn
with the collision operator defined by
\beqn\label{eq:Boltzmann2}
Q(g,f)= \int_{\R^d \times \Sp^{d-1}} B(|v-v_*|,\cos\theta) 
\Big( g(v_*')f(v') - g(v_*)f(v) \Big) \, dv_*\,d\sigma,
\eeqn
and where the post-collisional velocities $v'$ and $v_*'$ are given by
\begin{equation}\label{eq:velocities}
v' = \frac{v+v_*}{2} + \frac{|v-v_*|}{2}\sigma, \qquad
v_*' = \frac{v+v_*}{2} - \frac{|v-v_*|}{2}\sigma
\end{equation} 
and $\cos\theta= \sigma \cdot (v-v_*)/|v-v_*|$. 

We assume that the collision kernel $B$ satisfies $B(|z|,\cos\theta) = \Gamma(|z|) b(\cos\theta)$ (for more information on the collision kernel we refer to \cite{Villani-BoltzmannBook}) for some nonnegative functions $\Gamma$ and $b$. When the interaction potential is proportional to $r^{-s}$, where $r$ denotes the distance between particles, then we have
$$
\Gamma(|z|)=|z|^\gamma, \qquad \sin^{d-2}\theta\,b(\cos\theta) \sim C_b\, \theta^{-1-\nu}\;\text{when } \theta\sim 0,
$$
with $\gamma = (s-2d+2)/s$, for some constant $C_b>0$ and some fixed $\nu \in (0,2)$. For example, in the 3-dimensional case we have $\nu=2/s$.

In this work we are concerned with the case of \emph{true Maxwellian molecules} $\gamma=0$ (which corresponds to $s=2d-2$), we shall then consider through the paper the following assumption :
\begin{equation}\label{eq:Bmaxwell}
\bal
&B(|z|,\cos\theta)=b(\cos\theta) , \\
&\int_{\Sp^{d-1}} b(\cos\theta)(1-\cos\theta)^{\alpha+1/4}\, d\sigma < + \infty, \quad \forall\,\alpha>0.
\eal
\end{equation}
We remark that in this case we have $\int_{\Sp^{d-1}} b(\cos\theta)\, d\sigma = \infty$ but 
$\int_{\Sp^{d-1}} b(\cos\theta)(1-\cos\theta)\, d\sigma <\infty$.

Another possible way to describe the pre and post-collisional velocities is the $\omega$-representation (see \cite{Villani-BoltzmannBook})
\begin{equation}\label{w-rep}
v' = v- (v-v_*,\omega) \omega , \quad v_*' = v_* + (v-v_*,\omega)\omega, \quad \omega\in \Sp^{d-1},
\end{equation}
which gives us
$$
Q_B(f,f) = \int_{\R^d \times \Sp^{d-1}} \widetilde B(|v-v_*|,\omega) 
\Big( f(v_*')f(v') - f(v_*)f(v) \Big) \, dv_*\,d\omega,
$$
where $\widetilde B(z,\omega) = |2 ( z/|z|, \omega )|^{d-2}\, B(z,\sigma) = |z|^\gamma b_{\omega}(\alpha)$ and $\alpha$ is the angle formed by $z$ and $\omega$, and the following relation holds
\beqn\label{eq:sigmaw}
\sigma = \frac{v-v_*}{|v-v_* |} - 2 \left(\omega,\frac{v-v_*}{|v-v_* |} \right) \omega.
\eeqn

\bigskip

Let us now present the many-particle model \cite{McKean,Kac1956,MMchaos,CCLLV,MMWchaos}.
Given a pre-collisional system of velocities $V=(v_1,\dots,v_N)\in \R^{dN}$ and a collision kernel $B(|z|,\cos\theta) = \Gamma(|z|) b(\cos\theta)$, the process is:
for any $i' \ne j'$, pick a random time $T(\Gamma(|v_{i'} - v_{j'}|))$ of collision accordingly to an exponential law of parameter $\Gamma(|v_{i'} - v_{j'}|)$ and choose the minimum time $T_1$ and the colliding pair $(v_i,v_j)$; draw $\sigma\in \Sp^{d-1}\subset \R^d$ according to the law $b(\cos\theta_{ij})$, with $\cos\theta_{ij} = \sigma \cdot (v_i-v_j)/|v_{i} - v_{j}|$; after collision the new velocities become $V_{ij}' = ( v_1,\dots, v_i',\dots,v_j',\dots,v_N)$ with
\beqn\label{eq:vitesses}
v_i' = \frac{v_i+v_j}{2} + \frac{|v_i-v_j|}{2}\sigma, \qquad
v_j' = \frac{v_i+v_j}{2} - \frac{|v_i-v_j|}{2}\sigma.
\eeqn

Iterating this construction we built then the associated Markov process $(\VV_t)_{t\ge 0}$ on $\R^{dN}$. As explained in the introduction, we can also consider this process on $\SS^N(\EE)$.
Then, after a rescaling of time and denoting $F^N_t$ the law of $\VV_t$, the master equation is given in dual form by \cite{MMchaos,MMWchaos},
\beqn\label{eq:MasterBoltzmann}
\bal
\partial_t \ps{F^N_t}{\varphi} &= \ps{F^N_t}{G^N_B \varphi} 
\eal
\eeqn
where, for any $V \in \R^{dN}$,
\beqn\label{eq:BoltzmannGenerator}
\bal
G^N_B \varphi (V) &= \frac{1}{2N}\sum_{i,j=1}^N \Gamma(|v_i-v_j|) \int_{\Sp^{d-1}} b(\cos\theta_{ij}) \left( \varphi_{ij}' - \varphi\right)\, d\sigma
\eal
\eeqn
with the shorthand notation $\varphi_{ij}' = \varphi(V_{ij}')$ and $\varphi=\varphi(V) \in C_b(\R^{dN})$. We shall consider the case of true Maxwellian molecules, i.e.\ $\Gamma(|z|)=1$ and $b(\cos\theta)$ satisfying \eqref{eq:Bmaxwell}.

\subsection{Limit equation}\label{ssec:limit}
We present here the limit spatially homogeneous Landau equation for Maxwellian molecules and some of its basic properties, for more information we refer to \cite{Villani-BoltzmannBook,Vi1,Vi2}.

The Landau equation is a kinetic model in plasma physics that describes the evolution of the density $f$ of a gas in the phase space of all positions and velocities of particles. Assuming that the density function does not depend on the position, we obtain the \emph{spatially homogeneous Landau equation} in the form
\begin{equation}\label{eq:landau}
\partial_t f  = Q_L(f,f),
\end{equation}
where $f=f(t,v)$ is the density of particles with velocity $v \in \R^d$ at 
time $t \in \R^+$. The Landau operator is given by
\begin{equation}\label{eq:landau2}
Q_L(g,f) = \partial_{\alpha} \left\{ \int_{\R^d} a_{\alpha \beta}(v-v_*)\left( g(v_*)\partial_{\beta} f(v) - \partial_{\beta} g(v_*) f(v) \right) \, dv_* \right\},
\end{equation}
where here and below we shall use the convention of implicit summation over indices.
Moreover, the matrix $a$ is nonnegative, symmetric and depends on the interaction between particles. If two particles interact with a potential proportional to $1/r^s$, where $r$ denotes their distance, then we have
$$
a_{\alpha \beta} (z)= \Lambda |z|^{\gamma+2} \Pi_{\alpha \beta}(z),\qquad \Pi_{\alpha \beta}(z) = \delta_{\alpha\beta} - \frac{z_\alpha z_\beta}{|z|^2},
$$
with $\gamma=(s-2d+2)/s$ and some constant $\Lambda \in (0,\infty)$.
As for the Boltzmann equation, we only consider the case of Maxwellian molecules $\gamma=0$, i.e.
\beqn\label{fonction-a}
a_{\alpha\beta} (z)= \Lambda |z|^{2} \Pi_{\alpha\beta}(z).
\eeqn

We also define
\begin{equation}\label{fonctions-bc}
\begin{aligned}
&b_\alpha(z)= \partial_\beta  a_{\alpha\beta} (z) = -\Lambda (d-1) z_\alpha,
 \qquad
c(z) = \partial_{\alpha\beta}  a_{\alpha\beta} (z) = -3\Lambda(d-1),
\end{aligned}
\end{equation}
and we denote
\begin{equation*}\label{fonctions-abc_convolution}
\begin{aligned}
\bar a_{\alpha\beta}=a_{\alpha\beta} * f,\;
\bar b_\alpha=b_\alpha * f,\;
\bar c=c * f.
\end{aligned}
\end{equation*}
Hence, we can write the Landau equation in another form
\begin{equation}\label{landau2}
\begin{aligned}
\partial_t f = \bar a_{\alpha\beta} \partial_{\alpha\beta} f -
               \bar c f.
\end{aligned}
\end{equation}
Moreover, let $\varphi = \varphi(v)$ be a test function, then we have the following weak forms
\begin{equation}\label{landau-faible0}
\begin{aligned}
\int Q_L(f,f)\varphi 
=&-\frac{1}{2}\int dv \,dv_* \, ff_* \, a(v-v_*)
\left( \frac{\nabla f}{f} -  \frac{\nabla_{\!*} f_*}{f_*} \right)
(\nabla \varphi - \nabla_{\!*} \varphi_*)
\end{aligned}
\end{equation}
or
\begin{equation}\label{landau-faible}
\begin{aligned}
\int Q_L(f,f)\varphi 
=&\frac{1}{2}\int dv \, dv_* \, ff_* \, a_{\alpha\beta}(v-v_*)(\partial_{\alpha\beta} \varphi +(\partial_{\alpha\beta} \varphi )_*) \\
&+ \int dv \, dv_* \, ff_* \, b_{\alpha}(v-v_*)(\partial_{\alpha} \varphi -(\partial_\alpha \varphi)_*),
\end{aligned}
\end{equation}
where hereafter we use the notation $f = f(v)$, $f_* = f(v_*)$, $\nabla f = \nabla f(v)$, $\nabla_* f_* = \nabla f(v_*)$ and for a matrix $a$ and vectors $z,w$ we denote $a z w = a_{\alpha \beta} z_\alpha  w_\beta$.

This equation satisfies the conservation of mass, momentum and energy. 
Moreover, the entropy $H(f)=\int f\log f$ is non-increasing, indeed taking $\varphi=\log f$ we obtain
\begin{equation}\label{A-landau}
\begin{aligned}
\frac{d}{dt} H(f) 
&=: - D(f) \\
&=
-\frac{1}{2}\int ff_* \,a(v-v_*)
\left( \frac{\nabla f}{f} -  \frac{\nabla_{\!*} f_*}{f_*} \right)
\cdot \left( \frac{\nabla f}{f} -  \frac{\nabla_{\!*} f_*}{f_*} \right) \,dv \,dv_*
\leq 0,
\end{aligned}
\end{equation}
since $a$ is nonnegative, which is the Landau version of the $H$-theorem of Boltzmann. For more information we refer to \cite{Vi1}.

The Landau equation was introduced by Landau in 1936. Later it was proven that the Landau equation can be obtained as a limit of the Boltzmann equation when grazing collisions prevail (see \cite{Des,ArsenevBuryak,DegondLucquin,Vi2} and the references therein for more details).

\subsection{Master equation}\label{ssec:master}
We derive a master equation for the Landau model. 
It is based on \cite{Vi2} where they use the grazing collisions limit to pass from Boltzmann to Landau limit equations (see also \cite{Des,ArsenevBuryak,DegondLucquin}). Since we know the master equation for the Boltzmann model \eqref{eq:MasterBoltzmann}, we shall perform the grazing collisions limit to obtain a master equation for the Landau model. As explained in the introduction, the master equation we derive here (see \eqref{Lmaster2}) is the same as introduced by Balescu and Prigorine in the 1950's, and it is also studied in the works \cite{KL,MPS}.

\subsubsection*{Grazing collisions}We present here the grazing collision limit as in \cite{Vi2}. Consider the true Maxwellian molecules collision kernel $b$ satisfying \eqref{eq:Bmaxwell}.
We say that $b_\eps$ concentrates on grazing collisions if: 
\begin{equation}\label{eq:grazing}
\left\{
\begin{array}{c}
\forall\,\theta_0>0, \qquad\sup_{\theta>\theta_0} b_{\eps}(\cos\theta) \to 0  \quad 
\mbox{when}\quad \varepsilon \to 0 \\\\
\Lambda_\eps = \displaystyle  \int_{\Sp^{d-1}} b_{\eps}(\cos\theta) (1 -\cos\theta) \, d\sigma
\to  \Lambda \in (0, \infty) \quad 
\mbox{when}\quad \varepsilon \to 0 .  
\end{array}
\right.
\end{equation}

\medskip

For the sake of simplicity, to derive the Landau master equation in this section, we suppose  the dimension $d=3$ to perform the computations following \cite{Vi2}, the other cases being the same.

From \eqref{eq:sigmaw}, using a spherical coordinate system (in dimension $d=3$) with axis $v-v_*$, we have
$$
\sigma = \frac{v-v_*}{|v-v_* |}\cos\theta + (\cos\phi \,\vec h + \sin\phi \,\vec i)\sin\theta.
$$ 
Moreover we have $|(v-v_*,\omega)| = |v-v_*| \sin(\theta/2)$. Finally we can write the operator in the following way (see \cite{Vi2})
\beqn\label{Qgrazing}
Q_B(f,f) = \int_0^{2\pi}d\phi \int_{0}^{\pi/2}d\theta\int_{\R^d} dv_* \, \zeta(\theta)(f' f_*' - ff_*),
\eeqn
with $\zeta(\theta) = \sin^{d-2}\theta \, b(\cos\theta)$.
In this case, we can rewrite \eqref{eq:grazing} and say that $\zeta_\varepsilon$ concentrates on grazing collisions if for all 
$\theta_0 \geq 0$ it hold
\begin{equation}\label{rasantes}
\left\{
\begin{array}{c}
 \sup_{\theta\ge \theta_0}\zeta_\varepsilon(\theta) \to 0 \quad 
\mbox{when}\quad \varepsilon \to 0  \\\\
\displaystyle
\Lambda_\eps := \frac{\pi}{2}\int_{0}^{\pi/2} 
\sin^2\frac{\theta}{2}\,\zeta_\varepsilon(\theta) \, d\theta\to \Lambda \in(0,\infty) \quad 
\mbox{when}\quad \varepsilon \to 0   .
\end{array}
\right.
\end{equation}

Let us consider then the Boltzmann master equation \eqref{eq:MasterBoltzmann}-\eqref{eq:BoltzmannGenerator}, using the form of \eqref{Qgrazing}, that is
$$
\begin{aligned}
(G_B^N\varphi)(V)  &= \frac{1}{2N}\sum_{i,j=1}^N \int_{0}^{2\pi}\!\!\!d\phi \int_{0}^{\pi/2}\!\!\!d\theta\, 
\zeta(\theta)(\varphi_{ij}' - \varphi).
\end{aligned}
$$
In this equation, we take a second order Taylor expansion of $\varphi_{ij}'$ and obtain
\begin{equation}\label{taylor0}
\begin{aligned}
\varphi(V_{ij}') - \varphi(V)=&D\varphi[V](V_{ij}' - V) 
          +  \frac{1}{2}(V_{ij}' - V)^T D^2\varphi[V](V_{ij}' - V) \\ 
          &+ O(|V_{ij}' - V|^3) .
\end{aligned}
\end{equation}
With the incoming and outgoing velocities $V$ and $V_{ij}'$ (see \eqref{eq:vitesses}), we have
$$
V_{ij}' - V = (0,\dots , 0,v_i'-v_i,0,\dots ,0, v_j'-v_j,0,
\dots ,0).
$$
In \eqref{taylor0}, $D\varphi[V]$ and $D^2\varphi[V]$ are given by 
$$
D\varphi[V] = \left( \nabla_{i}\varphi(V) \right)_{1\leq i \leq N} , \quad
D^2\varphi[V] = \left( \nabla^{2}_{\!ij}\varphi(V) \right)_{1\leq i,j \leq N}
$$
where $\nabla_{i}\varphi = \left( \partial_{i,\alpha} \varphi \right)_{1\leq\alpha\leq 3}$ and 
$\nabla^{2}_{ij}\varphi = \left( \partial_{v_{i,\alpha}}\partial_{v_{j,\beta}} \varphi \right)_{1\leq\alpha,\beta\leq 3}$.
Now we substitute $V_{ij}' - V$ in \eqref{taylor0} and we get
\begin{equation}
\begin{aligned}
\varphi(V_{ij}') - \varphi(V)
 =& \nabla_{i}\varphi(V)(v_i'-v_i)+\nabla_{j}\varphi(V)(v_j'-v_j) \\
     &+  \frac{1}{2}\left\{ \nabla^2_{ii}\varphi(V)(v_i'-v_i)^2 +\nabla^2_{jj}\varphi(V)(v_j'-v_j)^2\right. \\
&+\left.\nabla^{2}_{\!ij}\varphi(V)(v_i'-v_i)(v_j'-v_j) +\nabla^{2}_{\!ji}\varphi(V)(v_i'-v_i)(v_j'-v_j) \right\}\\
&+O(|V_{ij}' - V|^3).
\end{aligned}
\end{equation}
Finally, using \eqref{eq:vitesses} and \eqref{w-rep} with $v_i$ and $v_j$, one obtains
\begin{equation}\label{taylor}
\begin{aligned}
\varphi(V_{ij}') - \varphi(V)
 =& -(v_i-v_j,\omega)(\nabla_{i}\varphi-\nabla_{j}\varphi,\omega) &(=T_1)\\
     &+  \frac{1}{2}(v_i-v_j,\omega)^2
\left\{ \nabla^2_{ii}\varphi +\nabla^2_{jj}\varphi
-\nabla^{2}_{\!ij}\varphi -\nabla^{2}_{\!ji}\varphi\right\}(\omega,\omega) &(=T_2)\\
&+O\left(|v_i - v_j|^3\sin^3 \frac{\theta}{2}\right).
\end{aligned}
\end{equation}
For each pair of particles $i$ and $j$, in the orthonormal basis $\{\frac{v_i-v_j}{|v_i-v_j|},\vec h,\vec i \}$, one has
$$
\omega = \frac{v_i-v_j}{|v_i-v_j|} \sin \frac{\theta}{2} + 
(\cos\phi \,\vec{h} + \sin\phi \,\vec{i})\cos \frac{\theta}{2}
$$
and then, using the fact that linear combinations of $\cos\phi$ and $\sin\phi$ vanish when integrated over $\phi$, we can compute the contribution of $T_1$ integrated over $\phi$
$$
-\int^{2\pi}_{0}d\phi (v_i-v_j,\omega)(\nabla_{i}\varphi-\nabla_{j}\varphi,\omega)
= -2\pi \sin^2\frac{\theta}{2} \left(v_i-v_j,\nabla_{i}\varphi-\nabla_{j}\varphi \right).
$$
Now we have to compute the integral of $T_2$ over $\phi$, we denote 
$$
\lambda_{\alpha\beta} = \{ \partial_{v_{i,\alpha}}\partial_{v_{i,\beta}} \varphi  + 
\partial_{v_{j,\alpha}}\partial_{v_{j,\beta}} \varphi  - 
\partial_{v_{i,\alpha}}\partial_{v_{j,\beta}} \varphi  -
\partial_{v_{j,\alpha}}\partial_{v_{i,\beta}} \varphi  \} /2
$$
and in the same orthonormal basis, we compute
$$
A=|v_i-v_j|^2\sin^2\frac{\theta}{2} \int_0^{2\pi}d\phi \lambda_{\alpha\beta}\omega_{\alpha}\omega_{\beta}.
$$
Again, linear combinations of $\cos\phi$ and $\sin\phi$ vanish, which implies
\begin{equation}\label{taylor2}
\begin{aligned}
A &= |v_i-v_j|^2\sin^2\frac{\theta}{2}\int_0^{2\pi}d\phi \left( 
\lambda_{11}\sin^2\frac{\theta}{2} +
\lambda_{22}\cos^2\phi\cos^2\frac{\theta}{2} +
\lambda_{33}\sin^2\phi\cos^2\frac{\theta}{2}
\right) \\
&= 2\pi |v_i-v_j|^2 \left(  \lambda_{11}\sin^4\frac{\theta}{2} +
\frac{\lambda_{22}}{2}\cos^2\frac{\theta}{2}\sin^2\frac{\theta}{2} +
\frac{\lambda_{33}}{2}\cos^2\frac{\theta}{2}\sin^2\frac{\theta}{2}
\right)
\end{aligned}
\end{equation}
and we remark that the first coefficient is of order greater than $2$ in $\theta$.

We introduce $\Pi_{\alpha\beta}(v_i-v_j)$ the projection over the orthogonal space of $\frac{v_i-v_j}{|v_i-v_j|}$, and the dominant term of \eqref{taylor2} when $\theta\to 0$ is

$$
\pi |v_i-v_j|^2 \sin^2\frac{\theta}{2} \Pi_{\alpha\beta}(v_i-v_j) \lambda_{\alpha\beta},
$$
or in matricial notation
$$
\pi |v_i-v_j|^2 \sin^2\frac{\theta}{2} \Pi(v_i-v_j) : \frac{\left(
\nabla^2_{ii}\varphi +\nabla^2_{jj}\varphi -\nabla^2_{ij}\varphi -\nabla^2_{\!ji}\varphi
\right)}{2}.
$$
Finally, we obtain
\begin{equation}\label{taylor3}
\begin{aligned}
&\frac{1}{2}\int^{2\pi}_{0}d\phi(\varphi(V_{ij}') - \varphi(V)) =
-\frac{\pi}{2} \sin^2\frac{\theta}{2} \left(2(v_i-v_j),\nabla_{i}\varphi-\nabla_{j}\varphi \right)\\ 
&\qquad\qquad 
+\frac{\pi}{4} |v_i-v_j|^2 \sin^2\frac{\theta}{2} \Pi(v_i-v_j) : \left(
\nabla^2_{ii}\varphi +\nabla^2_{jj}\varphi -\nabla^2_{ij}\varphi -\nabla^2_{\!ji}\varphi \right) \\ 
&\qquad \qquad
+ O\left( |v_i-v_j|^2\theta^4 \wedge 1 \right) + 
O\left( |v_i-v_j|^3\theta^3 \wedge 1 \right).
\end{aligned}
\end{equation}

Consider now the Boltzmann master equation with kernel $\zeta_\eps$ satisfying the grazing collisions \eqref{rasantes} and plug \eqref{taylor3} in it, we obtain then
\begin{equation}
\begin{aligned}
G_B^N\varphi (V)
&= \frac{1}{N}\sum_{i,j=1}^N\int_{0}^{\pi/2}\!\!\!d\theta \,
\zeta_\varepsilon(\theta)\frac{1}{2}\int_{0}^{2\pi}\!\!\!d\phi(\varphi_{ij}' - \varphi) \\
&= 
\frac{1}{N}\sum_{i,j=1}^N\int_{0}^{\pi/2}\!\!\!d\theta \,
                \zeta_\varepsilon(\theta)
\left\{
-\frac{\pi}{2} \sin^2\frac{\theta}{2} \left(2(v_i-v_j),\nabla_{i}\varphi-\nabla_{j}\varphi \right)    \right. \\ 
&\quad 
+ \frac{\pi}{4} |v_i-v_j|^2 \sin^2\frac{\theta}{2} \Pi(v_i-v_j) : \left(
\nabla^2_{ii}\varphi +\nabla^2_{jj}\varphi -\nabla^2_{ij}\varphi -\nabla^2_{\!ji}\varphi \right) \\ 
&\quad 
+ O\left( |v_i-v_j|^2\theta^4 \wedge 1 \right) + 
O\left( |v_i-v_j|^3\theta^3 \wedge 1 \right)
\bigg\}.
\end{aligned}
\end{equation}
This can be written in the following way
\begin{equation}
\begin{aligned}
G_B^N\varphi (V) &= 
\frac{1}{N}\sum_{i,j=1}^N \frac{\pi}{2}\int_{0}^{\pi/2}\!\!\!d\theta \,
              \sin^2\frac{\theta}{2}  \,\zeta_\varepsilon(\theta)\,
  (-2|v_i-v_j|^{2}\frac{(v_i-v_j)}{|v_i-v_j|^2})\cdot(\nabla_{i}\varphi-\nabla_{j}\varphi )     \\ 
&\quad + \frac{1}{2N}\sum_{i,j=1}^N \frac{\pi}{2}\int_{0}^{\pi/2}\!\!\!d\theta \,
 \sin^2\frac{\theta}{2}    \, \zeta_\varepsilon(\theta)\,
  |v_i-v_j|^{2}\Pi(v_i-v_j) : \left(
\nabla^2_{ii}\varphi +\nabla^2_{jj}\varphi -\nabla^2_{ij}\varphi -\nabla^2_{\!ji}\varphi \right) \\ 
&\quad + \frac{1}{N}\sum_{i,j=1}^N\int_{0}^{\pi/2}\!\!\!d\theta \,
                \zeta_\varepsilon(\theta) \left(
O\left( |v_i-v_j|^2\theta^4 \wedge 1 \right) + 
O\left( |v_i-v_j|^3\theta^3 \wedge 1 \right)\right).
\end{aligned}
\end{equation}
As in \cite{Vi2}, the last term converges to $0$ when $\eps\to 0$. Then we have, 
using \eqref{rasantes} and the definition of the functions 
$a$ in \eqref{fonction-a}  
and 
$b$ in \eqref{fonctions-bc},  
when $\varepsilon\to 0$,
\begin{equation}\label{eq:generateurLandau}
\begin{aligned}
G_B^N\varphi (V) &\longrightarrow 
\frac{1}{N}\sum_{i,j=1}^N 
  -2\Lambda|v_i-v_j|^{2}\frac{(v_i-v_j)}{|v_i-v_j|^2} \cdot(\nabla_{i}\varphi-\nabla_{j}\varphi )     \\ 
&\quad 
+ \frac{1}{2N}\sum_{i,j=1}^N \Lambda
  |v_i-v_j|^{2}\Pi(v_i-v_j) : \left(
\nabla^2_{ii}\varphi +\nabla^2_{jj}\varphi -\nabla^2_{ij}\varphi -\nabla^2_{\!ji}\varphi \right) \\ 
&=\frac{1}{N}\sum_{i,j=1}^N 
 b(v_i-v_j) \cdot(\nabla_{i}\varphi-\nabla_{j}\varphi )     \\ 
&\quad
+ \frac{1}{2N}\sum_{i,j=1}^N a(v_i-v_j) : \left(
\nabla^2_{ii}\varphi +\nabla^2_{jj}\varphi -\nabla^2_{ij}\varphi -\nabla^2_{\!ji}\varphi \right) =: G^N_L \varphi
\end{aligned}
\end{equation}
and that defines the Landau generator $G_L^N$.
Finally, we derive the following Landau master equation
\begin{equation}\label{Lmaster2}
\begin{aligned}
\partial_t \ps{F^N_t}{\varphi} = \left\langle F^N_t, G^N_L \varphi \right\rangle &=  
\frac{1}{N}\int \sum_{i,j=1}^{N}b(v_i-v_j)\cdot(\nabla_i\varphi - \nabla_j \varphi) F^N_t (dV)\\
&\quad
+\frac{1}{2N}\int \sum_{i,j=1}^{N}
a(v_i-v_j):(\nabla^2_{ii}\varphi +\nabla^2_{jj}\varphi -\nabla^2_{ij}\varphi -\nabla^2_{\!ji}\varphi) F^N_t (dV).
\end{aligned}
\end{equation}

\begin{rem}\label{rem:master}
In the paper \cite{FonGueMe}, the Landau equation is studied with a
probabilistic approach. 
In particular they prove that the following process associated to a $N$-particle system, for $i=1,...,N$,
\begin{equation}\label{Nlandau-processus}
dX_{t}^{i}=\frac{\sqrt 2}{\sqrt{N}} 
            \sum_{k=1}^{N} \sigma(X_{t}^{i}-X_{t}^{k})\,dB_{t}^{i,k} 
+ \frac{2}{N} \sum_{k=1}^{N} b(X_{t}^{i}-X_{t}^{k})\,dt,
\end{equation}
where $B^{i,k}$ are $N^2$ independent $\R^d$-valued Brownian motions, converges to the process 
\begin{equation}\label{landau-processus}
X_{t}=X_{0}+ \sqrt2\int_{0}^{t} \int_{\R^d} \sigma(X_{s}-y) W(dy,ds)
+ 2\int_{0}^{t}\int_{\R^d}  b(X_{s}-y)P_s(dy)ds
\end{equation}
where $P_t$ is the law of $V_t$ and $W$ is a white noise in space-time. Moreover the process \eqref{landau-processus} is associated to the spatially homogeneous Landau equation with the coefficients 
$a_{\alpha\beta}(z)~:=~\left(\sigma\sigma^*\right)_{\alpha\beta}(z)$ and
$b_{\alpha}(z):= \partial_{\beta}a_{\alpha\beta}(z) $.
Then, the Kolmogorov equation of \eqref{Nlandau-processus} is, for a test function $\varphi : \R^{dN}\longrightarrow \R$ and where $F^N_t$ represents the law of $X_t$,
\begin{equation}\label{Lmaster1}
\begin{aligned}
\partial_t \ps{F^N_t}{\varphi}&= \ps{F^N_t}{G^N_2\varphi}\\
&= \frac{1}{N}\sum_{i,j=1}^{N}  \int_{\R^{dN}} 
b(v_{i}-v_{j})\cdot \left(\nabla_{i}\varphi(V) -\nabla_{j}\varphi(V)\right) F^N_t (dV)\\
&\quad
+ \frac{1}{2N}\sum_{i,j=1}^{N} \int_{\R^{dN}} 
 a(v_{i}-v_{j}) : \left(\nabla^2_{ii}\varphi(V)+\nabla^2_{jj}\varphi(V)\right) 
F^N_t (dV).
\end{aligned}
\end{equation}
In \cite{Fournier}, instead of the process \eqref{landau-processus} it was used a similar process with only $N$ independent Brownian motion, i.e.\ replacing $B^{i,k}$ by $B^i$ in \eqref{landau-processus}, and it gives the same master equation \eqref{Lmaster1}. We remark that this equation differs from \eqref{Lmaster2} by the terms 
$
\sum_{i,j} a(v_i-v_j) : \left( -\nabla^2_{ij}\varphi -\nabla^2_{\!ji}\varphi \right).
$
We will see later in Lemma~\ref{lem:A1i0} and Remark~\ref{rem:conservative} that \eqref{Lmaster2} is conservative whereas \eqref{Lmaster1} is not.

\end{rem}

\smallskip

Now, in order to obtain a $N$-particle SDE associated to \eqref{eq:generateurLandau}-\eqref{Lmaster2}, we shall modify \eqref{landau-processus}. Consider then, for $i=1,\dots,N$,
$\R^d$-valued random variables $(X^i_t)_{t\ge 0}$ satisfying the following equation
\beqn\label{eds0}
\forall\, i=1,\dots,N,
\quad
dX_{t}^{i}=\frac{\sqrt2}{\sqrt{N}} 
            \sum_{ \substack{k=1\\ k\neq i}}^{N} \sigma(X_{t}^{i}-X_{t}^{k})\,dZ_{t}^{i,k} 
+ \frac{2}{N}\sum_{\substack{k=1\\ k\neq i}}^{N} b(X_{t}^{i}-X_{t}^{k}) \, dt
\eeqn
where, for all $1\le i\le N$ and $i<k$, $Z^{i,k}_t=B^{i,k}_t$ are $N(N-1)/2$ independent $\R^d$-valued Brownian motions and the other terms are anti-symmetric $Z^{k,i}_t = - B^{i,k}_t$. As in \eqref{landau-processus}, we have $a(z)=\sigma(z)\sigma^*(z)$ and $\sigma$ is symmetric (recall that $a$ also is), i.e. $\sigma(-z)=\sigma(z)$. 
%
%
Consider a test function $\phi:\R^{dN}\to \R$ and let $F^N_t$ be the law of $X_t$, then from a straightforward computation we easily deduce that the Kolmogorov equation of \eqref{eds0} is given by \eqref{Lmaster2}.

%

\section{The consistency-stability method for the Landau equation}\label{sec:abs-thm}

In this section we adapt the ``consistency-stability method'' developed in \cite{MMchaos} (see also \cite{MMWchaos}) in order to be able to apply it later to the Landau equation in Section~\ref{sec:app}. We shall introduce our abstract framework in Section~\ref{ssec:framework}, then we state and prove the main result of this section (Theorem~\ref{thm:abstract}) in Section~\ref{ssec:abstheo}.

\smallskip

Before going into details let us briefly explain the method, and we refer the reader to \cite{MMchaos,MMWchaos} for more information. We want to compare a solution to the $N$-particle system $F^N_t$ on $\PPP(E^N)$ to a solution to the limit mean-field equation $f_t$ on $\PPP(E)$, and the difficulty here arises from the fact that these two solution does not belong to the same functional space. 
The first step is then to define a common functional framework to be able to compare the $N$-particle dynamics to the limit dynamics. We hence work at the level of the full limit space $\PPP(\PPP(E))$ and, at the dual level, $C_b(\PPP(E))$, through some projections: we project the $N$-particle dynamics thanks to the empirical measures, and we project the limit dynamics by pullback. 
More precisely, at the level of the $N$-particle dynamics, we introduce the semigroup $S^N_t$ acting on $\PPP(E^N)$ associated to the flow, as well as the semigroup $T^N_t$ acting on $C_b(E^N)$ in duality with it. Moreover, we introduce, at the level of the limit dynamics, the (nonlinear) semigroup $S^\infty_t$ acting on $\PPP(E)$ associated to the limit mean-field equation, and also the associate linear semigroup $T^\infty_t$ acting on $C_b(\PPP(E))$ by pullback. 
We prove then convergence and stability estimates between the linear semigroups $T^N_t$ and $T^\infty_t$ as $N \to \infty$. In order to do this, we identify the regularity required to prove a consistency estimate between the generators $G^N$ and $G^\infty$ of the semigroups $T^N_t$ and $T^\infty_t$, and we prove the corresponding stability estimate on $T^\infty_t$, based on stability estimates on the (nonlinear) limit semigroup $S^\infty_t$. We observe here that we need to introduce an abstract differential calculus (see Definitions~\ref{def:holder} and \ref{def:diff-calculus}) for functions on the space of probability measures.

It is worth mentioning that the novelty of the method here with respect to \cite{MMchaos} appears in the regularity (of order two, instead of order one) that we need in the consistency estimate (see Assumption {\bf (A3)}) and the corresponding stability estimate (see Assumption {\bf (A4)}). This comes from the fact that the Landau equation possesses a different diffusive structure than the Boltzmann equation, which is treated in \cite{MMchaos}.

\subsection{Abstract framework}\label{ssec:framework}

Consider a Polish space $E$ and we shall denote by $\PPP(E)$ the space of probability measures on $E$. Consider also $E^N$ and the space of symmetric probability measures $ \Psym(E^N)$, more precisely, we say that $F^N\in \PPP(E^N)$ is symmetric if for all $\varphi\in C_b(E^N)$ we have that
$$
\int_{E^N} \varphi \, dF^N = \int_{E^N} \varphi_\sigma \, dF^N,
$$
for any permutation $\sigma$ of $\{1,\dots,N\}$, and where 
$$
\varphi_\sigma := \varphi(V_\sigma)=\varphi (v_{\sigma(1)},\dots,v_{\sigma(N)}),
$$
for $V=(v_1,\dots,v_N) \in E^N$.

\medskip
\noindent
{\it $N$-particle system framework.} 
We consider the trajectories $\VV^N_t \in E^N$, $t \ge 0$, of particles, and we assume that this flow commutes with permutations (which means that particles are indistinguishable).
To this flow in $E^N$ corresponds a semigroup $S^N_t$ that acts on $\Psym (E^N)$ for the probability density of particles in the phase space $E^N$, which is defined by
\begin{equation}\label{eq:SNt}
\forall\, F^N_0 \in \Psym(E^N), \;
\varphi \in C_b(E^N), \quad
\left\langle S^N_t (F^N_0) , \varphi \right\rangle = {\mathbb E} ( \varphi (\VV^N_t)),
\end{equation}
where the bracket denotes the duality between $\PPP(E^N)$ and $C_b(E^N)$, and $\mathbb E$ is the expectation associated to the space of probabilities in which the process $\VV^N_t$ is built, in other words $F^N_t  := S^N_t(F^N_0)$ is the law of $\VV^N_t$. 
Since the process $(\VV^N_t)_{t}$ commutes with permutations, the semigroup $S^N_t$ acts on $\Psym(E^N)$, which means that if the law $F^N_0$ of $\VV^N_0$ lies in $\Psym(E^N)$ then the 
law $F^N_t$ of $\VV^N_t$ also lies in $\Psym(E^N)$ for later times. We associate to the $c_0$-semigroup $S^N_t$ on $\Psym(E^N)$ a linear evolution equation with generator $A^N$
$$
\partial_t F^N = A^N F^N , \quad F^N \in \Psym(E^N),
$$
which is the forward Kolmogorov or Master equation on the law of $(\VV^N_t)$.

We also consider the semigroup $T^N_t$ acting on the function space $C_b(E^N)$ of observables of the evolution system $(\VV^N_t)_t$ on $E^N$, which is in duality with $S^N_t$ by
\begin{equation}\label{eq:TNt}
\forall\, F^N \in \Psym(E^N), \;
\varphi \in C_b(E^N), \quad
\left\langle F^N , T^N_t(\varphi) \right\rangle = \left\langle S^N_t (F^N) , \varphi \right\rangle .
\end{equation}
To the $c_0$-semigroup $T^N_t$ on $C_b(E^N)$ we can associate a linear evolution equation with generator $G^N$ by
$$
\partial_t \varphi = G^N \varphi, \quad \varphi \in C_b(E^N),
$$
which is the backward Kolmogorov equation.

\medskip
\noindent
{\it Limit mean-field equation framework.} 
We consider a (nonlinear) semigroup $S^\infty_t$ acting on $\PPP(E)$ associated to an evolution equation and some operator $Q$. More precisely, for any $f_0 \in \PPP(E)$ (and here we may assume additional bounds), we have $S^\infty_t (f_0) := f_t$ where $(f_t)_{t\ge 0} \in C([0,\infty); \PPP(E))$ is the solution of
\begin{equation}\label{eq:Q(ft)}
\partial_t f_t = Q(f_t),
\end{equation}
with initial data $f_{| t=0} = f_0$.

We consider the associated pullback semigroup $T^\infty_t$ acting on $C_b(\PPP(E))$ by
\begin{equation}\label{eq:Tinftyt}
\forall\, f \in \PPP(E), \;
\Phi \in C_b(\PPP(E)), \quad
T^\infty_t[\Phi] (f) := \Phi (S^\infty_t (f)).
\end{equation}
Remark that $T^\infty_t$ is linear as a function of $\Phi$, but in general $T^\infty[\Phi](f)$ is not linear as a function of $f$.

We associate to the semigroup $T^\infty_t$ the following linear evolution equation with some generator $G^\infty$,
$$
\partial_t \Phi = G^\infty \Phi, \quad \Phi \in C_b(\PPP(E)).
$$
We refer the reader to \cite[Remark 2.1]{MMchaos} for a heuristic explanation of the physical interpretation of the semigroup $T^\infty_t$: the semigroup of observables of the nonlinear equation~\eqref{eq:Q(ft)}.

\bigskip

%

As explained above, we define some applications relating this objects in order to be able to compare the two dynamics.
We define the application $\pi^N_E : E^N / \mathfrak S_N \to  \PPP(E)$, where $\mathfrak S_N$ denotes the group of permutations of $\{1,\dots,N \}$, by
\begin{equation}\label{}
\pi^N_E (V):= \mu^N_V = \frac1N\sum_{i=1}^{N} \delta_{v_i}, 
\end{equation}
and $\mu^N_V$ is called the empirical measure associated to $V$.
We introduce moreover the map $\pi^N_C : C_b( \PPP(E))  \to C_b(E^N)$ given by
\begin{equation}\label{}
\begin{aligned}
&\forall\, V \in E^N,\;\forall\, \Phi \in C_b( \PPP(E)), \qquad 
 \pi^N_C[\Phi](V):= (\Phi \circ \pi^N_E) (V) = \Phi(\mu^N_V).
\end{aligned}
\end{equation}
Then we define the application $\pi^N_P :  \Psym(E^N) \to  \PPP( \PPP(E))$ by
\begin{equation}\label{}
\begin{aligned}
&\forall\, F^N \in  \Psym(E^N),\;\forall\, \Phi \in C_b( \PPP(E)), \\
&\Ps{\pi^N_P (F^N)}{\Phi} :=\Ps{F^N}{\pi^N_C (\Phi)}= \int_{E^N} \Phi(\mu^N_V) F^N(dV),
\end{aligned}
\end{equation}
where the first bracket is the duality $ \PPP( \PPP(E)) \leftrightarrow C_b( \PPP(E))$ and the second one is the duality $ \Psym(E^N) \leftrightarrow C_b(E^N)$.
Finally, the map $R^N : C_b(E^N)  \to C_b( \PPP(E))$ is defined by
\begin{equation}\label{eq:R^N_phi}
\begin{aligned}
&\forall\, \varphi \in C_b(E^N),\;\forall\, f \in  \PPP(E), \\
&R^N[\varphi] (f) := \Ps{\varphi}{f^{\otimes N}} = \int_{E^N} \varphi(V) f(dv_1)\cdots f(dv_N),
\end{aligned}
\end{equation}
and in the sequel we denote $R^\ell_\varphi := R^\ell[\varphi]$ for $\varphi \in C_b(E^\ell)$.
The functions $R^\ell_\varphi$ are the ``polynomials'' on the space $ \PPP(E)$, we will see later (Example~\ref{ex:R^l_phi}) that they are continuous and differentiable in the sense of Definitions \ref{def:holder} and \ref{def:diff-calculus}, where we develop a differential calculus on $ \PPP(E)$.

\medskip

For a given weight function $m : E\to\R_+$ we define the $N$-particle weight function $M^N_m$ by
\begin{equation}\label{eq=weight-function}
\forall\, V = (v_1,...,v_N) \in E^N,\qquad M^N_m (V):= \frac{1}{N}\sum_{i=1}^{N} m(v_i) = \ps{\mu^N_V}{m} = M_m(\mu^N_V).
\end{equation}

\begin{definition}\label{def:P_G}
For a given weight function $m_\GG : E\to\R_+$ we define the subspaces of probability measures
\begin{equation*}\label{}
  \PPP_{\GG} := \left\{ f \in  \PPP(E) ;\; \ps{f}{m_{\GG}} <\infty \right\}
\end{equation*}
and the corresponding bounded sets, for $a\in(0,\infty)$,
\begin{equation*}\label{}
  \BB\PPP_{\GG,a} := \left\{ f \in  \PPP_{\GG} ;\; \ps{f}{m_{\GG}} \leq a \right\}.
\end{equation*}
For a given constraint function ${\bf m}_\GG : E\to \R^D$ such that $\langle f, {\bf m}_\GG \rangle$ is well defined for any $f\in \PPP_{\GG}$ and a given space of constraints $\RRR_\GG \subset \R^D$, we define, for any ${\bf r} \in \RRR_\GG$, the constrained subsets
\begin{equation*}\label{}
  \PPP_{\GG,{\bf r}} := \{ f \in   \PPP_{\GG} ;\; \ps{f}{{\bf m}_\GG}={\bf r} \},
\end{equation*}
the corresponding bounded constrained subsets
\begin{equation*}\label{}
  \BB\PPP_{\GG,a,{\bf r}} := \{ f \in   \BB\PPP_{\GG,a} \;;\; \ps{f}{{\bf m}_\GG}={\bf r} \},
\end{equation*}
and the corresponding space of increments
\begin{equation*}\label{}
\II \PPP_{\GG} := \{ g-f  ;\; \exists\, {\bf r}\in \RRR_\GG \text{ s.t. } g,f \in   \PPP_{\GG,{\bf r}} \}.
\end{equation*}

\end{definition}

We shall consider a distance $\dist_\GG$ which is either defined on the whole space $\PPP_{\GG}$ or such that there is a Banach space $\GG \supset \II \PPP_{\GG}$ endowed with a norm $\| \cdot \|_\GG$ such that $\dist_\GG$ is defined on $\PPP_{\GG,{\bf r}}$ for any ${\bf r} \in \RRR_\GG$ by 
$$
\dist_\GG (g,f) = \| g-f \|_\GG, \quad \forall f,g\in \PPP_{\GG,{\bf r}} . 
$$

\begin{definition}\label{def:space-duality}
We say that two spaces $\FF$ and $\PPP_\GG$, where $\FF \subset C_b(E)$ is a normed vector space endowed with the norm $\|\cdot\|_\FF$ and $\PPP_\GG \subset \PPP(E)$ is a subspace of probability measures endowed with a metric $\dist_\GG$, satisfy a duality inequality if
\begin{equation}\label{}
\forall\, f,g \in \PPP_\GG\,,\; \forall\, \varphi \in \FF \qquad 
|\langle f-g, \varphi \rangle | \le   \dist_{\GG}(g,f)\, \| \varphi \|_{\FF},
\end{equation}
where here $\langle \cdot, \cdot \rangle$ corresponds to the usual duality between $\PPP(E)$ and $C_b(E)$. In the case in which $\dist_\GG$ is associated to a normed vector space $\GG$, this amounts to the usual duality inequality $| \langle h, \varphi \rangle | \le \| h \|_\GG \, \| \varphi \|_\FF $.

\end{definition}


\begin{definition}\label{def:holder}
Consider two metric spaces $\widetilde\GG_1$ and $\widetilde\GG_2$, some weight function $\Lambda : \widetilde\GG_1\to [1,\infty)$ and $\eta\in (0,1]$. We denote by $C^{0,\eta}_{\Lambda} (\widetilde\GG_1;\widetilde\GG_2)$ the (weighted) space of functions with $\eta$-H\"older regularity, that is functions $\SS : \widetilde\GG_1 \to \widetilde\GG_2$ such that there exists a constant $C>0$ so that
\begin{equation}\label{}
\forall\, f,g \in \widetilde\GG_1\,, \quad 
\dist_{\widetilde\GG_2} \left( \SS(f),\SS(g)   \right)\leq C\,\Lambda(g,f)\, \dist_{\widetilde\GG_1} (f,g)^{\eta}.
\end{equation}
where $\Lambda(g,f) := \max\{ \Lambda(g),\Lambda(f)  \}$.
\end{definition}

We define then a higher order differential calculus.

\begin{definition}\label{def:diff-calculus}
Consider two normed spaces $\GG_1$ and $\GG_2$, two metric spaces 
$\widetilde\GG_1$ and $\widetilde\GG_2$ such that $\widetilde\GG_i - \widetilde\GG_i \subset\GG_i$, some weight function $\Lambda : \widetilde\GG_1\to [1,\infty)$ and $\eta\in (0,1]$. We denote by $C^{2,\eta}_{\Lambda} (\widetilde\GG_1;\widetilde\GG_2)$ the (weighted) space of functions two times continuously differentiable from $\widetilde\GG_1$ to $\widetilde\GG_2$, and such that the second derivative satisfies some weighted $\eta$-H\"older regularity (in the sense of Definition \ref{def:holder}).

More precisely, these are functions $\SS : \widetilde\GG_1 \to \widetilde\GG_2$ continuous, such that there exists maps (for $j=1,2$)
$D^j\SS : \widetilde\GG_1 \to \LL^j(\GG_1,\GG_2)$ , where $\LL^j(\GG_1,\GG_2)$ is the space of $j$-multilinear applications from $\GG_1$ to $\GG_2$, and there exist some constants $C_{j}>0$, so that we have for all $f,g \in \widetilde\GG_1$,
\begin{equation}\label{eq:diff-calculus}
\begin{aligned}
\Norm{ \SS(g) - \SS(f)}_{\GG_2} 
&\leq C_{0}\,\Lambda(g,f)\, {\Norm{g-f}}_{\GG_1}^{\eta_0}, \\
\Norm{\Ps{D\SS[f]}{g-f}}_{\GG_2} 
&\leq C_{1}\,\Lambda(g,f)\, {\Norm{g-f}}_{\GG_1}^{\eta_0} ,\\
\Norm{ \SS(g) - \SS(f) - \Ps{D\SS[f]}{g-f}}_{\GG_2} 
&\leq C_{2}\,\Lambda(g,f)\, {\Norm{g-f}}_{\GG_1}^{1+\eta_1}, \\
\Norm{  \Ps{D^2\SS[f]}{(g-f)^{\otimes 2}}}_{\GG_2} 
&\leq C_{3}\,\Lambda(g,f)\, {\Norm{g-f}}_{\GG_1}^{1+\eta_1}, \\
\Norm{ \SS(g) - \SS(f) -\sum_{i=1}^{2} \Ps{D^i\SS[f]}{(g-f)^{\otimes i}}}_{\GG_2} 
&\leq C_{4}\,\Lambda(g,f)\, {\Norm{g-f}}_{\GG_1}^{2+\eta},
\end{aligned}
\end{equation}
where $\eta_0,\eta_1 \in [\eta,1]$.

We define then the semi-norms on $C^{2,\eta}_{\Lambda} (\widetilde\GG_1;\widetilde\GG_2)$
$$
\left[  \SS \right]_{C^{1,0}_{\Lambda}} := \sup_{f\in \widetilde \GG_1, h\in \GG_1} 
\frac{\Norm{\Ps{D\SS[f]}{h}}_{\GG_2}}{\Lambda(f)\, {\Norm{h}}_{\GG_1}^{\eta_0}},\qquad
\left[  \SS \right]_{C^{2,0}_{\Lambda}} := \sup_{f\in \widetilde \GG_1, h\in \GG_1} 
\frac{\Norm{\Ps{D^2\SS[f]}{(h,h)}}_{\GG_2}}{\Lambda(f)\, {\Norm{h}}_{\GG_1}^{1+\eta_1}}
$$
and
$$
\begin{aligned} 
\left[  \SS \right]_{C^{0,\eta_0}_{\Lambda}} &:= \sup_{f,g \in \widetilde \GG_1}
\frac{\Norm{ \SS(g) - \SS(f)}_{\GG_2}}{\Lambda(g,f)\, {\Norm{g-f}}_{\GG_1}^{\eta_0}},\\
\left[  \SS \right]_{C^{1,\eta_1}_{\Lambda}} &:= \sup_{f,g \in \widetilde \GG_1}
\frac{\Norm{ \SS(g) - \SS(f) - \Ps{D\SS[f]}{g-f}}_{\GG_2}}{\Lambda(g,f)\, {\Norm{g-f}}_{\GG_1}^{1+\eta_1}},\\
\left[  \SS \right]_{C^{2,\eta}_{\Lambda}} &:= \sup_{f,g \in \widetilde \GG_1}
\frac{\Norm{ \SS(g) - \SS(f) -\sum_{i=1}^{2} \Ps{D^i\SS[f]}{(g-f)^{\otimes i}}}_{\GG_2}}{\Lambda(g,f)\, {\Norm{g-f}}_{\GG_1}^{2+\eta}}.
\end{aligned}
$$
Finally we combine these semi-norms into 
$$
\norm{\SS}_{C^{2,\eta}_{\Lambda}} := \left[  \SS \right]_{C^{0,\eta_0}_{\Lambda}} 
+ \left[  \SS \right]_{C^{1,\eta_1}_{\Lambda}} + \left[  \SS \right]_{C^{2,\eta}_{\Lambda}}
+ \left[  \SS \right]_{C^{1,0}_{\Lambda}}
+ \left[  \SS \right]_{C^{2,0}_{\Lambda}}.
$$

\end{definition}


This differential calculus holds for composition, more precisely for $\UU \in C^{k,\eta}_{\Lambda_\UU}(\widetilde\GG_1;\widetilde\GG_2)$ and $\VV \in C^{k,\eta}_{\Lambda_\VV}(\widetilde\GG_2;\widetilde\GG_3)$ we have $\SS = \VV\circ\UU \in C^{k,\eta_\SS}_{\Lambda_\SS}(\widetilde\GG_1;\widetilde\GG_3)$ for some appropriate weight function $\Lambda_\SS$ and exponent $\eta_\SS$. We now state the following lemma

\begin{lemma}\label{lem:chain-rule}
Let $\GG_i$ be normed spaces and $\widetilde\GG_i$ be metric spaces for $i=1,2,3$, such that $\widetilde\GG_i - \widetilde\GG_i \subset\GG_i$.
Consider $\UU \in C^{2,\eta}_{\Lambda} \cap C^{1,(1+2\eta)/3}_{\Lambda} \cap C^{0,(2+\eta)/3}_{\Lambda} (\widetilde\GG_1;\widetilde\GG_2)$, with $\eta\in(0,1]$, and $\VV \in C^{2,1}(\widetilde\GG_2;\widetilde\GG_3)$. Then the composition function $\SS = \VV \circ \UU$ satisfies 
$$
\SS \in C^{2,\eta}_{\Lambda^{3}} \cap C^{1,(1+2\eta)/3}_{\Lambda^{3}} \cap C^{0,(2+\eta)/3}_{\Lambda^{3}}(\widetilde\GG_1;\widetilde\GG_3)
$$ 
and we have 
\begin{equation*}\label{}
\begin{aligned}
&D\SS[f] = D\VV[\UU(f)] \circ D\UU[f], \\
&D^2\SS[f] = D^2\VV[\UU(f)] \circ \left( D\UU[f]\otimes D\UU[f]\right) + D\VV[\UU(f)]\circ D^2\UU[f].
\end{aligned}
\end{equation*}
More precisely, the following estimates hold
\begin{equation*}\label{}
\begin{aligned}
\left[  \SS \right]_{C^{0,(2+\eta)/3}_{\Lambda}} &\leq 
\left[  \VV \right]_{C^{0,1}}\left[  \UU \right]_{C^{0,(2+\eta)/3}_{\Lambda}} , \\
\left[  \SS \right]_{C^{1,0}_{\Lambda}} &\leq 
\left[  \VV \right]_{C^{1,0}}\left[  \UU \right]_{C^{1,0}_{\Lambda}} , \\ 
\left[  \SS \right]_{C^{1,(1+2\eta)/3}_{\Lambda^{2}}} &\leq 
\left[  \VV \right]_{C^{1,0}}\, \left[ \UU \right]_{C^{1,(1+2\eta)/3}_{\Lambda}} +
\left[  \VV \right]_{C^{1,1}} \, \left[ \UU \right]_{C^{0,(2+\eta)/3}_{\Lambda}}^2 ,\\
\left[  \SS \right]_{C^{2,0}_{\Lambda^{2}}} &\leq 
\left[  \VV \right]_{C^{1,0}}\, \left[ \UU \right]_{C^{2,0}_{\Lambda}} +
\left[  \VV \right]_{C^{2,0}} \, \left[ \UU \right]_{C^{1,0}_{\Lambda}}^2 ,\\
\left[ \SS \right]_{C^{2,\eta}_{\Lambda^{3}}} 
&
\leq   \left[  \VV \right]_{C^{1,0}}\, 
 \left[ \UU \right]_{C^{2,\eta}_{\Lambda}} 
+ \left[  \VV \right]_{C^{2,0}}\,
 \left[  \UU \right]_{C^{1,(1+2\eta)/3}_{\Lambda}}^2 \\
&\quad
+2\left[  \VV \right]_{C^{2,0}}\,
 \left[  \UU \right]_{C^{1,0}_{\Lambda}}  \, \left[  \UU \right]_{C^{1,(1+2\eta)/3}_{\Lambda}} 
+\left[  \VV \right]_{C^{2,1}}  \left[  \UU \right]_{C^{0,(2+\eta)/3}_{\Lambda}}^3.
\end{aligned}
\end{equation*}

\end{lemma}

\begin{proof}[Proof of Lemma \ref{lem:chain-rule}]
Let $f,g \in \widetilde\GG_1$ and $\bar f, \bar g \in \widetilde\GG_2$.

By Definition \ref{def:diff-calculus} with $\UU \in C^{2,\eta}_{\Lambda} \cap C^{1,(1+2\eta)/3}_{\Lambda} \cap C^{0,(2+\eta)/3}_{\Lambda}(\widetilde\GG_1;\widetilde\GG_2)$ and $\VV \in C^{2,1}(\widetilde\GG_2;\widetilde\GG_3)$, we have
\begin{equation}\label{eq:U}
\begin{aligned}
\UU(g)-\UU(f) &= \Ps{D\UU[f]}{g-f} + R^1_\UU(g,f) ,\\
\UU(g)-\UU(f) &= \Ps{D\UU[f]}{g-f} + \Ps{D^2\UU[f]}{(g-f)^{\otimes 2}}  +R^2_\UU(g,f),
\end{aligned}
\end{equation}
with 
\begin{eqnarray}
{\norm{ \UU(g) - \UU(f)}}_{\GG_2} &\leq& \left[  \UU \right]_{C^{0,\eta_0}_{\Lambda}} \, \Lambda(g,f) \, {\norm{g-f}}_{\GG_1}^{\eta_0}, \label{eq:a3}\\
{\norm{ \Ps{D\UU[f]}{g-f}}}_{\GG_2} &\leq& \left[  \UU \right]_{C^{1,0}_{\Lambda}} \, \Lambda(g,f) \, {\norm{g-f}}_{\GG_1}^{\eta_0}, \label{eq:a1}\\
{\norm{R^1_\UU(g,f)}}_{\GG_2} &\leq& \left[  \UU \right]_{C^{1,\eta_1}_{\Lambda}} \, \Lambda(g,f)
 \, {\norm{g-f}}_{\GG_1}^{1+\eta_1} , \label{eq:a4}\\
 {\norm{ \Ps{D^2\UU[f]}{(g-f)^{\otimes 2}}}}_{\GG_2} &\leq& \left[  \UU \right]_{C^{2,0}_{\Lambda}} \, \Lambda(g,f) \, {\norm{g-f}}_{\GG_1}^{1+\eta_1}, \label{eq:a2}\\
{\norm{R^2_\UU(g,f)}}_{\GG_2} &\leq& \left[  \UU \right]_{C^{2,\eta}_{\Lambda}} \, \Lambda(g,f)
 \, {\norm{g-f}}_{\GG_1}^{2+\eta}  \label{eq:a5},
\end{eqnarray}
where, for simplicity, we denote $\eta_0 = (2+\eta)/3$ and $\eta_1 = (1+2\eta)/3$.

Similarly we have for $\VV$,
\begin{equation}\label{eq:V}
\begin{aligned}
\VV(\bar g)-\VV(\bar f) &= \Ps{D\VV[\bar f]}{\bar g-\bar f} + R^1_\VV(\bar g,\bar f), \\
\VV(\bar g)-\VV(\bar f) &= \Ps{D\VV[\bar f]}{\bar g-\bar f} + \Ps{D^2\VV[\bar f]}{(\bar g-\bar f)^{\otimes 2}}  +R^2_\VV(\bar g,\bar f),
\end{aligned}
\end{equation}
with 
\begin{eqnarray}
{\norm{ \VV(\bar g) - \VV(\bar f)}}_{\GG_3} &\leq& \left[  \VV \right]_{C^{0,1}} \,  {\norm{\bar g-\bar f}}_{\GG_2}, \label{eq:b3}\\
{\norm{ \Ps{D\VV[\bar f]}{\bar g-\bar f}}}_{\GG_3} &\leq& \left[  \VV \right]_{C^{1,0}} \, {\norm{\bar g-\bar f}}_{\GG_2} , \label{eq:b1}\\
{\norm{R^1_\VV(\bar g,\bar f)}}_{\GG_3} &\leq& \left[  \VV \right]_{C^{1,1}} 
 \, {\norm{\bar g-\bar f}}_{\GG_2}^{2} , \label{eq:b4}\\
 {\norm{ \Ps{D^2\VV[\bar f]}{(\bar g-\bar f)^{\otimes 2}}}}_{\GG_3} &\leq& \left[  \VV \right]_{C^{2,0}}  \, {\norm{\bar g-\bar f}}_{\GG_2}^{2} , \label{eq:b2}\\
{\norm{R^2_\VV(\bar g,\bar f)}}_{\GG_3} &\leq& \left[  \VV \right]_{C^{2,1}} 
 \, {\norm{\bar g-\bar f}}_{\GG_2}^{3}  \label{eq:b5}.
\end{eqnarray}

Using these estimates we can compute the same type of estimates for $\SS=\VV\circ\UU$. We obtain first, thanks to \eqref{eq:b3} and \eqref{eq:a3}, that
\begin{equation*}\label{}
\begin{aligned}
{\norm{\SS(g) - \SS(f)}}_{\GG_3} &= {\norm{\VV\left(\UU(g)\right) - \VV\left(\UU(f)\right)}}_{\GG_3} \\
&\leq \left[  \VV \right]_{C^{0,1}} {\norm{\UU(g)-\UU(f)}}_{\GG_2} \\
&\leq \left[  \VV \right]_{C^{0,1}}
\left[  \UU \right]_{C^{0,\eta_0}_{\Lambda}} \Lambda(g,f){\norm{g-f}}_{\GG_1}^{\eta_0}
\end{aligned}
\end{equation*}
which implies $\left[  \SS \right]_{C^{0,(2+\eta)/3}_{\Lambda}} \leq 
\left[  \VV \right]_{C^{0,1}}\left[  \UU \right]_{C^{0,(2+\eta)/3}_{\Lambda}}$.

We also have, using \eqref{eq:V} and \eqref{eq:U},
\begin{equation*}\label{}
\begin{aligned}
\SS(g) - \SS(f) &= \VV\left(\UU(g)\right) - \VV\left(\UU(f)\right) \\
&=  \Ps{D\VV[\UU(f)]}{\UU(g)-\UU(f)} + R^1_\VV(\UU(g),\UU(f))\\
&= \Ps{D\VV[\UU(f)]}{\left\{ \Ps{D\UU[f]}{g-f} + R^1_\UU(g,f) \right\}} \\
&\quad+ R^1_\VV(\UU(g),\UU(f)),
\end{aligned}
\end{equation*}
from which we deduce $\Ps{D\SS[f]}{g-f} = \Ps{D\VV[\UU(f)]}{\left( \Ps{D\UU[f]}{g-f} \right)}$ 
and, by \eqref{eq:b1} and \eqref{eq:a1},
\begin{equation*}
\begin{aligned}
\|\Ps{D\SS[f]}{g-f}\|_{\GG_3} &\le \left[  \VV \right]_{C^{1,0}}  \|\Ps{D\UU[f]}{g-f} \|_{\GG_2} \\
&\le \left[  \VV \right]_{C^{1,0}} \left[  \UU \right]_{C^{1,0}_\Lambda} \Lambda(f) \|g-f\|_{\GG_1}^{\eta_0},
\end{aligned}
\end{equation*}
which yields $\left[  \SS \right]_{C^{1,0}_\Lambda} \le \left[  \VV \right]_{C^{1,0}} \left[  \UU \right]_{C^{1,0}_\Lambda}$.

Therefore, we obtain using \eqref{eq:b1}, \eqref{eq:b4}, \eqref{eq:a4} and \eqref{eq:a3}, 
\begin{equation*}\label{}
\begin{aligned}
&\left\| \SS(g) - \SS(f) - \Ps{D\SS[f]}{g-f} \right\|_{\GG_3} \\ 
&\qquad 
\leq {\norm{\Ps{D\VV[\UU(f)]}{ R^1_\UU(g,f) }}}_{\GG_3} +{\norm{R^1_\VV(\UU(g),\UU(f))}}_{\GG_3}\\
&\qquad
\leq  \left[  \VV \right]_{C^{1,0}}  \, {\norm{R^1_\UU(g,f)}}_{\GG_2} +  
\left[  \VV \right]_{C^{1,1}} \, {\norm{\UU(g)-\UU(f)}}_{\GG_2}^{2} \\
&\qquad \leq
\left[  \VV \right]_{C^{1,0}}\, \left[ \UU \right]_{C^{1,\eta_1}_{\Lambda}} \, \Lambda(g,f) \,
{\norm{g-f}}_{\GG_1}^{1+\eta_1} 
+ \left[  \VV \right]_{C^{1,1}} \, \left[ \UU \right]_{C^{0,\eta_0}_{\Lambda}}^2 \, \Lambda(g,f)^2 \, {\norm{g-f}}_{\GG_1}^{2\eta_0}.
\end{aligned}
\end{equation*}
Since $1+\eta_1 = 2\eta_0 = 1 + (1+2\eta)/3$ and $\Lambda \ge 1$, the last inequality implies 
$$
\left[  \SS \right]_{C^{1,(1+2\eta)/3}_{\Lambda^{2}}} \leq 
\left[  \VV \right]_{C^{1,0}}\, \left[ \UU \right]_{C^{1,(1+2\eta)/3}_{\Lambda}} +
\left[  \VV \right]_{C^{1,1}} \, \left[ \UU \right]_{C^{0,(2+\eta)/3}_{\Lambda}}^2.
$$ 

Finally, from \eqref{eq:V} and \eqref{eq:U}, we have
\begin{equation*}\label{}
\begin{aligned}
&\SS(g) - \SS(f) = \VV\left(\UU(g)\right) - \VV\left(\UU(f)\right) \\
&\quad=  \Ps{D\VV[\UU(f)]}{\UU(g)-\UU(f)} + \Ps{D^2\VV[\UU(f)]}{\big(\UU(g)-\UU(f)\big)^{\otimes 2}} +R^2_\VV(\UU(g),\UU(f))\\
&\quad
=  \Ps{D\VV[\UU(f)]}{\Big( \Ps{D\UU[f]}{g-f} + \Ps{D^2\UU[f]}{(g-f)^{\otimes 2}}  +R^2_\UU(g,f)\Big)} \\
&\quad\quad
+ \Ps{D^2\VV[\UU(f)]}{\Big(\Ps{D\UU[f]}{g-f} + R^1_\UU(g,f)\Big)^{\otimes 2}}\\
&\quad \quad
+R^2_\VV(\UU(g),\UU(f)),
\end{aligned}
\end{equation*}
which yields 
$$
\begin{aligned}
\Ps{D^2\SS[f]}{(g-f)^{\otimes 2}} &= \Ps{D\VV[\UU(f)]}{\big( \Ps{D^2\UU[f]}{(g-f)^{\otimes 2}} \big)} \\
&+ \Ps{D^2\VV[\UU(f)]}{\big(\Ps{D\UU[f]}{g-f} \big)^{\otimes 2}}.
\end{aligned}
$$
Hence we obtain, with \eqref{eq:b1}, \eqref{eq:b2}, \eqref{eq:a2} and \eqref{eq:a1},
$$
\begin{aligned}
&\left\| \Ps{D^2\SS[f]}{(g-f)^{\otimes 2}} \right\|_{\GG_3} \\
&\qquad\le \left[  \VV \right]_{C^{1,0}} \left\| \Ps{D^2\UU[f]}{(g-f)^{\otimes 2}} \right\|_{\GG_2} 
+ \left[  \VV \right]_{C^{2,0}}   \left\|\Ps{D\UU[f]}{g-f} \right\|_{\GG_2}^2 \\
&\qquad\le \left[  \VV \right]_{C^{1,0}} \left[  \UU \right]_{C^{2,0}_\Lambda} \,\Lambda(f) \,\|g-f\|_{\GG_1}^{1+\eta_1}
+ \left[  \VV \right]_{C^{2,0}} \left( \left[  \UU \right]_{C^{1 ,0}_\Lambda} \, \Lambda(f)\, 
\|g-f\|_{\GG_1}^{\eta_0} \right)^2 \\
&\qquad\le\left( \left[  \VV \right]_{C^{1,0}} \left[  \UU \right]_{C^{2,0}_\Lambda} 
+ \left[  \VV \right]_{C^{2,0}} \left[  \UU \right]_{C^{1 ,0}_\Lambda}^2 \right) \Lambda(f)^2 \, \|g-f\|_{\GG_1}^{1+(1+2\eta)/3},
\end{aligned}
$$
which gives $\left[  \SS \right]_{C^{2,0}_{\Lambda^2}} \le  \left[  \VV \right]_{C^{1,0}} \left[  \UU \right]_{C^{2,0}_\Lambda} 
+ \left[  \VV \right]_{C^{2,0}} \left[  \UU \right]_{C^{1 ,0}_\Lambda}^2$.

Now, for the last estimate we obtain
\begin{equation*}\label{}
\begin{aligned}
&{\norm{\SS(g) - \SS(f) - \Ps{D\SS[f]}{g-f} - \Ps{D^2\SS[f]}{(g-f)^{\otimes 2}}}}_{\GG_3} \\ 
&\qquad 
\leq {\Norm{\Ps{D\VV[\UU(f)]}{ R^2_\UU(g,f) }}}_{\GG_3} 
+{\Norm{\Ps{D^2\VV[\UU(f)]}{ \left( R^1_\UU(g,f)\right)^{\otimes2} }}}_{\GG_3} \\
&\qquad \quad
+2\,{\Norm{\Ps{D^2\VV[\UU(f)]}{ \left( \Ps{D\UU[f]}{g-f}\otimes R^1_\UU(g,f)\right) }}}_{\GG_3} \\
&\qquad \quad
 +{\Norm{R^2_\VV(\UU(g),\UU(f))}}_{\GG_3}\\
\end{aligned}
\end{equation*}
and using the equations \eqref{eq:b3} to \eqref{eq:b5} and \eqref{eq:a3} to \eqref{eq:a5}, it gives
\begin{equation*}\label{}
\begin{aligned}
&{\norm{\SS(g) - \SS(f) - \Ps{D\SS[f]}{g-f} - \Ps{D^2\SS[f]}{(g-f)^{\otimes 2}}}}_{\GG_3} \\
& \qquad \leq 
\left[  \VV \right]_{C^{1,0}}  \, {\Norm{R^2_\UU(g,f)}}_{\GG_2} 
+\left[  \VV \right]_{C^{2,0}}  \, {\Norm{  R^1_\UU(g,f)}}_{\GG_2}^{2} \\
& \qquad\quad
+2\left[  \VV \right]_{C^{2,0}} \, {\Norm{ \Ps{D\UU[f]}{g-f} }}_{\GG_2}
{\Norm{ R^1_\UU(g,f) }}_{\GG_2} 
 +\left[  \VV \right]_{C^{2,1}} {\Norm{\UU(g)-\UU(f)}}_{\GG_2}^{3}\\
& \qquad\leq 
 \left[  \VV \right]_{C^{1,0}} \, 
 \left[ \UU \right]_{C^{2,\eta}_{\Lambda}} \, \Lambda(g,f) \, {\Norm{g-f}}_{\GG_1}^{2+\eta} \\
& \qquad\quad
 +\left[  \VV \right]_{C^{2,0}}  \,
 \left[  \UU \right]_{C^{1,\eta_1}_{\Lambda}}^2   
 \, \Lambda(g,f)^2 \, {\Norm{g-f}}_{\GG_1}^{2+2\eta_1} \\
 & \qquad\quad
+2\left[  \VV \right]_{C^{2,0}}  \,
 \left[  \UU \right]_{C^{1,0}_{\Lambda}}\, \left[  \UU \right]_{C^{1,\eta_1}_{\Lambda}}
 \, \Lambda(g,f)^2 \, {\Norm{g-f}}_{\GG_1}^{1+\eta_1+\eta_0}    \\
& \qquad\quad
 + \left[  \VV \right]_{C^{2,1}}  \left[  \UU \right]_{C^{0,\eta_0}_{\Lambda}}^3 \, \Lambda(g,f)^3 \, {\Norm{g-f}}_{\GG_1} ^{3\eta_0}.
\end{aligned}
\end{equation*}
Since $1+\eta_1+\eta_0 = 3\eta_0 = 2 + \eta < 2+2\eta_1$, we deduce
$$
\begin{aligned}
\left[ \SS \right]_{C^{2,\eta}_{\Lambda^{3}}} 
&
\leq   \left[  \VV \right]_{C^{1,0}} \, 
 \left[ \UU \right]_{C^{2,\eta}_{\Lambda}} 
+ \left[  \VV \right]_{C^{2,0}} \,
 \left[  \UU \right]_{C^{1,(1+2\eta)/3}_{\Lambda}}^2 \\
&+2\left[  \VV \right]_{C^{2,0}} \,
 \left[  \UU \right]_{C^{1,0}_{\Lambda}}\,\left[  \UU \right]_{C^{1,(1+2\eta)/3}_{\Lambda}} 
+\left[  \VV \right]_{C^{2,1}}  \left[  \UU \right]_{C^{0,(2+\eta)/3}_{\Lambda}}^3.
\end{aligned}
$$

\end{proof}

\begin{ex}\label{ex:R^l_phi}
Consider the pair $\FF$ and $\PPP_\GG$ in duality (Definition \ref{def:space-duality}) where $\FF\subset C_b(E)$, and consider $\varphi = \varphi_1\otimes\cdots\otimes \varphi_\ell \in \FF^{\otimes\ell}$. Then the application $R^\ell_\varphi$ defined in \eqref{eq:R^N_phi} 
is in $C^{2,1}(\PPP_\GG;\R)$. Consider $f,g \in \PPP_\GG$, then we have thanks to the multi-linearity of $R^\ell_\varphi$ (see \cite{MMchaos,MMWchaos}) that
$$
\bal
\left| R^\ell_\varphi(g) - R^\ell_\varphi(f) \right| 
        &\le \ell \,\| \varphi \|_{\FF\otimes (L^\infty)^{\ell-1}}\, \| g-f\|_{\GG}, \\
        \left| D R^\ell_\varphi[f](g-f) \right| 
        &\le \ell \,\| \varphi \|_{\FF\otimes (L^\infty)^{\ell-1}}\, \| g-f\|_{\GG}, \\ 
 \left| R^\ell_\varphi(g) - R^\ell_\varphi(f) - D R^\ell_\varphi[f](g-f) \right|    
 	&\le \frac{\ell(\ell-1)}{2} \, \| \varphi \|_{\FF^{\otimes 2}\otimes (L^\infty)^{\ell-2}}\, \| g-f\|_{\GG}^2 \\
	 \left|D^2 R^\ell_\varphi[f](g-f)^{\otimes 2} \right|    
 	&\le \frac{\ell(\ell-1)}{2} \, \| \varphi \|_{\FF^{\otimes 2}\otimes (L^\infty)^{\ell-2}}\, \| g-f\|_{\GG}^2 \\
\Big| R^\ell_\varphi(g) - R^\ell_\varphi(f) - D R^\ell_\varphi[f](g-f)&- D^2 R^\ell_\varphi[f](g-f)^{\otimes 2}\Big|  \\  
 	&\le \frac{\ell(\ell-1)(\ell-2)}{6} \, \| \varphi \|_{\FF^{\otimes 3}\otimes (L^\infty)^{\ell-3}} \, \| g-f\|_{\GG}^3,
\eal
$$
where we define
$$
\| \varphi \|_{\FF^{\otimes j}\otimes (L^\infty)^{\ell-j}} := \max_{i_1,\dots,i_j\text{ distincts in } [1,\ell] }
\left(  \| \varphi_{i_1}\|_{\FF} \cdots \| \varphi_{i_1}\|_{\FF} 
\prod_{k\neq i_1,\dots,i_j}  \| \varphi_{k}\|_{L^\infty} \right).
$$
\end{ex}


Since we shall need to endow the subspaces of probability measures in Definition~\ref{def:P_G} with metrics, let us give some useful examples.
We denote by $\PPP_{p}(\R^d)$ the space of probabilities with finite moments up to order 
$p$, more precisely $\PPP_{p}(\R^d) := \{ f\in  \PPP(\R^d) \;;\; \ps{|v|^p}{f}<\infty    \}$.

\begin{definition}[Monge-Kantorovich-Wasserstein distance]\label{def:wasserstein}
For $f,g\in  \PPP_{p}(\R^d)$ we define the distance
\begin{equation*}\label{}
W_p^p (f,g) := \inf_{\pi\in \Pi(f,g)} \int_{\R^d\times\R^d} \dist(x,y)^p \,\pi(dx,dy)
\end{equation*}
where $\Pi(f,g)$ is the set of probability measures on $\R^d\times\R^d$ with marginals $f$ and $g$ respectively.

In a analogous way, we also define, for $\mu,\nu \in \PPP(\PPP(\R^d))$ and a distance $D$ over $\PPP(\R^d)$, the distance
\begin{equation*}\label{}
\WW_{1,D} (\mu,\nu) := \inf_{\pi\in \Pi(\mu,\nu)} \int_{\PPP(\R^d)\times\PPP(\R^d)} D(f,g) \,\pi(df,dg)
\end{equation*}
where $\Pi(\mu,\nu)$ denotes the set of probability measures on $\PPP(\R^d)\times\PPP(\R^d)$ with marginals $\mu$ and $\nu$.
\end{definition}

\begin{definition}[Fourier based distances]\label{def:toscani}
For $f,g\in  \PPP_{s}(\R^d)$ we define the distance
\begin{equation}\label{eq:toscani}
\abs{f-g}_s := \sup_{\xi\in \R^d} \frac{\abs{\hat f(\xi) - \hat g(\xi)}}{\abs{\xi}^s}
\end{equation}
which is well defined if $f$ and $g$ have equal moments up to order $s-1$ if $s$ is a integer or $\lfloor s \rfloor$ if not.
We shall denote by $\HH^{-s}(\R^d)$ the space associated to this norm.
\end{definition}

\begin{definition}[General Fourier based distances]\label{def:general-toscani}
Let $k\in \N^*$ and set
$$
m_{\GG} := |v|^k, \qquad {\bf m}_{\GG} := \left( v^\alpha \right)_{\alpha \in \N^d}, \; |\alpha|\le k-1,
$$
and 
$$
v^\alpha = \left( v_1^{\alpha_1}, \dots, v_d^{\alpha_d} \right), \qquad 
\alpha=(\alpha_1,\dots,\alpha_d).
$$
We define
$$
\forall\, f\in \II  \PPP_{\GG} , \qquad {\norm{f}}_{\GG} = |f|_s := \sup_{\xi\in\R^d} \frac{|\hat f(\xi)|}{|\xi|^s}, \; s\in (0,k].
$$
We extend the above norm to $M^1_k(\R^d)$, where we denote $M^1(\R^d)$ the space of finite Radon measures and $M^1_k (\R^d)$ its subspace of measures with finite moments up to order $k$, in the following way. First we define for
$f\in M^1_{k-1}(\R^d) $ and $\alpha\in \N^d$, $|\alpha|\le k-1$ the moment
$$
M_\alpha[f] := \int_{\R^d} v^\alpha f(dv).
$$
For a fixed smooth function with compact support $\chi \in C^\infty_c(\R^d)$ such that $\chi =1$ over $\{ \xi\in \R^d , |\xi|\le 1\}$, we define the function $\MM_k [f]$ by its Fourier transform
$$
\hat \MM_k[f](\xi) := \chi(\xi) \left( \sum_{|\alpha|\le k-1} M_\alpha[f]\, \frac{\xi^\alpha}{\alpha!} \,(-i)^{|\alpha|}  \right).
$$
We define then the norm
\begin{equation}\label{eq:general-toscani}
\Nt f \Nt_k := |f-\MM_k[f]|_k + \sum_{|\alpha|\le k-1} |M_\alpha[f]|,
\end{equation}
where as above $|h|_k := \sup_\xi \frac{|\hat h(\xi)|}{|\xi|^k}$.

\end{definition}

\subsection{Abstract theorem}\label{ssec:abstheo}

We state the assumptions of our abstract theorem~\ref{thm:abstract}.

\begin{assump}[$N$-particle system]\label{A1}
The semigroup $T^N_t$ and its generator $G^N$ are well defined on $C_b(E^N)$ and are invariant under permutation so that $F^N_t$ is well defined. Moreover, we assume that the following conditions hold: 

\begin{enumerate}[(i)]

\item {\it Conservation constraint}: There exists a constraint function ${\bf m}_{\GG_1} : E\to \R^D$ and a subset ${\bf R}_{{\GG_1}} \subset \R^D$ such that defining the set
$$
\E_N := \{ V\in E^N; \ps{\mu^N_V}{{\bf m}_{\GG_1}} \in {\bf R}_{{\GG_1}} \}
$$
there holds
$$
\forall\, t\ge 0, \qquad \supp F^N_t \subset \E_N.
$$

\item {\it Propagation of integral moment bound}: There exists a weight function $m_{\GG_1}$ and a constant $C_{m_{\GG_1}}>0$ that does not depend on $N$, such that
$$
\forall\,N\ge 1 \qquad \sup_{t\ge 0}  \left\langle F^N_t , M^N_{m_{\GG_1}} \right\rangle \le C_{m_{\GG_1}}.
$$


\end{enumerate} 
\end{assump}

\begin{assump}[Assumptions for the existence of the pullback semigroup]\label{A2}
Consider the weight function $m_{\GG_1}$, the constraint function ${\bf m}_{\GG_1} : E \to \R^D$ and the set of constraints ${\bf R}_{\GG_1} \subset \R^D$ of Assumption {\bf (A1)}.
Then consider the associated probability space $ \PPP_{\GG_1}$ and the corresponding space of increments $\II \PPP_{\GG_1}$ as in Definition \ref{def:P_G}. Finally, consider a Banach space $\GG_1 \supset \II \PPP_{\GG_1}$ such that $\PPP_{\GG_1,{\bf r}}$ is endowed with the distance $\dist_{\GG_1}$ induced by the norm $\| \cdot \|_{\GG_1}$ for any constraint vector ${\bf r} \in {\bf R}_{\GG_1}$.

Assume that, for some $\delta\in (0,1]$ and some $\bar a \in (0,\infty)$, we have for any $a\in(\bar a,\infty)$ and ${\bf r} \in {\bf R}_{\GG_1}$: 
\begin{enumerate}[(i)]
\item The equation \eqref{eq:Q(ft)} generates a semigroup $S^{\infty}_t : \BB  \PPP_{\GG_1,a, {\bf r}}\to \BB  \PPP_{\GG_1,a, {\bf r}}$, which is $\delta$-H\"older continuous locally uniformly in time, in the sense that for any $\tau\in(0,\infty)$ there exists $C_\tau >0$ such that
$$
\forall\,f,g \in \BB  \PPP_{\GG_1,a, {\bf r}}, \qquad 
\sup_{0\le t\le \tau} \| S^{\infty}_t (g)- S^{\infty}_t (f) \|_{\GG_1} 
\le C_\tau \| g-f \|_{\GG_1}^{\delta}.
$$

\item The application $Q$ is bounded and $\delta$-H\"older continuous from $\BB  \PPP_{\GG_1,a, {\bf r}}$ into $\GG_1$.

\end{enumerate}

\end{assump}

\medskip

As a consequence os this assumption, the semigroups $S^\infty_t$ and $T^\infty_t$ are well-defined as well as the generator $G^\infty$ thanks to the following result \cite[Lemma 2.11]{MMchaos} (see also \cite[Lemma 4.1]{MMWchaos})
\begin{lemma}\label{lem:G_infty}
Assume {\bf (A2)}. For any $a \in (\bar a,\infty)$ and ${\bf r} \in {\bf R}_{\GG_1}$, the
pullback semigroup $T^\infty_t$ defined by 
$$ 
\forall \, f \in \BB \PPP_{\GG_1,a; {\bf r}}, \ \Phi \in  C_b( \BB
\PPP_{\GG_1,a, {\bf r}}), \quad T^\infty _t [\Phi](f) := \Phi\left( S^{\infty}_t (f)\right)
$$
is a $c_0$-semigroup of contractions on the Banach space $C_b( \BB \PPP_{\GG_1,a,{\bf r}})$. 

Its generator $G^\infty$ is an unbounded linear operator on $C_b(\BB P_{\GG_1,a,{\bf r}})$ with domain $\hbox{Dom}(G^\infty)$ containing $C^{1,\eta}_b( \BB \PPP_{\GG_1,a,{\bf r}})$. On the latter space, it is defined by the formula
\begin{equation}\label{eq:formulaGinfty}
\forall \, \Phi \in  C^{1,\eta}_b( \BB \PPP_{\GG_1,a,{\bf r}}), \ 
\forall \, f \in \BB \PPP_{\GG_1,a,{\bf r}}, \quad 
\left( G^\infty \Phi \right) (f) :=
\left \langle D\Phi[f], Q(f)\right\rangle.
\end{equation}

\end{lemma}

\medskip

\begin{assump}[Convergence of the generators]
Let $\PPP_{\GG_1}, m_{\GG_1}, \RRR_{\GG_1}$ be such as introduced in Assumption {\bf (A2)}. Define a weight function $1\le m'_{\GG_1}\le C m_{\GG_1}$ and the corresponding weight $\Lambda_1(f) := \langle f , m'_{\GG_1} \rangle $.

We assume that there exist a function $\eps(N)$ going to $0$ as $N\to\infty$ and $\eta \in (0,1]$ such that for all $\Phi \in \bigcap_{{\bf r} \in {\bf R}_{\GG_1}}  C^{2,\eta}_{\Lambda_1}(  \PPP_{\GG_1,{\bf r}};\R)$ we have
\begin{equation}\label{eq:hypA3}
\Norm{\left(M^N_{m_{\GG_1}}\right)^{-1} \left(G^N \pi^N_C  - \pi^N_C G^\infty \right)\Phi}_{L^\infty(\E_N)} \leq \eps(N) \sup_{{\bf r} \in \RRR_{\GG_1}} \left( \left[\Phi\right]_{C^{1,\eta}_{\Lambda_1}(  \PPP_{\GG_1,{\bf r}};\R)} + \left[\Phi\right]_{C^{2,0}_{\Lambda_1}(  \PPP_{\GG_1,{\bf r}};\R)}\right).
\end{equation}

\end{assump}

\medskip

\begin{assump}[Differential stability]
Consider a Banach space $\GG_2 \supset \GG_1$, $\GG_1$ defined in Assumption {\bf (A2)}, and the associated probability space $\PPP_{\GG_2}$ as in Definition \ref{def:P_G}, with weight function $m_{\GG_2}$, constraint function ${\bf m}_{\GG_2}$ and endowed with the metric induced from $\GG_2$.

We assume that the limit semigroup $S^\infty_t$ satisfies 
$$
S^{\infty}_t \in C^{2,\eta}_{\Lambda_2} \cap C^{1,(1+2\eta)/3}_{\Lambda_2} \cap C^{0,(2+\eta)/3}_{\Lambda_2}(  \PPP_{\GG_1,{\bf r}};  \PPP_{\GG_2}),
$$ 
for any ${\bf r} \in \RRR_{\GG_1}$, and that there exists $C_{4}^\infty>0$ such that
\begin{equation}\label{eq:hypA3}
\sup_{{\bf r} \in \RRR_{\GG_1}} \int_0^\infty \left( 
\left[ S^{\infty}_t \right]_{C^{1,(1+2\eta)/3}_{\Lambda_2}}
+ \left[ S^{\infty}_t \right]_{C^{0,(2+\eta)/3}_{\Lambda_2}}^2
+\left[ S^{\infty}_t \right]_{C^{2,0}_{\Lambda_2}}
+ \left[ S^{\infty}_t \right]_{C^{1,0}_{\Lambda_2}}^2   
 \right) \,dt \leq C_{4}^\infty,
\end{equation}
with $\Lambda_2 = \Lambda_1^{1/3}$ and where $\eta$ and $\Lambda_1$ are the same as in Assumption {\bf (A3)}.

\end{assump}

\medskip

\begin{assump}[Weak stability]
We assume that, for some probability space $  \PPP_{\GG_3}$ associated to a weight function $m_{\GG_3}$, a constraint ${\bf m}_{\GG_3}$, a set of constraints $\RRR_{\GG_3}$ and a distance $\dist_{\GG_3}$, there exists a constant $C_{5}^\infty>0$ such that for any ${\bf r}\in \RRR_{\GG_3}$,
\begin{equation}\label{eq:hypA5}
\forall\, f,g \in  \PPP_{\GG_3,{\bf r}}, \quad \sup_{t\ge 0} \,  \dist_{\GG_3}\left( S^{\infty}_t(f), S^{\infty}_t(g) \right)\leq C_{5}^\infty
\dist_{\GG_3} (f,g ).
\end{equation}

\end{assump}

\medskip

\begin{thm}[Abstract theorem]\label{thm:abstract} 
Let us consider a family of $N$-particle initial conditions $F^N_0 \in    \Psym(E^N)$ and the associated solution $F^N_t = S^N_t(F^N_0)$. Consider also a one-particle initial condition $f_0\in  \PPP(E)$ and the associated solution $f_t = S^{\infty}_t(f_0)$. 
Assume that Assumptions {\bf (A1)-({A2})-({A3})-({A4})-({A5})} hold for some spaces $  \PPP_{\GG_i}$, $\GG_i$ and $\FF_i$, $i=1,2,3$, with $\FF_i \subset C_b(E)$ and where $\FF_i$ and $\PPP_{\GG_i}$ are in duality.

Then there exists a constant $C\in(0,\infty)$ such that for any $N,\ell \in \N$, with $N\geq 2\ell$, and for any $\varphi = \varphi_1 \otimes \cdots \otimes \varphi_\ell \in \FF^{\otimes \ell}$, $\FF := \FF_1 \cap \FF_2 \cap \FF_3$
we have
\begin{equation}\label{eq:thm}
\begin{aligned}
&\sup_{t\ge 0} 
\left| \left\langle S^N_t (F^N_0) - \left(S^{\infty}_t(f_0)\right)^{\otimes N} , 
\varphi \otimes {\bf 1}^{N-\ell} \right\rangle   \right| \\
&\qquad
\leq C \Bigg[ \frac{\ell^2}{N}\, \|\varphi\|_{L^\infty} + C_{m_{\GG_1}}\,C_{4}^\infty\,\eps(N)\, \ell^3\, \|\varphi\|_{\FF_2^{\otimes 3}\otimes(L^\infty)^{\ell-3}} \\
&\qquad\qquad\qquad
+ C_{5}^\infty\, \ell\, \|\varphi\|_{\FF_3\otimes(L^\infty)^{\ell-1}} \,\WW_{1,\GG_3}(\pi^N_P F^N_0, \delta_{f_0})     \Bigg],
\end{aligned}
\end{equation}
where we recall that $\WW_{1,\GG_3}$ is defined in Definition \ref{def:wasserstein}.

As a consequence, if $F^N_0$ is $f_0$-chaotic the propagation of chaos holds uniformly in time and with quantitative rates in the number $N$ of particles (depending on quantitative rates of the chaoticity of the initial data).
\end{thm}

\begin{proof}[Proof of Theorem \ref{thm:abstract}] We split the term \eqref{eq:thm} into three parts:
\begin{equation*}
\begin{aligned}
&  \left\langle S^N_t (F^N_0) - \left(S^{\infty}_t(f_0)\right)^{\otimes N} , \varphi \otimes {\bf 1}^{N-\ell} \right\rangle   \\
&\qquad \leq   \left\langle S^N_t (F^N_0) , \varphi \otimes {\bf 1}^{N-\ell} \right\rangle 
- \left\langle S^N_t (F^N_0) , R^{\ell}_{\varphi}\circ \pi^N_E \right\rangle    \qquad &(=:T_1)\\
&\qquad\quad  +   \left\langle F^N_0  ,  T^N_t(R^{\ell}_{\varphi}\circ \pi^N_E ) \right\rangle 
- \left\langle F^N_0 , (T^{\infty}_t R^{\ell}_{\varphi})\circ \pi^N_E \right\rangle   \qquad&(=:T_2)\\
&\qquad\quad  +   \left\langle F^N_0  ,  (T^{\infty}_t R^{\ell}_{\varphi})\circ \pi^N_E \right\rangle 
- \left\langle \left(S^{\infty}_t(f_0)\right)^{\otimes N} , \varphi \otimes {\bf 1}^{N-\ell}  \right\rangle   \qquad&(=:T_3)
\end{aligned}
\end{equation*}
and we evaluate each of them. The first term is estimated by a combinatorial argument, which corresponds to the price we have to pay by using the injection based on the empirical measure. 
The second term is where the two dynamics are effectively compared, and in order to do so we first express the difference between the linear semigroups in terms of the difference of their generators, which are all well-defined thanks to Assumption {\bf (A1)} and {\bf (A2)} together with Lemma~\ref{lem:G_infty}; then the consistency estimate {\bf (A3)} on the generators, the stability estimate {\bf(A4)} on the limit semigroup and the Assumption {\bf (A1)} on the $N$-particle system yield a control of $T_2$.
Finally, thanks to the weak stability estimate {\bf (A5)} on the limit semigroup, the third term can be controlled in terms of the chaoticity of the initial data.

\medskip
\noindent
{\it Step 1.}
For the first term $T_1$, a classical combinatorial trick (see \cite{S6}, \cite[Lemma 2.14]{MMWchaos}, \cite[Lemma 3.3]{MMchaos}) implies
$$
| T_1 | \leq \frac{2\ell^2\norm{\varphi}_{L^{\infty}(E^\ell)}}{N}.
$$

\medskip
\noindent
{\it Step 2.}
We investigate now the second term $T_2$. First of all, thanks to Assumptions {\bf (A1)} and {\bf (A2)} the semigroups and generators are well-defined, and
we may rewrite $T_2$ as the difference of two dynamics in $C_b (\PPP(E))$ in the following way
$$
T_2 = \left\langle F^N_0 , (T^N_t \pi^N_C - \pi^N_C T^\infty_t) R^\ell_\varphi \right\rangle, 
$$ 
and we recall that by the definition of the generators it hold
$$
\frac{d}{ds} T^N_s = G^N T^N_s , \quad
\frac{d}{ds} T^\infty_s = G^\infty T^\infty_s.
$$ 
We then deduce that
\begin{equation*}
\begin{aligned}
T^N_t\pi^N_C    - \pi^N_C T^{\infty}_t 
&= -\int_0^t \frac{d}{ds} [ T^N_{t-s} \, \pi^N_C \, T^{\infty}_{s}   ] ds \\
 &= \int_0^t  T^N_{t-s} \, (G^N\pi^N_C - \pi^N_C G^\infty) \, T^{\infty}_{s}   ds,
\end{aligned}
\end{equation*}
from which we obtain, for any $t\ge 0$,
$$
\begin{aligned}
T_2 &= \left\langle F^N_0, \int_0^t    T^N_{t-s} \, (G^N\pi^N_C - \pi^N_C G^\infty) \, T^{\infty}_{s}   ds \, R^\ell_\varphi   \right\rangle \\
&= \int_0^t   \left\langle  F^N_0    ,  T^N_{t-s} \, (G^N\pi^N_C - \pi^N_C G^\infty) \, T^{\infty}_{s}    R^\ell_\varphi    \right\rangle \, ds.
\end{aligned}
$$
Then it follows that
\begin{equation}\label{eq:T2}
\begin{aligned}
|T_2| &\leq \int_0^\infty \Abs{ \Ps{(M^N_{m_{\GG_1}})S^N_{t-s}(F^N_0)}{(M^N_{m_{\GG_1}})^{-1}[G^N\pi^N_C - \pi^N_C G^\infty] \, T^{\infty}_{s} R^{\ell}_{\varphi}}    }ds\\
&\leq \sup_{t\ge 0} \Ps{M^N_{m_{\GG_1}}}{S^N_t(F^N_0)} \int_0^\infty {\Norm{(M^N_{m_{\GG_1}})^{-1}[G^N\pi^N_C - \pi^N_C G^\infty] \, T^{\infty}_{s} R^{\ell}_{\varphi}}}_{L^\infty(\E_N)} ds \\
&\leq C_{m_{\GG_1}} \, \eps(N)  \sup_{{\bf r}\in {\bf R}_{\!\GG_1}} \int_0^\infty \left({\left[T^{\infty}_{s} R^{\ell}_{\varphi}\right]}_{C^{1,\eta}_{\Lambda_1}(  \PPP_{\GG_1,{\bf r}};\R)} + {\left[T^{\infty}_{s} R^{\ell}_{\varphi}\right]}_{C^{2,0}_{\Lambda_1}(  \PPP_{\GG_1,{\bf r}};\R)} \right) \,ds,
\end{aligned}
\end{equation}
thanks to Assumptions {\bf (A1)} and {\bf (A3)}.

We fix ${\bf r} \in {\bf R}_{\GG_1}$ and we know that the application $ T^{\infty}_{s} R^{\ell}_{\varphi}  = R^{\ell}_{\varphi} \circ S^{\infty}_s$ and by Assumption {\bf (A4)} the limit (nonlinear) semigroup $S^\infty_s$ satisfies 
$$
S^{\infty}_s \in 
C^{2,\eta}_{\Lambda_2} \cap C^{1,(1+2\eta)/3}_{\Lambda_2} \cap C^{0,(2+\eta)/3}_{\Lambda_2} (  \PPP_{\GG_1,{\bf r}};  \PPP_{\GG_2}) .
$$ 
Moreover, with $\varphi \in \FF_2^{\otimes \ell}$ we have $R^{\ell}_{\varphi}\in C^{2,1} (  \PPP_{\GG_2};\R)$  (see Example~\ref{ex:R^l_phi}). Finally, by Lemma \ref{lem:chain-rule} we obtain that 
$$ 
T^{\infty}_{s} R^{\ell}_{\varphi}  \in C^{2,\eta}_{\Lambda_2^{3}} 
\cap C^{1,(1+2\eta)/3}_{\Lambda_2^{3}}
\cap C^{0,(2+\eta)/3}_{\Lambda_2^{3}} (\PPP_{\GG_1,{\bf r}};\R)
$$ 
with 
\begin{equation*}
\begin{aligned}
{\left[T^{\infty}_{s} R^{\ell}_{\varphi}\right]}_{C^{1,(1+2\eta)}_{\Lambda_2^{3}}(  \PPP_{\GG_1,{\bf r}};\R)} 
&\leq \norm{R^{\ell}_{\varphi}}_{C^{2,1}(\PPP_{\GG_2};\R)} \left( 
\left[ S^{\infty}_t \right]_{C^{1,(1+2\eta)/3}_{\Lambda_2}(  \PPP_{\GG_1,{\bf r}};  \PPP_{\GG_2})}
+ \left[ S^{\infty}_t \right]^2_{C^{0,(2+\eta)/3}_{\Lambda_2}(  \PPP_{\GG_1,{\bf r}};  \PPP_{\GG_2})}\right),\\ 
{\left[T^{\infty}_{s} R^{\ell}_{\varphi}\right]}_{C^{2,0}_{\Lambda_2^{3}}(  \PPP_{\GG_1,{\bf r}};\R)} 
&\leq \norm{R^{\ell}_{\varphi}}_{C^{2,1}(\PPP_{\GG_2};\R)}\left( 
\left[ S^{\infty}_t \right]_{C^{2,0}_{\Lambda_2}(\PPP_{\GG_1,{\bf r}};\PPP_{\GG_2})}
+ \left[ S^{\infty}_t \right]_{C^{1,0}_{\Lambda_2}(\PPP_{\GG_1,{\bf r}};\PPP_{\GG_2})}^2
\right).
\end{aligned}
\end{equation*}
From Assumption {\bf (A4)}, $\Lambda_2 = \Lambda_1^{1/3}$ and the estimate of $\norm{R^{\ell}_{\varphi}}_{C^{2,1}}$ in Example~\ref{ex:R^l_phi}, we can deduce, plugging the last estimate on \eqref{eq:T2}, that
\begin{equation}\label{}
\begin{aligned}
|T_2| &\leq 
C_{m_{\GG_1}} \,
  C_{4}^\infty  \, \eps(N) \, \ell^3\, \|\varphi\|_{\FF_2^{\otimes 3}\otimes(L^\infty)^{\ell-3}}.
\end{aligned}
\end{equation}

\medskip
\noindent
{\it Step 3.}
We rewrite the third term as
$$
\begin{aligned}
T_3 
&= \int_{E^N} R^{\ell}_{\varphi} \left( S^\infty_t (\mu^N_V) \right) \, F^N_0 (dV)
- \prod_{j=1}^\ell \int_E \varphi_j (v) \, S^\infty_t(f_0) (dv) \\
&= \int_{E^N} R^{\ell}_{\varphi} \left( S^\infty_t (\mu^N_V) \right) \, F^N_0 (dV)
- \int_{E^N} \left(\prod_{j=1}^\ell \int_E \varphi_j (v) \, S^\infty_t(f_0) (dv) \right) \, F^N_0 (dV) \\
&= \int_{E^N}  \left( R^{\ell}_{\varphi} \left( S^\infty_t (\mu^N_V) \right) 
- R^{\ell}_{\varphi} \left( S^\infty_t (f_0) \right)    \right) \, F^N_0(dV).
\end{aligned}
$$
We hence deduce, thanks to the Assumption {\bf (A5)} and the fact that $R^{\ell}_\varphi \in C^{0,1}(  \PPP_{\GG_3};\R)$ for $\varphi \in \FF_3^{\otimes\ell}$, that
\begin{equation}\label{eq:T3}
\begin{aligned}
|T_3| &
\leq \| R^{\ell}_{\varphi} \|_{C^{0,1}(\PPP_{\GG_3};\R)} 
\int_{E^N} \dist_{\GG_3}(S^{\infty}_t(\mu^N_V),S^{\infty}_t(f_0)) \, F^N_0(dV)  \\
&\leq \ell \, \|\varphi\|_{\FF_3\otimes(L^\infty)^{\ell-1}} \, C_{5}^\infty
\int_{E^N} \dist_{\GG_3}(\mu^N_V,f_0) \, F^N_0(dV). 
\end{aligned}
\end{equation}
With the definition of the Wasserstein distance (see Definition \ref{def:wasserstein}), since the set $\Pi(\pi^N_P F^N_0,\delta_{f_0}) = \{ \pi^N_P F^N_0 \otimes \delta_{f_0}  \}$ has only one element, 
we obtain
\begin{equation}\label{}
\begin{aligned}
\WW_{1,\GG_3}(\pi^N_P F^N_0,\delta_{f_0}) &:= \inf_{\pi\in \Pi(\pi^N_P F^N_0,\delta_{f_0})} \iint_{  \PPP(E) \times   \PPP(E) }
\dist_{\GG_3}(f,g) \, \pi(df,dg)\\
&= \iint_{  \PPP(E) \times   \PPP(E) }
\dist_{\GG_3}(f,g) \, \pi^N_P F^N_0(df) \, \delta_{f_0}(dg)\\
&= \int_{  \PPP(E) }
\dist_{\GG_3}(f,f_0) \, \pi^N_P F^N_0(df) \\
&= \int_{E^N}
\dist_{\GG_3}(\mu^N_V,f_0) \, F^N_0(dV).
\end{aligned}
\end{equation}
where we have used the definition of $\pi^N_P$ (see Section \ref{ssec:framework}) in the last equality.
We conclude then plugging this estimate on \eqref{eq:T3}.

\end{proof}


\section{Application to the Landau equation}\label{sec:app}

In this section we will use the consistency-stability method presented in the Section \ref{sec:abs-thm} to show the propagation of chaos for the Landau equation for Maxwellian molecules. 
In order to prove some estimates for the Landau equation that we need to apply the method of Section~\ref{sec:abs-thm}, we shall prove first the same type of estimates for the Boltzmann equation (as in \cite{MMchaos}) with a collisional kernel satisfying the grazing collisions \eqref{eq:grazing}. Then passing to the limit of grazing collisions we will recover the same results for the Landau equation.

\medskip

Our main theorems are:

\begin{thm}\label{thm:chaos}
Consider a $N$-particle initial condition $F^N_0 \in \Psym(\R^{dN})$ and, for all $t\ge0$, the associated solution of the $N$-particle Landau dynamics $F^N_t = S^N_t(F^N_0)$.
Consider also a one-particle initial condition $f_0\in  \PPP_6(\R^d)$, with zero momentum $\int v f_0 = 0$ and energy $\int |v|^2 f_0 =:\EE \in (0,\infty)$, and the associated solution of the limit (mean-field) Landau equation $f_t = S^\infty_t(f_0)$. 
Suppose further that there exists $\EE_0 \in (0,\infty)$ such that
\beqn\label{eq:suppF0}
\supp F^N_0 \subset \left\{ V \in \R^{dN} \mid  \frac1N \sum_{i=1}^N |v_i|^2 \le \EE_0 \right\}.
\eeqn
Then, for $\ell \in \N^*$, for all
$$
\varphi = \varphi_1\otimes\cdots \otimes \varphi_\ell \in \FF^{\otimes\ell}, \quad 
\FF := \left\{ \varphi:\R^d \to \R ;\; \| \varphi \|_\FF = \int_{\R^d} (1+|\xi|^6) |\hat\varphi(\xi)|\, d\xi <\infty  \right\},
$$
and for any $N\ge 2\ell$, there is a constant $C>0$ such that
\begin{equation*}
\begin{aligned}
&\sup_{t\ge 0} \left|\Ps{S^{N}_t(F^N_0) - (S^{\infty}_t(f_0))^{\otimes N}}{\varphi} \right|\\
&\qquad\leq
C\left[  \frac{\ell^2 }{N}\, \|\varphi\|_{L^\infty} + \frac{\ell^3}{N}\|\varphi\|_{\FF^3\otimes(L^\infty)^{\ell-3}} + \ell\,\|\varphi\|_{W^{1,\infty}\otimes(L^\infty)^{\ell-1}} \, \WW_{1,W_2}(\pi^N_P F^N_0,\delta_{f_0})  \right].
\end{aligned}
\end{equation*}
As a consequence, if $F^N_0$ is $f_0$-chaotic the third term of the right-hand side goes to $0$ when $N\to\infty$, which implies the propagation of chaos uniformly in time.

\end{thm}

\begin{thm}\label{thm:Wchaos}
Consider the same framework of Theorem \ref{thm:chaos}. Assume moreover that $f_0\in  \PPP_{6}(\R^d)\cap L^p(\R^d)$ for some $p>1$ and let $F^N_0 := [ f_0^{\otimes N}]_{\SS^N(\EE)} \in \Psym(\SS^N(\EE))$ (observe that \eqref{eq:suppF0} is satisfied by this choice of initial data with $\EE_0 = \EE$). Then it holds:

\begin{enumerate}[(1)]
\item For all $0<\epsilon< 9[(7d+6)^2(d+9)]^{-1}$, there exists a constant $C_\epsilon>0$ such that
\begin{equation*}
\sup_{t\ge 0} \frac{W_1(F^N_t,f^{\otimes N}_t)}{N}\leq 
C_\epsilon \, N^{-\epsilon}.
\end{equation*}

\item For all $t\ge 0$, for all $N\in \N^*$
$$
\frac{W_1(F^N_t,\gamma^N)}{N} \le p(t),
$$
for a polynomial rate $p(t)\to 0$ as $t\to\infty$ and where $\gamma^N$ is the uniform probability measure on $\SS^N(\EE)$.

\end{enumerate}


\end{thm}

\begin{rem}
This theorem also holds (with different quantitative rates) for other choices of initial data $F^N_0$ that are $f_0$-chaotic. In particular, if we consider $f_0 \in \PPP_6(\R^d)$ with compact support and $F^N_0 = f_0^{\otimes N} \in \Psym(\R^{dN})$. 
\end{rem}

The proof of Theorem~\ref{thm:chaos} relies on the proof of assumptions {\bf (A1)-(A2)-(A3)-(A4)-(A5)}, with a suitable choice of spaces, and then on the application of Theorem~\ref{thm:abstract}.
Furthermore, we shall prove Theorem~\ref{thm:Wchaos} using Theorem~\ref{thm:chaos} and some results from \cite{HaurayMischler,MMchaos,KC}.

\subsection{Proof of assumption (A1)}

Consider the $N$-particle SDE \eqref{eds0}. Since $b$ and $\sigma$ are Lipschitz, existence and uniqueness hold by standard arguments (see \cite[Chapter 5]{SV}). Hence it defines a semigroup $T^N_t$, we can then define its generator $G^N=G^N_L$ (given by \eqref{eq:generateurLandau}-\eqref{Lmaster2}) and its dual semigroup $S^N_t$, as explained in Section \ref{ssec:framework}.

\medskip

We have the following lemma.

\begin{lemma}\label{lem:A1i0}
The dynamics of the $N$-particle system \eqref{Lmaster2} conserves momentum and energy, more precisely there holds, for all $t\ge0$,
$$
\int_{\R^{dN}} \varphi\left( \sum_{i=1}^N v_{i,\alpha} \right) F^N_t(dV)
=\int_{\R^{dN}}  \varphi\left( \sum_{i=1}^N v_{i,\alpha} \right) F^N_0(dV)
, \qquad \alpha\in \{1,\dots,d  \}
$$
and
$$
\int_{\R^{dN}}  \varphi\left( |V|^2 \right) F^N_t(dV)
=\int_{\R^{dN}}  \varphi\left( |V|^2 \right) F^N_0(dV).
$$
\end{lemma}

\begin{rem}\label{rem:conservative}
We can easily observe during the proof that if we consider the $N$-particle system of Remark~\ref{rem:master} with generator $G^N_2$ \eqref{Lmaster1}, which is different from the system we considered here \eqref{Lmaster2}, we have conservation of energy
$$
\partial_t \left\langle f^N_t, \sum_{i=1}^{N} |v_i|^2 \right\rangle = 
\left\langle f^N_t, G^N_2 \sum_{i=1}^{N} |v_i|^2 \right\rangle = 0,
$$ 
however this is not true for all functions $\varphi=\varphi(|V|^2)$. Then, Lemma~\ref{lem:A1i}, which is a consequence of this lemma, does not hold for the $N$-particle system of Remark~\ref{rem:master}.
\end{rem}

\begin{proof}[Proof of Lemma \ref{lem:A1i0}]
Let us prove the second equality (energy conservation), the proof of the first one (momentum conservation) being similar.
Consider the Landau master equation \eqref{Lmaster2} and $\varphi(V)=\varphi(\abs{V}^2)$ smooth enough, we have then
$$
\nabla_{i}(\varphi(\abs{V}^2)) = \left( \partial_{v_{i,\alpha}} \varphi( |V|^2) \right)_{1\leq\alpha\leq d}
=2 \varphi'(|V|^2)v_i
$$
and, for $i\not =j$,
$$
\begin{aligned}
& \partial_{v_{i,\alpha}}\partial_{v_{j,\beta}} {\varphi(\abs{V}^2)}  = 
4\varphi''(\abs{V}^2)v_{i,\alpha}v_{j,\beta} \\
 &\partial_{v_{i,\alpha}}\partial_{v_{i,\beta}} {\varphi(\abs{V}^2)} 
= 4\varphi''(\abs{V}^2)v_{i,\alpha}v_{i,\beta} +2\varphi'(|V|^2)\delta_{\alpha\beta} .
\end{aligned}
$$
Denoting $\varphi'=\varphi'(\abs{V}^2)$ and $\varphi''=\varphi''(\abs{V}^2)$ for simplicity, we obtain
$$
\begin{aligned}
(\nabla^2_{ii}\varphi +\nabla^2_{jj}\varphi -\nabla^2_{ij}\varphi -\nabla^2_{\!ji}\varphi)_{\alpha\beta} 
&= 4\varphi'\delta_{\alpha\beta} 
+ 4\varphi'' \left( v_{i,\alpha}v_{i,\beta} + v_{j,\alpha}v_{j,\beta}
- v_{i,\alpha}v_{j,\beta} - v_{j,\alpha}v_{i,\beta}    \right)\\
&=4\varphi'\delta_{\alpha\beta} 
+ 4\varphi'' (v_i-v_j)_\alpha(v_i-v_j)_\beta.
\end{aligned}
$$
Therefore we have
$$
\begin{aligned}
b(v_i-v_j)(\nabla_i\varphi(\abs{V}^2) - \nabla_j \varphi(\abs{V}^2)) &= -2|v_i-v_j|^{\gamma}(v_i-v_j)\cdot2\varphi'(\abs{V}^2)(v_i-v_j)\\
 &= -4\varphi'(\abs{V}^2)|v_i-v_j|^{\gamma +2} \\
\end{aligned}
$$
and
$$
\begin{aligned}
& a(v_i-v_j):(\nabla^2_{ii}\varphi +\nabla^2_{jj}\varphi -\nabla^2_{ij}\varphi -\nabla^2_{\!ji}\phi)=\\
&\quad = |v_i-v_j|^{\gamma}\sum_{\alpha,\beta=1}^{d}\Bigg\{
\left[|v_i-v_j|^2\delta_{\alpha\beta}-(v_i-v_j)_\alpha(v_i-v_j)_\beta\right]4\varphi'\delta_{\alpha\beta} \\
&\quad\quad 
+\left[|v_i-v_j|^2\delta_{\alpha\beta}-(v_i-v_j)_\alpha(v_i-v_j)_\beta\right]
 4\varphi''(v_i-v_j)_\alpha(v_i-v_j)_\beta\Bigg\}\\
&\quad =: |v_i-v_j|^{\gamma}\{ T_1 + T_2\}.
\end{aligned}
$$
Computing $T_1$ we have
$$
\begin{aligned}
T_1&=4\varphi'\sum_{\alpha,\beta=1}^{d} \left[|v_i-v_j|^2\delta_{\alpha\beta}-(v_i-v_j)_\alpha(v_i-v_j)_\beta\right]\delta_{\alpha\beta}\\
&=8\varphi' \abs{v_i-v_j}^2,
\end{aligned}
$$
and computing $T_2$
$$
\begin{aligned}
T_2&=4\varphi''\sum_{\alpha,\beta=1}^{d} 
\left[|v_i-v_j|^2(v_i-v_j)_\alpha(v_i-v_j)_\beta\delta_{\alpha\beta}-(v_i-v_j)^2_\alpha(v_i-v_j)^2_\beta\right]\\
&=4\varphi''\left\{ 
 |v_i-v_j|^4 - \left[ \sum_{\alpha=1}^d (v_i-v_j)^2_\alpha  \right]^2
\right\}\\
&=0.
\end{aligned}
$$
Gathering previous estimates we obtain
$$
\begin{aligned}
&b(v_i-v_j)(\nabla_i\varphi - \nabla_j \varphi) 
+ \frac{1}{2}a(v_i-v_j):(\nabla^2_{ii}\varphi +\nabla^2_{jj}\varphi -\nabla^2_{ij}\varphi -\nabla^2_{\!ji}\varphi)\\
&\quad= -4\varphi'\,|v_i-v_j|^{\gamma +2} 
+ \frac{1}{2}8\varphi' \, \abs{v_i-v_j}^{\gamma +2} \\
&\quad=0,
\end{aligned}
$$
which implies, for all $t\ge 0$,
$$
\begin{aligned}
\int \varphi(\abs{V}^2) \, F^N_t (dV) = \int \varphi(\abs{V}^2) \, F^N_0( dV).
\end{aligned}
$$

\end{proof}

\begin{lemma}\label{lem:A1i}
Consider $F^N_0$ such that 
$$
\supp F^N_0 \subset \{V \in \R^{dN} ; \; M^N_2(V)= \frac1N \sum_{i=1}^N |v_i|^2 \leq \EE_0\}.
$$
Then there holds
$$
\forall\, t>0, \quad \supp F^N_t \subset \{V \in \R^{dN} ; \; M^N_2(V)\leq \EE_0\}.
$$
\end{lemma}

\begin{proof}[Proof of Lemma \ref{lem:A1i}]
It is a consequence of Lemma \ref{lem:A1i0}, with $\varphi(|V|^2) = {\bf 1}_{|V|^2 > N\EE_0}$. 
Consider a mollifier $(\rho_\eta)$ for $\eta>0$, i.e. $\rho_\eta(x) = \eta^{-1}\rho(\eta^{-1}x)$, with $\rho\in C^\infty_c (\R)$, $\rho \ge 0$ and $\supp \rho \subset B_1$, and define $\varphi_\eta = \rho_\eta * \varphi$. Using Lemma~\ref{lem:A1i0} we have, for all $\eta$ and for all $t\ge 0$,
$$
\int_{\R^{dN}} \varphi_\eta \, F^N_t(dV) = \int_{\R^{dN}} \varphi_\eta \, F^N_0(dV).
$$
Passing to the limit $\eta\to 0$ we obtain
$$
\int_{\R^{dN}} {\bf 1}_{|V|^2 > N\EE_0} \, F^N_t(dV) = \int_{\R^{dN}} {\bf 1}_{|V|^2 > N\EE_0} \, F^N_0(dV) = 0.
$$
\end{proof}

\begin{lemma}\label{lem:A1ii}
Consider $F^N_0$ such that $\Ps{F^N_0}{M^N_k} \le C_k$ for $k>2$. Then there holds
$$
\sup_{t\ge 0}\Ps{F^N_t}{M^N_k} \le C_k.
$$
\end{lemma}

\begin{proof}[Proof of Lemma \ref{lem:A1ii}]
Consider $F^{N,\eps}_t$ the solution of the Boltzmann $N$-particle system \eqref{eq:MasterBoltzmann}-\eqref{eq:BoltzmannGenerator} with grazing collisions \eqref{eq:grazing}. Then from \cite[Lemma 5.3]{MMchaos}, we obtain the desired result for $F^{N,\eps}_t$ with a constant independent of $\eps$. We conclude passing to the grazing collisions limit $\eps\to0$.
\end{proof}

Consider the constraint function ${\bf m}_{\GG_1} : \R^d\to \R_+ \times \R^d$, ${\bf m}_{\GG_1}(v) = (|v|^2,v)$ with the set of constraints $\RRR_{\GG_1} : \{ (r,\bar r) \in \R_+\times \R^d; \; |\bar r|^2 \le r \le \EE_0\}$. We have then $\E_N = \{ V \in \R^{dN}; \, M^N_2(V) \le \EE_0  \}$ and Lemma \ref{lem:A1i} proves Assumption {\bf (A1i)}. Moreover, Lemma \ref{lem:A1ii} proves {\bf (A1ii)} with the weight function $m_{\GG_1}(v) := \langle v \rangle^6 = (1+|v|^2)^3$.

\subsection{Proof of assumption (A2)}\label{ssec:A2} 
Let us define the spaces of probabilities (and the corresponding bounded, constrained and increments subsets, see Definition~\ref{def:P_G})
$$
\PPP_{\GG_1} := \{ f\in \PPP(\R^d);\; \langle  f , m_{\GG_1} \rangle = M_6(f) < \infty \},
$$
and, for ${\bf r} \in \RRR_{\GG_1}$, more precisely ${\bf r}=(r,\bar r)=(r, \bar r_1,\dots, \bar r_d)$, the constrained space
$$
\PPP_{\GG_1,{\bf r}} := \{ f \in \PPP_{\GG_1}; \; \ps{f}{|v|^2}=r, \, \ps{f}{v_\alpha}= \bar r_\alpha \text{ for } \alpha=1,\dots,d \}.
$$
We define then for some $a\in(0,\infty)$ the bounded set
$$
\BB\PPP_{\GG_1,a} := \{ f\in \PPP_{\GG_1} ; \; M_6(f) \le a\},
$$
and, for any ${\bf r} \in \RRR_{\GG_1}$, the bounded constrained set
$$
\BB\PPP_{\GG_1,a,{\bf r}} :=  \{f\in \BB\PPP_{\GG_1,a} ; \; \ps{f}{|v|^2}=r, \, \ps{f}{v_\alpha}=\bar r_\alpha \text{ for } \alpha=1,\dots,d   \}.
$$
Moreover we define the vector space
$$
\GG_1 := \{ \rho \in M^1_6(\R^d); \; \langle \rho , 1 \rangle = \langle \rho , v_\alpha \rangle = \langle \rho , |v|^2 \rangle = 0 \text{ for } \alpha = 1,\dots, d    \}
$$
endowed with the (general) Fourier-based norm $\| \cdot \|_{\GG_1} = \Nt \cdot \Nt_2$ defined in \eqref{eq:general-toscani} (see Definition~\ref{def:general-toscani}).

Finally, we endow these probability measure spaces with the distance $\dist_{\GG_1}$ associated with the norm $\| \cdot \|_{\GG_1} = \Nt \cdot \Nt_2$. Remark that for any $f,g \in \PPP_{\GG_1, \bf r}$, for some $\bf r \in \bf R_{\GG_1}$, it holds $\| f-g \|_{\GG_1} = |f-g|_2$, where $| \cdot |_2$ is the usual Fourier-based norm defined in \eqref{eq:toscani} (see Definition~\ref{def:toscani}), because $f$ and $g$ have same momentum.

\medskip

Let us state some know results. Concerning the Cauchy theory for the spatially homogeneous Landau equation for Maxwellian molecules \eqref{eq:landau}, we refer to Villani \cite{Vi1} for a $L^1(\R^d)$-theory and to Gu\'erin \cite{Gue1,Gue2} for a $\PPP(\R^d)$-theory. More precisely, if $ f_0 \in \PPP_2(\R^d)$ then there exists a probability flow solution $ (f_t)_{t \ge 0}$ to \eqref{eq:landau}, where $f_t \in \PPP_2 (\R^d) $, that conserves momentum and energy from \cite{Gue1}, moreover this solution is unique from \cite{Gue2} or from Lemma~\ref{lem:A2i} below. Furthermore we also have uniform in time propagation of moment bounds from \cite{Vi1}. Therefore, it follows that there exists $\bar a \in (0,\infty)$ such that for any $a \in [\bar a,\infty)$ and any $\bf r \in \bf R_{\GG_1}$, the nonlinear semigroup $S^\infty_t$ verifies $S^\infty_t : \BB \PPP_{\GG_1, a, \bf r} \to \BB\PPP_{\GG_1, a, \bf r}$.

\Black

We investigate now the H\"older regularity of the semigroup in the following lemma.

\begin{lemma}\label{lem:A2i}
Let $f_0,g_0 \in  \PPP_2(\R^d)$ with same momentum, i.e.\ $\langle f_0,v \rangle = \langle g_0,v \rangle$, and consider the solutions $(f_t)_{t \ge 0}$ and $(g_t)_{t\ge 0}$ of Landau equation for Maxwellian molecules \eqref{eq:landau}-\eqref{eq:landau2} associated to $f_0$ and $g_0$, respectively. Then
\begin{equation}\label{eq:A2i}
\sup_{t\geq 0} {\abs{f_t-g_t}}_{2} \le {\abs{f_0-g_0}}_{2}.
\end{equation}
\end{lemma}

\begin{rem}
Let us mention that this result can be found in \cite{Vi1} proving uniqueness for the Landau equation for Maxwellian molecules. There the author indicates that we can prove it using the known result for the Boltzmann equation for Maxwellian molecules from \cite{ToVi} and then passing to the limit of grazing collisions.  
\end{rem}

\begin{proof}[Proof of Lemma \ref{lem:A2i}]
Let us split the prove into two steps. First we prove the lemma for the Boltzmann equation then we recover the result for Landau equation passing to the limit of grazing collisions. 

\medskip
\noindent
{\it Step 1.} 
We shall prove first the desired result for the Boltzmann equation with true Maxwellian molecules. This result is proved in \cite{ToVi,MMchaos}, but we write it for completeness and because we want to pass to the limit of grazing collisions.

Consider the solutions $f_t^\eps$ and $g_t^\eps$ of Boltzmann equation for Maxwellian molecules \eqref{eq:Boltzmann}-\eqref{eq:Boltzmann2} with initial data $f_0$ and $g_0$, respectively, and with collision kernel $b_\eps$ satisfying \eqref{eq:grazing}.
Denote $d^\eps:=g^\eps-f^\eps$ and $s^\eps:=g^\eps+f^\eps$, then the equation satisfied by $d^\eps$ is
\begin{equation*}\label{}
\partial_t d^\eps = \frac{1}{2}\Big[ Q_{B,\eps}(s^\eps,d^\eps) + Q_{B,\eps}(d^\eps,s^\eps) \Big].
\end{equation*}
Performing the Fourier transform (see \cite{Bobylev}) and denoting $D^\eps = \hat d^\eps$, $S^\eps = \hat s^\eps$, we have
\begin{equation*}\label{}
\partial_t D^\eps(\xi) =  \int_{\Sp^{d-1}} b_{\eps}(\sigma\cdot \hat \xi)
\left[ \frac{D^\eps(\xi^+) S^\eps(\xi^-)}{2} + \frac{D^\eps(\xi^-) S^\eps(\xi^+)}{2} - D^\eps(\xi)   \right]d\sigma
\end{equation*}
where $\xi^+ = \frac{\xi + \abs{\xi}\sigma}{2}$, $\xi^- = \frac{\xi - \abs{\xi}\sigma}{2}$ and $\hat \xi = \xi/\abs{\xi}$.

We recall that $b_{\eps}$ is not integrable so we perform the following cut-off, which will be relaxed in the end,
\begin{equation}\label{eq:cut-off}
\int_{\Sp^{d-1}} b_{\eps}^K(\sigma\cdot\hat \xi)d\sigma=K,\qquad b_{\eps}^K=b_{\eps}\mathbf 1_{\abs{\theta}\geq \delta(K)},
\end{equation}
for some function $\delta$ such that $\delta(K)\to 0$ as $K\to+\infty$, so that $b_{\eps} = b_{\eps}^K + b_{\eps}^C$. In \cite{ToVi,MMchaos}, we observe that the remainder term
$$
R_{\eps}^K (\xi ) :=  \int_{\Sp^{d-1}} b_{\eps}^C (\sigma\cdot \hat \xi)
\left[ \frac{D^\eps(\xi^+) S^\eps(\xi^-)}{2} + \frac{D^\eps(\xi^-) S^\eps(\xi^+)}{2} - D^\eps(\xi)   \right]d\sigma
$$
verifies, for any $\xi\in \R^d$, $|R_\eps^K(\xi)| \le r_\eps^K |\xi|^2$, where $r_\eps^K \to 0$ as $K\to\infty$, and $r_\eps^K$ depends on the second order moments of $d^\eps$ and $s^\eps$. Indeed, 
using that $D^\eps(0)=\partial_{\xi_i} D^\eps (0)=0$ for all $i \in \{ 1 , \dots, d \}$, $S(0)=2$, and the fact that $\sup_{|\eta|\le |\xi|}|\partial_{\xi_i}\partial_{j} D^\eps(\eta)|$ and $\sup_{|\eta|\le |\xi|}|\partial_{\xi_i}\partial_{j} D^\eps(\eta)|$ are bounded thanks to the bounds on the second order moments of $d^\eps$ and $s^\eps$, there holds
$$
\bal
&|D^\eps(\xi^+) S^\eps(\xi^-) + D^\eps(\xi^-) S^\eps(\xi^+) - 2 D^\eps(\xi)| \\
&\qquad \le |S^\eps(\xi^-)| |D^\eps(\xi^+) - D^\eps(\xi)| + |D^\eps(\xi)||S^\eps(\xi^-) - S^\eps(0)|
+ |D^\eps(\xi^-)| |S^\eps(\xi^+)| \\
&\qquad \le C |\xi|^2 (1-\cos\theta)^{1/2},
\eal
$$
and we conclude since $b_{\eps}^C (\cos\theta)(1-\cos\theta)^{1/2}$ is integrable.

Using that $\norm{S^\eps}_{\infty}\le 2$, we have
\begin{equation*}\label{}
\frac{d}{dt}\frac{ \abs{D^\eps(\xi)}}{\abs{\xi}^2} + K\frac{ \abs{D^\eps(\xi)}}{\abs{\xi}^2} \leq \sup_{\xi\in\R^d}\frac{ \abs{D^\eps(\xi)}}{\abs{\xi}^2} 
\left( \sup_{\xi\in\R^d}  \int_{\Sp^{d-1}} b^K_{\eps}(\sigma\cdot \hat \xi)
\left( \abs{\hat\xi^+}^2 + \abs{\hat\xi^-}^2 \right) d\sigma\right) + r_\eps^K
\end{equation*}
with
$$
\abs{\hat\xi^+}^2 = \frac{1}{2}\left( 1 +  \sigma\cdot \hat \xi  \right),\qquad
\abs{\hat\xi^-}^2 = \frac{1}{2}\left( 1 - \sigma\cdot \hat \xi  \right).
$$
One obtains
\begin{equation*}\label{}
\frac{d}{dt}\frac{ \abs{D^\eps(\xi)}}{\abs{\xi}^2} + K\frac{ \abs{D^\eps(\xi)}}{\abs{\xi}^2} \leq K \sup_{\xi\in\R^d}\frac{ \abs{D^\eps(\xi)}}{\abs{\xi}^2} 
+ r_\eps^K,
\end{equation*}
and by a Gronwall's lemma one deduces
\begin{equation*}\label{}
\sup_{\xi\in\R^d}\frac{ \abs{D^\eps_t(\xi)}}{\abs{\xi}^2} \leq  \sup_{\xi\in\R^d}\frac{ \abs{D^\eps_0(\xi)}}{\abs{\xi}^2} + t \,r_\eps^K.
\end{equation*}
Relaxing the cut-off $K\to\infty$ one proves
\begin{equation}\label{eq:d-eps}
|f_t^\eps - g_t^\eps|_2 \le |f_0-g_0|.
\end{equation}

\medskip
\noindent
{\it Step 2.} 
Since \eqref{eq:d-eps} does not depend on $\eps$ and the solution of the Boltzmann equation $f_t^\eps$ converges towards the solution of the Landau equation $f_t$ (see \cite{Vi2}) when $\eps\to0$, we obtain the desired result. 

\end{proof}

Therefore the Landau semigroup $S^{\infty}_t$ is $C^{0,1}$ from $\BB\PPP_{\GG_1, a, \bf r}$ to $\BB\PPP_{\GG_1, a, \bf r}$ and Assumption~{\bf (A2i)} is proved.

\medskip

To prove {\bf (A2ii)} we use \cite[Lemma 5.5]{MMchaos}, valid for the Boltzmann operator with grazing collisions $Q_{B,\eps}$, which says that there exists $C>0$ and $\delta\in(0,1]$ such that for any $f,g \in \BB\PPP_{\GG_1,a,{\bf r}}$ we have
$$
| Q_{B,\eps}(f,f) - Q_{B,\eps}(g,g) |_2 \le C |f-g|_2^\delta,
$$
with a constant $C$ that does not depend on $\eps$.
Finally, passing to the limit of grazing collisions $\epsilon \to 0$, we have that $Q_{B,\eps}~\to~Q_L$ (see e.g.\ \cite{Des,Vi2}). We prove then Assumption~{\bf (A2ii)} also for the Landau equation.

\subsection{Proof of assumption (A3)}\label{ssec:A3}

Let $\Lambda_1(f) := \langle f, m'_{\GG_1} \rangle$ with the weight function $m'_{\GG_1}(v) := \langle v \rangle^4$, where we recall that $m_{\GG_1}=\langle v \rangle^6$, and then consider the generator $G^N$ of the Landau master equation \eqref{Lmaster2}. 

Then we have the following lemma, which proves {\bf (A3)}.

\begin{lemma}\label{lem:A3}
For all $\Phi\in \bigcap_{{\bf r} \in \RRR_{\GG_1}} C^{2,\eta}_{\Lambda_1}( \PPP_{\GG_1,{\bf r}};\R)$ there exists $C>0$ such that
\begin{equation}\label{eq:A3}
\Norm{\left(M^N_{m_{\GG_1}}\right)^{-1}\left(G^N\pi^N_C - \pi^N_C G^\infty \right)\Phi}_{L^\infty(\mathbb E_N)} \leq \frac{C}{N} \sup_{{\bf r}\in \RRR_{\GG_1}}
\left[\Phi \right]_{C^{2,0}_{\Lambda_1}( \PPP_{\GG_1,{\bf r}};\R)}.
\end{equation}
\end{lemma}

\begin{proof}[Proof of Lemma \ref{lem:A3}]
The application $\R^{dN} \to \PPP_{\GG_1}$, $V = (v_1, \dots, v_N)\mapsto \mu^{N}_{V}$ is of class $C^{2,1}$ with (see \cite[Lemma 7.4]{MMWchaos})
\begin{equation}
\partial_{v_{i,\alpha}} \mu^N_V = \frac{1}{N}\partial_{\alpha}\delta_{v_i},
\quad
\partial^2_{v_{i,\alpha},v_{i,\beta}} \mu^N_V  = \frac{1}{N}\partial^2_{\alpha\beta}\delta_{v_i} \\
\end{equation}
and for $i\neq j,\; \partial^2_{v_{i,\alpha},v_{j,\beta}} \mu^N_V=0$.
Let $\Phi\in C^{2,\eta}_{\Lambda_1}( \PPP_{\GG_1};\R)$, so the application $\R^{dN} \to \mathbb R$, $V\mapsto \Phi(\mu^{N}_{V})$ is also $C^{2,\eta}$. Indeed, let $\phi = D\Phi[\mu^{N}_{V}] \in \GG_1' $ and we have
$$
\begin{aligned}
\partial_{v_{i,\alpha}}(\Phi(\mu^{N}_{V})) &= 
\left\langle D\Phi[\mu^{N}_{V}] , \partial_{v_{i,\alpha}}\mu^{N}_{V} \right\rangle 
= \left\langle D\Phi[\mu^{N}_{V}] ,  \frac{1}{N}\partial_{v_{i,\alpha}}\delta_{v_i} \right\rangle=
\frac{1}{N}\partial_{\alpha}\phi(v_i) \\
\partial^2_{v_{i,\alpha},v_{i,\beta}}\Phi(\mu^{N}_{V}) &=
\left\langle D\Phi[\mu^{N}_{V}] , \frac{1}{N}
\partial^2_{v_{i,\alpha},v_{i,\beta}}\delta_{v_i} \right\rangle +
 D^2\Phi[\mu^{N}_{V}]\left( \frac{1}{N}\partial_{\alpha}\delta_{v_i},
  \frac{1}{N}\partial_{\beta}\delta_{v_i} \right) \\
&= \frac{1}{N}\partial^2_{\alpha,\beta}\phi(v_i) + 
\frac{1}{N^2} D^2\Phi[\mu^{N}_{V}]\left( \partial_{\alpha}\delta_{v_i},\partial_{\beta}\delta_{v_i}\right).
\end{aligned}
$$
We compute, for any $V \in \E_N = \{ V \in \R^{dN}; \, M^N_2(V) \le \EE_0  \}$,
$$
\begin{aligned}
(G^N \pi^N_C \Phi)(V)&=G^N \Phi(\mu^{N}_{V})\\
&= \frac{1}{N}\sum_{i,j=1}^{N}\sum_{\alpha=1}^{d} b_{\alpha}(v_i-v_j)
\left[ \partial_{v_{i,\alpha}}(\Phi(\mu^{N}_{V})) 
     - \partial_{v_{j,\alpha}}(\Phi(\mu^{N}_{V}))    \right] \\
&\quad
+ \frac{1}{2N}\sum_{i,j=1}^{N}\sum_{\alpha,\beta=1}^{d} 
   a_{\alpha\beta}(v_i-v_j)
\left[ \partial^2_{v_{i,\alpha},v_{i,\beta}}(\Phi(\mu^{N}_{V}) )
     + \partial^2_{v_{j,\alpha},v_{j,\beta}}(\Phi(\mu^{N}_{V}) )\right.\\ 
&\quad\quad\quad  \left.
     -  \partial^2_{v_{i,\alpha},v_{j,\beta}}(\Phi(\mu^{N}_{V}))
     -  \partial^2_{v_{j,\alpha},v_{i,\beta}}(\Phi(\mu^{N}_{V})) \right] \\
&= \frac{1}{N}\sum_{i,j=1}^{N}\sum_{\alpha=1}^{d} b_{\alpha}(v_i-v_j)
\left[ \frac{1}{N}\partial_{\alpha}\phi(v_i) -\frac{1}{N} \partial_{\alpha}\phi(v_j) \right] \\
&\quad + \frac{1}{2N}\sum_{i,j=1}^{N} \sum_{\alpha,\beta=1}^{d}
   a_{\alpha\beta}(v_i-v_j)
\left[ \frac{1}{N}\partial^2_{\alpha,\beta}\phi(v_i) + \frac{1}{N}\partial^2_{\alpha,\beta}\phi(v_j)    \right]  \quad(=:I_1)\\
&\quad +\frac{1}{2N}\sum_{i,j=1}^{N}\sum_{\alpha,\beta=1}^{d} 
   a_{\alpha\beta}(v_i-v_j)
\left[ \frac{1}{N^2} D^2\Phi[\mu^{N}_{V}]\left( \partial_{v_{i,\alpha}}\delta_{v_i},\partial_{v_{i,\beta}}\delta_{v_i}\right) \right. \\
&\quad\quad\quad  + \left.
\frac{1}{N^2} D^2\Phi[\mu^{N}_{V}]\left( \partial_{v_{j,\alpha}}\delta_{v_j},\partial_{v_{j,\beta}}\delta_{v_j}\right) 
-2 \frac{1}{N^2} D^2\Phi[\mu^{N}_{V}]\left( \partial_{v_{i,\alpha}}\delta_{v_i},\partial_{v_{j,\beta}}\delta_{v_j}\right)  \right] \quad(=:I_2).\\
\end{aligned}
$$
For the first term, using the empirical measure, we can write
$$
\begin{aligned}
I_1 =& \int\!\!\int \sum_{\alpha=1}^{d}b_{\alpha}(v-v_*)
\left[  \partial_{\alpha}\phi(v) - \partial_{\alpha}\phi(v_*)\right]
\mu_{V}^{N}(dv)\mu_{V}^{N}(dv_*) \\
&+ \frac{1}{2}\int\!\!\int \sum_{\alpha,\beta=1}^{d}
a_{\alpha\beta}(v-v_*)\left[ \partial^2_{\alpha,\beta}\phi(v) + \partial^2_{\alpha,\beta}\phi(v_*)    \right]\mu_{V}^{N}(dv)\mu_{V}^{N}(dv_*)\\
=& \ps{Q_L(\mu_{V}^{N},\mu_{V}^{N})}{\phi} = 
\ps{Q_L(\mu_{V}^{N},\mu_{V}^{N})}{D\Phi[\mu^{N}_{V}]} = 
 (G^{\infty} \Phi)(\mu^N_V)=
(\pi^N_C G^{\infty} \Phi)(V),
\end{aligned}
$$
thanks to Lemma \ref{lem:G_infty}.
For the second one, using that $|a_{\alpha\beta}(v_i-v_j)| \le |v_i-v_j|^2$ and also 
$|D^2\Phi[\mu^{N}_{V}]\left( \partial_{v_{i,\alpha}}\delta_{v_i},\partial_{v_{i,\beta}}\delta_{v_i}\right)| \le [\Phi ]_{C^{2,0}_{\Lambda_1}( \PPP_{\GG_1,{\bf r} };\R)} \, \Lambda_1(\mu^N_V) \, \|\partial_{v_{i,\alpha}}\delta_{v_i} \|_{\GG_1}^{1+\eta}$, we deduce that there exists $C>0$ such that
$$
\begin{aligned}
|I_2| &\leq \frac{C}{N} \, \left[\Phi \right]_{C^{2,0}_{\Lambda_1}( \PPP_{\GG_1,{\bf r} };\R)} \, \Lambda_1(\mu^N_V)
\frac{1}{N^2}\sum_{i,j=1}^{N}|v_i-v_j|^{2}.
\end{aligned}
$$
Since $\Lambda_1(\mu^N_V) = M^N_{m'_{\GG_1}} (V) \le  C \, M^N_{m_{\GG_1}} (V)$ and $M^N_2(V) \le \EE_0$, we conclude that
$$
\begin{aligned}
|I_2| 
&\leq \frac{C \EE_0}{N}  \, \left[\Phi \right]_{C^{2,0}_{\Lambda_1}( \PPP_{\GG_1,{\bf r} };\R)} \,
M^N_{m_{\GG_1}}(V),
\end{aligned}
$$
and therefore we prove \eqref{eq:A3}.

\end{proof}

\subsection{Proof of assumption (A4)}\label{ssec:A4}

In the same way of the Section \ref{ssec:A2}, we will use here the Boltzmann equation and then perform the asymptotic of grazing collisions to prove the results for the Landau equation.

We define the following equations, denoting by $Q$ the symmetrized version of the Landau operator $Q_{L}$ for Maxwellian molecules \eqref{eq:landau2}, i.e. $Q(f,g) = [Q_L(f,g) + Q_L(g,f)]/2$,
\begin{equation}\label{equations-A4-0-Landau}
\left\{
\begin{array}{l l}
\partial_t f = Q(f,f), & f\vert_{t=0} =f_0, \\
\partial_t g = Q(g,g), & g\vert_{t=0} =g_0, \\
\partial_t h =2 Q(f,h), & h\vert_{t=0} =g_0-f_0, \\
\partial_t u = 2 Q(f,u) + Q(h,h), & u\vert_{t=0} =0,
\end{array}
\right.
\end{equation}
and the new variables
$$
d:=g-f,\quad s:= g+f,\quad \w := g-f-h,\quad \psi:=g-f-h-u,
$$ 
which satisfy
\begin{equation}\label{equations-A4-Landau}
\left\{
\begin{array}{l l}
\partial_t d = Q(s,d), & d\vert_{t=0} =g_0-f_0, \\
\partial_t \w =  Q(s,w) +  Q(h,d), & \w\vert_{t=0} =0, \\
\partial_t \psi =  Q(s,\psi) +  Q(h,w)+ Q(u,d), 
& \psi\vert_{t=0} =0.
\end{array}
\right.
\end{equation}


\begin{lemma}\label{lem:A4a}
Consider $f_0,g_0\in  \PPP_{\GG_1,\bf r},\, {\bf r} \in \RRR_{\GG_1}$, and the solutions 
$f_t,g_t,h_t$ of \eqref{equations-A4-0-Landau}-\eqref{equations-A4-Landau}.
There exists $\lambda_1\in(0,\infty)$ that for any $\eta\in[2/3,1]$, there exists $C_{\eta} >0$ such that we have
\begin{equation}\label{eq:A4a}
\begin{aligned}
{\abs{g_t-f_t}}_2 &\leq C_{\eta} \, e^{-(1-\eta)\lambda_1 t}\, M_4( f_0 + g_0) ^{1/3} \, {\abs{g_0-f_0}}^\eta_2 ,\\
{\abs{h_t}}_2 &\leq C_{\eta} \, e^{-(1-\eta)\lambda_1 t} \, M_4( f_0 + g_0) ^{1/3} \, {\abs{g_0-f_0}}^\eta_2 .
\end{aligned}
\end{equation}

\end{lemma}

\begin{proof}[Proof of Lemma \ref{lem:A4a}]
We split the proof into two steps. Again, we shall first prove the lemma for Boltzmann equation with a kernel satisfying the grazing collisions, which is proved in \cite{MMchaos}, and then passing to the limit of grazing collisions we prove the same result for the Landau equation.

\medskip
\noindent
{\it Step 1.} 
Let us denote by $Q_{\eps}$ the symmetrized version of the Boltzmann operator $Q_{B,\eps}$ with Maxwellian molecules \eqref{eq:Boltzmann2} with kernel $b_\eps$ satisfying \eqref{eq:grazing}, i.e. $Q_{\eps}(f,g) = [Q_{B,\eps}(f,g) + Q_{B,\eps}(g,f)]/2$.

Consider the solutions $f_t^\eps$, $g_t^\eps$ and $h^\eps_t$ of
\begin{equation}\label{eq:feps}
\left\{
\begin{array}{l l}
\partial_t f^\eps = Q_{\eps}(f^\eps,f^\eps), & f^\eps\vert_{t=0} =f_0, \\
\partial_t g^\eps = Q_{\eps}(g^\eps,g^\eps), & g^\eps\vert_{t=0} =g_0, \\
\partial_t h^\eps =2 Q_{\eps}(f^\eps,h^\eps), & h^\eps\vert_{t=0} =g_0-f_0, 
\end{array}
\right.
\end{equation}
and define $d^\eps:=g^\eps-f^\eps$ which satisfies (where $s^\eps := g^\eps + f^\eps$)
\begin{equation*}
\partial_t d^\eps = Q_{\eps}(s^\eps,d^\eps), \qquad d^\eps\vert_{t=0} =g_0-f_0.
\end{equation*}
As in Lemma~\ref{lem:A2i} we denote $D^\eps = \hat d^\eps$ and $S^\eps = \hat s^\eps$.
Define $\bar D^\eps = D^\eps - \hat\MM_4[d^\eps]$ (see Definition~\ref{def:general-toscani}). Then the equation satisfied by $\bar D^\eps $ is
$$
\partial_t \bar D^\eps  = \hat Q_{\eps}(D^\eps, S^\eps) - \partial_t \hat\MM_4[d^\eps]
= \hat Q_{\eps}( \bar D^\eps , S^\eps) + \hat Q_{\eps}(\hat\MM_4[d^\eps],S^\eps) - \hat\MM_4[Q_{\eps}(d^\eps,s^\eps)].
$$ 
From \cite[Lemma 5.6]{MMchaos} we know that, for any $\xi\in\R^d$,
$$
\left|\hat Q_{\eps}(\hat\MM_4[d^\eps],S^\eps) - \hat\MM_4[Q_{\eps}(d^\eps,s^\eps)]     \right|
\leq C |\xi|^4 \sum_{|\alpha|\leq 3} |M_\alpha[f^\eps-g^\eps]|,
$$
and also, from \cite[Theorem 8.1]{tanaka}, that there are constants $C,\delta>0$ such that for all $t\ge 0$ 
$$
\sum_{|\alpha|\leq 3} |M_\alpha[f_t^\eps-g_t^\eps]| \le C e^{-\delta t} \sum_{|\alpha|\leq 3} |M_\alpha[f_0-g_0]|.
$$

Then, following \cite{MMchaos} and
performing the same cut-off as in the proof of Lemma \ref{lem:A2i}, we have that
\begin{equation}\label{}
\begin{aligned}
\frac{d}{dt}\frac{\abs{\bar D^\eps }}{\abs{\xi}^4} + K \frac{\abs{\bar D^\eps }}{\abs{\xi}^4}
&\leq 
\left(\sup_{\xi\in\R^d} \frac{\abs{\bar D^\eps }}{\abs{\xi}^4} \right) 
\left(\sup_{\xi\in\R^d} \int_{\Sp^{d-1}} b^K_{\eps}(\sigma\cdot \hat \xi)
\left( \abs{\hat \xi^+}^4 + \abs{\hat \xi^-}^4    \right)d\sigma \right) \\
&\quad+ C\, e^{-\delta t} \left(  \sum_{\abs{\alpha}\leq 3} \Abs{M_\alpha[f_0-g_0]} \right)
+ \frac{|R_\eps^K|}{|\xi|^4}.
\end{aligned}
\end{equation}
where the remainder term
$$
R_\eps^K(\xi) := \int_{\Sp^{d-1}} b_\eps^C (\sigma\cdot \hat\xi) 
\left[ \frac12 \bar D^\eps (\xi^+) S^\eps(\xi^-) + \frac12 \bar D^\eps (\xi^-) S^\eps(\xi^+)
- \bar D^\eps(\xi)   \right] d\sigma
$$
satisfies, for any $\xi\in\R^d$, $|R_\eps^K(\xi)|\leq r_\eps^K |\xi|^4$, with $r_\eps^K \to 0$ as $K\to\infty$, and $r_\eps^K$ depends on the fourth order moments of $d^\eps$ and $s^\eps$. Indeed, we have
$$
\bal
&| \bar D^\eps (\xi^+) S^\eps(\xi^-) +  \bar D^\eps (\xi^-) S^\eps(\xi^+)- 2D(\xi)| \\
&\qquad\le
|S^\eps(\xi^-) |   |\bar D^\eps (\xi^+) - \bar D^\eps (\xi^-)|
+ |\bar D^\eps (\xi)| |S^\eps(\xi^-) - S^\eps(0)| 
+ |\bar D^\eps (\xi^-)|  |S^\eps(\xi^+)| \\
&\qquad \le C |\xi|^4 (1-\cos\theta)^{1/2},
\eal
$$
where we use that $\nabla^\alpha_\xi \bar D^\eps (0)=0$ for all multi-index $|\alpha|\leq 3$ and also that $\sup_{|\eta| \le |\xi|} \nabla^\alpha_\xi \bar D^\eps (\eta)$ and $\sup_{|\eta| \le |\xi|} \nabla^\alpha_\xi S^\eps (\eta)$ are bounded for $|\alpha|=4$ thanks to the bounds on the fourth moment of $d^\eps$ and $s^\eps$. As in Lemma~\ref{lem:A2i}, the claim follows since $b_\eps^C (\cos\theta)(1-\cos\theta)^{1/2}$  is integrable.

We denote
$$
\lambda_K := \int_{\Sp^{d-1}} b^K_{\eps}(\sigma\cdot \hat \xi)
\left( \abs{\hat \xi^+}^4 + \abs{\hat \xi^-}^4    \right)d\sigma = 
\int_{\Sp^{d-1}} b^K_{\eps}(\sigma\cdot \hat \xi)\frac{1}{2}
\left(  1 + (\sigma\cdot\hat \xi)^2   \right)d\sigma
$$
and we compute
$$
\begin{aligned}
\lambda_K - K &= - \frac{1}{2}\int_{\Sp^{d-1}} b^K_{\eps}(\sigma\cdot \hat \xi)
\left(  1 - (\sigma\cdot\hat \xi)^2   \right)d\sigma\\
&\xrightarrow[K\to\infty]{} - \frac{1}{2}\int_{\Sp^{d-1}} b_{\eps}(\sigma\cdot \hat \xi)
\left(  1 - (\sigma\cdot\hat \xi)^2   \right)d\sigma =: - \bar\lambda_\eps \in (-\infty,0)\\
&
\xrightarrow[\eps\to 0]{}  - \bar\lambda \in (-\infty,0).
\end{aligned}
$$
One can now apply Gronwall's lemma to obtain
\begin{equation*}\label{}
\begin{aligned}
\sup_{\xi\in\R^d}\frac{\abs{ \bar D^\eps_t}}{\abs{\xi}^4} 
\leq e^{(\lambda_K -K)t} \sup_{\xi\in\R^d} \frac{\abs{\bar D^\eps_0}}{\abs{\xi}^4}  
+ C \left(  \sum_{\abs{\alpha}\leq 3} \Abs{M_\alpha[f_0-g_0]} \right) 
\left( \frac{e^{-\delta t} - e^{(\lambda_K-K)t}}{K-\lambda_K - \delta} \right) \\
\quad + r_\eps^K\left( \frac{1 - e^{(\lambda_K-K)t}}{K-\lambda_K}\right).
\end{aligned}
\end{equation*}
Then relaxing the cut-off $K\to\infty$ and choosing  $0<\lambda < \min(\delta,\bar\lambda_\eps)$ one has (remark that $\lambda$ depends on $\eps$)
\begin{equation}\label{eq:D}
\begin{aligned}
\sup_{\xi\in\R^d}\frac{\abs{ \bar D^\eps_t}}{\abs{\xi}^4} 
\leq C \, e^{-\lambda t} \left( \sup_{\xi\in\R^d} \frac{\abs{\bar D^\eps_0}}{\abs{\xi}^4}  
+  \sum_{\abs{\alpha}\leq 3} \Abs{M_\alpha[f_0-g_0]} \right) .
\end{aligned}
\end{equation}
Using a standard interpolation argument \cite{MMchaos}, one obtains
\begin{equation}\label{eq:inter}
\begin{aligned}
{\abs{g-f}}_2 &\leq {\abs{g-f - \MM_4[g-f]}}_2 + C  \left(  \sum_{\abs{\alpha}\leq 3} 
\Abs{M_\alpha[g-f]} \right) \\
&\leq \| \hat g- \hat f - \hat \MM_4[f-g] \|_{L^\infty}^{1/2}   \, {\abs{g-f - \MM_4[g-f]}}^{1/2}_4 + C  \left(  \sum_{\abs{\alpha}\leq 3} \Abs{M_\alpha[g-f]} \right) \\
&\leq  C \, M_4( f_0 + g_0) \, e^{-(\lambda/2)t}.
\end{aligned}
\end{equation}
Finally one concludes by writing
\begin{equation}\label{eq:f-g}
\begin{aligned}
{\abs{g_t^\eps-f_t^\eps}}_2 &\leq {\abs{g_t^\eps-f_t^\eps}}^\eta_2\,{\abs{g_t^\eps-f_t^\eps}}^{1-\eta}_2 \\
&\leq C_{\eta} \, e^{-(1-\eta)\lambda t} \, M_4( f_0 + g_0) ^{1/3} \, {\abs{g_0-f_0}}^\eta_2
\end{aligned}
\end{equation}
where we have used the last estimate \eqref{eq:inter}, Lemma \ref{lem:A2i} and the fact that $M_4( f_0 + g_0) ^{1-\eta} \le M_4( f_0 + g_0) ^{1/3}$ for $\eta\in[2/3,1]$ . For $h_t$ one proves the result by the same computations.

\medskip
\noindent
{\it Step 2.} 
Let us now deduce the result for solutions $f_t$ and $g_t$ of the Landau equation. Coming back to
\eqref{eq:D} and choosing $0<\lambda_1< \min(\delta, \bar \lambda)$, where $\bar\lambda_\eps \to \bar\lambda\in(0,\infty)$ as $\eps\to 0$, we recover \eqref{eq:f-g} with the exponent $\lambda_1$ which does not depend on $\eps$. Hence, passing to the limit $\eps\to 0$, we have $g^\eps-f^\eps \to g-f$ and then
$$
|g_t - f_t |_2 \le C_\eta \, e^{-(1-\eta)\lambda_1 t} \, M_4( f_0 + g_0) ^{1/3} \, {\abs{g_0-f_0}}^\eta_2.
$$ 
Rigorously, we write
$$
|g_t - f_t |_2 \le |g_t - g_t^\eps |_2 + |f_t - f_t^\eps |_2 + |g_t^\eps - f_t^\eps |_2,
$$
then for the third term on the right-hand side we use \eqref{eq:f-g} with exponent $\lambda_1$ that does not depend on $\eps$, and for the other two terms we use that $g_t^\eps$ weakly converges towards $g_t$ in $L^1$ (see Villani \cite{Vi2}), hence $|g_t - g_t^\eps |_2 \to 0$ when $\eps\to 0$ and we deduce
$$
|g_t - f_t |_2 \le C_\eta \, e^{-(1-\eta)\lambda_1 t} \, M_4( f_0 + g_0) ^{1/3} \, {\abs{g_0-f_0}}^\eta_2.
$$

\end{proof}


\begin{lemma}\label{lem:A4b}
Consider $f_0,g_0\in  \PPP_{\GG_1,\bf r},\, {\bf r} \in \RRR_{\GG_1}$, and the solutions $f_t$,  $g_t$, $h_t$, $\w_t$ and $u_t$ of \eqref{equations-A4-0-Landau} and \eqref{equations-A4-Landau}.
There exists $\lambda_1\in(0,\infty)$ that for any $\eta\in[2/3,1]$, there exists $C_{\eta}$ such that we have
\begin{equation}\label{eq:A4b}
\begin{aligned}
{\abs{g_t - f_t - h_t}}_4 &\leq C_{\eta} \, e^{-(1-\eta)\lambda_1 t} \, M_4( f_0 + g_0) ^{1/3} \,{\abs{g_0-f_0}}^{1+\eta}_2\\
{\abs{u_t}}_4 &\leq C_{\eta} \, e^{-(1-\eta)\lambda_1 t} \, M_4( f_0 + g_0) ^{1/3} \,{\abs{g_0-f_0}}^{1+\eta}_2
\end{aligned}
\end{equation}

\end{lemma}

\begin{proof}[Proof of Lemma \ref{lem:A4b}]
Let us split the proof into two steps.

\medskip
\noindent
{\it Step 1.} 
As in Lemma~\ref{lem:A4a}, we consider $Q_{\eps}$ the symmetrized version of the Boltzmann operator $Q_{B,\eps}$ and the solutions $f_t^\eps$, $g_t^\eps$ and $h_t^\eps$ of \eqref{eq:feps}.

Consider also $u_t^\eps$ solution of
\begin{equation}\label{eq:reps}
\partial_t u^\eps = 2 Q_{\eps}(f^\eps,u^\eps) + Q_{\eps}(h^\eps,h^\eps), \qquad 
u^\eps\vert_{t=0} =0,
\end{equation}
and define $\w^\eps:=g^\eps-f^\eps-h^\eps$ which satisfies
\begin{equation*}
\partial_t \w^\eps =  Q_{\eps}(s^\eps,\w^\eps) +  Q_{\eps}(h^\eps,d^\eps), \qquad \w^\eps\vert_{t=0} =0.
\end{equation*}

First of all, we remark that $\w_t^\eps$ has moments equals to zero up to order $3$. Indeed, let us prove that, for $\alpha\in\N^d$, 
\begin{equation}\label{eq:Malpha}
\begin{aligned}
\forall \, |\alpha|\le 3,\quad M_\alpha (\w_t^\eps) := \int_{\R^d} v^\alpha \, \w_t^\eps(v) \, dv = 0.
\end{aligned}
\end{equation}
Following \cite[Lemma 5.8]{MMchaos} we know that for Maxwellian molecules the $\alpha$-moment of the Boltzmann operator $Q_{B,\eps}(g,h)$ is a sum of terms given by the product of moments of $g$ and $h$, then we obtain
\begin{equation}\label{}
\begin{aligned}
\forall \, |\alpha|\le 3, \quad 
\frac{d}{dt} M_\alpha (\w_t^\eps)  = \sum_{\beta\le \alpha} a_{\alpha,\beta} \, M_{\beta}(\w_t^\eps)M_{\alpha-\beta}(s_t^\eps) + \sum_{\beta\le \alpha} a_{\alpha,\beta} \, M_{\beta}(h_t^\eps)M_{\alpha-\beta}(d_t^\eps)
\end{aligned}
\end{equation}
and we deduce that
\begin{equation}\label{}
\begin{aligned}
\forall \, |\alpha|\le 3, \quad 
\frac{d}{dt} M_\alpha (\w_t^\eps)  = \sum_{\beta\le \alpha} a_{\alpha,\beta}\, M_{\beta}(\w_t^\eps)M_{\alpha-\beta}(s_t^\eps) 
\end{aligned}
\end{equation}
because for all $|\alpha|\le 1$ we have $M_\alpha (h_t^\eps)=M_\alpha (d_t^\eps)=0$. 
We conclude to \eqref{eq:Malpha} by the fact that $\w_0=0$. Therefore $|\w^\eps|_4$ is well defined and we do not need to "take-off the moments of $\w^\eps$".


Let us denote $\Omega^\eps = \hat \w^\eps$ and $H^\eps = \hat h^\eps$. We perform then the same cut-off as in Lemmas \ref{lem:A2i} and \ref{lem:A4a} and we have the following equation for $\w_t^\eps$
\begin{equation}\label{}
\begin{aligned}
&\frac{d}{dt}\frac{\abs{\Omega^\eps(\xi)}}{\abs{\xi}^4} + K \frac{\abs{\Omega^\eps(\xi)}}{\abs{\xi}^4}\\
&\qquad\leq\frac12 \sup_{\xi\in\R^d} \int_{\Sp^{d-1}} b^K_{\eps}(\sigma\cdot \hat \xi)
\left( \frac{\abs{\Omega^\eps(\xi^+)}\abs{S^\eps(\xi^-)}}{\abs{\xi}^4} + \frac{\abs{\Omega^\eps(\xi^-)}\abs{S^\eps(\xi^+)}}{\abs{\xi}^4}   \right)d\sigma \qquad (=:T_1)\\
&\qquad\quad +\frac12\sup_{\xi\in\R^d} \int_{\Sp^{d-1}} b^K_{\eps}(\sigma\cdot \hat \xi)
\left( \frac{\abs{H^\eps(\xi^+)}\abs{D^\eps(\xi^-)}}{\abs{\xi}^4} + \frac{\abs{H^\eps(\xi^-)}\abs{D^\eps(\xi^+)}}{\abs{\xi}^4}   \right)d\sigma \qquad (=:T_2)\\
&\qquad\quad +\frac{|R_\eps^K|}{|\xi|^4},
\end{aligned}
\end{equation}
where the remainder term
$$
R_\eps^K(\xi) := \frac12 \int_{\Sp^{d-1}} b_\eps^C (\sigma\cdot \hat\xi) 
\left[ \Omega^\eps(\xi^+) S^\eps(\xi^-) + \Omega^\eps (\xi^-) S^\eps(\xi^+)+
H^\eps(\xi^+) D^\eps(\xi^-) + H^\eps (\xi^-) D^\eps(\xi^+)   \right] d\sigma
$$
satisfies, for any $\xi\in\R^d$, $|R_\eps^K(\xi)|\leq r_\eps^K |\xi|^4$, with $r_\eps^K \to 0$ as $K\to\infty$, and $r_\eps^K$ depends on moments of order $4$ of $d^\eps$, $s^\eps$, $h^\eps$ and $\w^\eps$. To see this, we argue as in Lemma~\ref{lem:A4a}, using that $\w^\eps$ has vanishing moments up to order $3$, see \eqref{eq:Malpha}. 

We compute first $T_1$ using the fact that ${\norm{S^\eps}}_\infty \le 2$
$$
\begin{aligned}
T_1 &\leq 
\left(\sup_{\xi\in\R^d}\frac{\abs{\Omega^\eps(\xi)}}{\abs{\xi}^4}\right)
\sup_{\xi\in\R^d}\int_{\Sp^{d-1}} b^K_{\eps}(\sigma\cdot \hat \xi)
\left( \abs{\hat \xi^+}^4 + \abs{\hat \xi^-}^4    \right)d\sigma \\
&\leq \lambda_K \sup_{\xi\in\R^d}\frac{\abs{\Omega^\eps(\xi)}}{\abs{\xi}^4}.
\end{aligned}
$$
where $\lambda_K$ is the same that in the proof of Lemma $\ref{lem:A4a}$. Next, we compute $T_2$
$$
\begin{aligned}
T_2 &\leq \left( \sup_{\xi\in\R^d}\frac{\abs{H^\eps(\xi)}}{\abs{\xi}^2}\right)
\left( \sup_{\xi\in\R^d}\frac{\abs{D^\eps(\xi)}}{\abs{\xi}^2}\right)
\sup_{\xi\in\R^d}\int_{\Sp^{d-1}} b^K_{\eps}(\sigma\cdot \hat \xi)
\left(  \frac{\abs{\xi^+}^2 \abs{\xi^-}^2}{\abs{\xi}^4}     \right)d\sigma \\
&\leq {\abs{h_t^\eps}}_2 {\abs{d_t^\eps}}_2 \sup_{\xi\in\R^d}\int_{\Sp^{d-1}} b^K_{\eps}(\sigma\cdot \hat \xi)
\left(  1-  \sigma\cdot \hat \xi  \right)d\sigma\\
&\leq \Lambda_\eps \, e^{-(1-\eta)\lambda t} \, M_4( f_0 + g_0) ^{1-\eta}\, {\abs{h_0}}_2 {\abs{d_0}}^{\eta}_2 
\end{aligned}
$$
where we have used the estimates of Lemmas \ref{lem:A2i} and \ref{lem:A4a}, and $\Lambda_\eps$ is defined in \eqref{eq:grazing}.
After these computations we obtain
\begin{equation*}\label{}
\begin{aligned}
\frac{d}{dt}\frac{\abs{\Omega^\eps(\xi)}}{\abs{\xi}^4} + K \frac{\abs{\Omega^\eps(\xi)}}{\abs{\xi}^4}
\leq \lambda_K \sup_{\xi\in\R^d} \frac{\abs{\Omega^\eps(\xi)}}{\abs{\xi}^4} +
\Lambda_\eps \, e^{-(1-\eta)\lambda t} \, M_4( f_0 + g_0) ^{1-\eta}\, {\abs{d_0}}^{1+\eta}_2 
+ r_\eps^K
\end{aligned}
\end{equation*}
and by Gronwall's lemma
\begin{equation}\label{}
\begin{aligned}
\sup_{\xi\in\R^d}\frac{\abs{\Omega^\eps_t(\xi)}}{\abs{\xi}^4} 
\leq
\Lambda_\eps\, M_4( f_0 + g_0) ^{1-\eta}\,{\abs{d_0}}^{1+\eta}_2  \left( \frac{e^{-(1-\eta)\lambda t} - e^{(\lambda_K-K)t}}{K-\lambda_K - (1-\eta)\lambda}   \right)\\
+ r_\eps^K\left( \frac{1 - e^{(\lambda_K-K)t}}{K-\lambda_K}\right).
\end{aligned}
\end{equation}

Finally, we conclude by relaxing the cut-off parameter $K\to\infty$ and choosing $(1-\eta)\lambda \in (0,\bar\lambda_\eps)$ where $\bar \lambda_\eps$ is the same as in Lemma \ref{lem:A4a}, therefore we have
\begin{equation}\label{eq:weps}
{\abs{\w_t^\eps}}_4 \leq C_\eta \, \Lambda_\eps \, e^{-(1-\eta)\lambda t}\, M_4( f_0 + g_0) ^{1-\eta} \,{\abs{g_0-f_0}}^{1+\eta}_2 .
\end{equation}

We obtain the same estimation for $u_t^\eps$.

\medskip
\noindent
{\it Step 2.} 
Consider the solutions $f$, $g$ and $h$ of \eqref{equations-A4-0-Landau} .  

Let us choose $\lambda_1$ such that $0<(1-\eta)\lambda_1<  \bar \lambda$, where $\bar\lambda_\eps \to \bar\lambda\in(0,\infty)$ as $\eps\to 0$. Then we recover \eqref{eq:weps} with the exponent $\lambda_1$ that does not depend on $\eps$. Hence, passing to the limit $\eps\to 0$, we have $g^\eps-f^\eps-h^\eps \to g-f-h$ (grazing collisions limit \cite{Vi2}), and in the right-hand side of \eqref{eq:weps} we have $\Lambda_\eps \to \Lambda \in (0, \infty)$ (see \eqref{eq:grazing}). Then it follows
$$
|g_t-f_t - h_t|_4 \le C_\eta \, \Lambda \, e^{-(1-\eta)\lambda_1 t}\, M_4( f_0 + g_0) ^{1/3} \,{\abs{g_0-f_0}}^{1+\eta}_2.
$$ 

\end{proof}


\begin{lemma}\label{lem:A4c}
Consider $f_0,g_0\in  \PPP_{\GG_1,\bf r},\, {\bf r} \in \RRR_{\GG_1}$, and the solution $\psi_t$ of \eqref{equations-A4-Landau}.
There exists $\lambda_1\in(0,\infty)$ such that for any $\eta\in[2/3,1]$, there exists $C_{\eta}$ such that we have
\begin{equation}\label{eq:A4c}
\begin{aligned}
{\abs{g_t-f_t - h_t - u_t}}_6 &\leq C_{\eta} \, e^{-(1-\eta)\lambda_1 t} \, M_4( f_0 + g_0) ^{1/3} \,{\abs{g_0-f_0}}^{2+\eta}_2
\end{aligned}
\end{equation}

\end{lemma}

\begin{proof}[Proof of Lemma \ref{lem:A4c}]
We prove the lemma in two steps.

\medskip
\noindent
{\it Step 1.} 
Consider the solutions $g_t^\eps$, $f_t^\eps$ and $h_t^\eps$ of \eqref{eq:feps} and $u_t^\eps$ solution of \eqref{eq:reps}. Define $\psi_t^\eps :=g_t^\eps - f_t^\eps-h_t^\eps - u_t^\eps $ that satisfies

\begin{equation*}
\partial_t \psi^\eps =  Q_{\eps}(s^\eps,\psi^\eps) +  Q_{\eps}(h^\eps,\w^\eps)+ Q_{\eps}(u^\eps,d^\eps), 
\qquad \psi^\eps\vert_{t=0} =0.
\end{equation*}

First of all, let us prove that $\psi_t^\eps$ has moments equals to zero up to order $5$, more precisely, for $\alpha\in\N^d$,
\begin{equation}\label{}
\begin{aligned}
\forall \, |\alpha|\le 5, \quad
 M_\alpha (\psi_t^\eps) := \int_{\R^d} v^\alpha \, \psi_t^\eps(v) \, dv = 0.
\end{aligned}
\end{equation}
In fact, as in the proof of Lemma \ref{lem:A4b}, we can compute the $\alpha$-moment of $\psi$
\begin{equation}\label{}
\begin{aligned}
\forall \, |\alpha|\le 5, \quad 
\frac{d}{dt} M_\alpha (\psi_t^\eps) &= \sum_{\beta\le \alpha} a_{\alpha,\beta} \, M_{\beta}(\psi_t^\eps)M_{\alpha-\beta}(s_t^\eps) + \sum_{\beta\le \alpha} a_{\alpha,\beta} \, M_{\beta}(h_t^\eps)M_{\alpha-\beta}(\w_t^\eps)  \\
&+ \sum_{\beta\le \alpha} a_{\alpha,\beta} \, M_{\beta}(r_t^\eps)M_{\alpha-\beta}(d_t^\eps).
\end{aligned}
\end{equation}
Since
\begin{equation}\label{}
\begin{aligned}
\forall \, |\alpha|\le 2,\;& M_\alpha (h_t^\eps) = M_\alpha(d_t^\eps) = 0,\\
\forall \, |\alpha|\le 3,\;& M_\alpha (\w_t^\eps) = M_\alpha(r_t^\eps) = 0,  
\end{aligned}
\end{equation}
we deduce that
\begin{equation}\label{eq:Mpsi}
\begin{aligned}
\forall \, |\alpha|\le 5,\; \frac{d}{dt} M_\alpha (\psi_t^\eps) = \sum_{\beta\le \alpha} a_{\alpha,\beta} \, M_{\beta}(\psi_t^\eps)M_{\alpha-\beta}(s_t^\eps) 
\end{aligned}
\end{equation}
and we conclude thanks to $\psi_0=0$. Then $|\psi^\eps|_6$ is well defined.

Denoting $\Psi^\eps = \hat \psi^\eps$ and $U^\eps = \hat u^\eps$, we perform the same cut-off as in Lemmas \ref{lem:A2i}, \ref{lem:A4a} and \ref{lem:A4b}, and it gives the following equation for $\psi_t$
\begin{equation}\label{}
\begin{aligned}
&\frac{d}{dt}\frac{\abs{\Psi^\eps(\xi)}}{\abs{\xi}^6} + K \frac{\abs{\Psi^\eps(\xi)}}{\abs{\xi}^4}\\
&\qquad\leq \frac12 \sup_{\xi\in\R^d} \int_{\Sp^{d-1}} b^K_{\eps}(\sigma\cdot \hat \xi)
\left( \frac{\abs{\Psi^\eps(\xi^+)}\abs{S^\eps(\xi^-)}}{\abs{\xi}^6} + \frac{\abs{\Psi^\eps(\xi^-)}\abs{S^\eps(\xi^+)}}{\abs{\xi}^6}   \right)d\sigma \qquad (=:T_1)\\
&\qquad\quad
+\frac12\sup_{\xi\in\R^d} \int_{\Sp^{d-1}} b^K_{\eps}(\sigma\cdot \hat \xi)
\left( \frac{\abs{H^\eps(\xi^+)}\abs{\Omega^\eps(\xi^-)}}{\abs{\xi}^6} + \frac{\abs{H^\eps(\xi^-)}\abs{\Omega^\eps(\xi^+)}}{\abs{\xi}^6}   \right)d\sigma \qquad (=:T_2) \\
&\qquad\quad
+\frac12\sup_{\xi\in\R^d} \int_{\Sp^{d-1}} b^K_{\eps}(\sigma\cdot \hat \xi)
\left( \frac{\abs{U^\eps(\xi^+)}\abs{D^\eps(\xi^-)}}{\abs{\xi}^6} + \frac{\abs{U^\eps(\xi^-)}\abs{D^\eps(\xi^+)}}{\abs{\xi}^6}   \right)d\sigma \qquad (=:T_3)\\
&\qquad\quad
 + \frac{|R_\eps^K|}{|\xi|^6},
\end{aligned}
\end{equation}
where the remainder term
\begin{multline}
R_\eps^K(\xi) := \frac12 \int_{\Sp^{d-1}} b_\eps^C (\sigma\cdot \hat\xi) 
\Big[ {\Psi^\eps(\xi^+)}{S^\eps(\xi^-)} + {\Psi^\eps(\xi^-)}{S^\eps(\xi^+)}
+H^\eps(\xi^+) \Omega^\eps(\xi^-)\\
 + H^\eps (\xi^-) \Omega^\eps(\xi^+)
+U^\eps(\xi^+) D^\eps(\xi^-) + U^\eps (\xi^-) D^\eps(\xi^+)
\Big] d\sigma
\end{multline}
satisfies, for any $\xi\in\R^d$, $|R_\eps^K(\xi)|\leq r_\eps^K |\xi|^6$, with $r_\eps^K \to 0$ as $K\to\infty$, and $r_\eps^K$ depends on moments of order $6$ of $d^\eps$, $s^\eps$, $h^\eps$, $\w^\eps$, $u^\eps$ and $\psi^\eps$. It easily follows arguing as in Lemmas~\ref{lem:A4a} and \ref{lem:A4c}, using \eqref{eq:Mpsi} and the bounds on moments of order $6$.

We compute first $T_1$ using the fact that ${\norm{S^\eps}}_\infty \le 2$
$$
\begin{aligned}
T_1 &\leq \sup_{\xi\in\R^d}\frac{\abs{\Psi^\eps(\xi)}}{\abs{\xi}^6}
\sup_{\xi\in\R^d}\int_{\Sp^{d-1}} b^K_{\eps}(\sigma\cdot \hat \xi)
\left( \abs{\hat \xi^+}^6 + \abs{\hat \xi^-}^6    \right)d\sigma \\
&\leq \alpha_K \sup_{\xi\in\R^d}\frac{\abs{\Psi^\eps(\xi)}}{\abs{\xi}^6}.
\end{aligned}
$$
Let us analyse $\alpha_K$,
$$
\alpha_K = \int_{\Sp^{d-1}} b^K_{\eps}(\sigma\cdot \hat \xi)
\left( \abs{\hat \xi^+}^6 + \abs{\hat \xi^-}^6   \right)d\sigma = 
\int_{\Sp^{d-1}} b^K_{\eps}(\sigma\cdot \hat \xi)\frac{1}{4}
\left(  1 + 3(\sigma\cdot\hat \xi)^2   \right)d\sigma
$$
and we compute
$$
\begin{aligned}
\alpha_K - K &= - \int_{\Sp^{d-1}} b^K_{\eps}(\sigma\cdot \hat \xi)\frac{1}{(4/3)}
\left(  1 - (\sigma\cdot\hat \xi)^2   \right)d\sigma\\
&\xrightarrow[K\to\infty]{} - \int_{\Sp^{d-1}} b_{\eps}(\sigma\cdot \hat \xi)\frac{1}{(4/3)}
\left(  1 - (\sigma\cdot\hat \xi)^2   \right)d\sigma =: - \bar\alpha_\eps \in (-\infty,0)\\
&\xrightarrow[\eps\to 0]{}  - \bar\alpha \in (-\infty,0).
\end{aligned}
$$

Next, we compute $T_2$
$$
\begin{aligned}
T_2 &\leq \left( \sup_{\xi\in\R^d}\frac{\abs{H^\eps(\xi)}}{\abs{\xi}^2}\right)
\left( \sup_{\xi\in\R^d}\frac{\abs{\Omega^\eps(\xi)}}{\abs{\xi}^4}\right)
\sup_{\xi\in\R^d}\int_{\Sp^{d-1}} b^K_{\eps}(\sigma\cdot \hat \xi)\frac12
\left(  \frac{\abs{\xi^+}^2 \abs{\xi^-}^4}{\abs{\xi}^2 \abs{\xi}^4}
+\frac{\abs{\xi^+}^4 \abs{\xi^-}^2}{\abs{\xi}^4 \abs{\xi}^2}     \right)d\sigma \\
&\leq {\abs{h_t^\eps}}_2 {\abs{d_t^\eps}}_2 \sup_{\xi\in\R^d}\int_{\Sp^{d-1}} b^K_{\eps}(\sigma\cdot \hat \xi)\frac12
\left(  \abs{\hat \xi^-}^4 + \abs{\hat \xi^-}^2 \right)d\sigma\\
&\leq \beta_K e^{-(1-\eta)\lambda t}\,  M_4( f_0 + g_0)^{1-\eta} \, {\abs{h_0}}_2 {\abs{d_0}}^{1+\eta}_2 
\end{aligned}
$$
where we have used the estimates of Lemmas \ref{lem:A2i} and \ref{lem:A4b}. We compute $\beta_K$
$$
\begin{aligned}
\beta_K &= \int_{\Sp^{d-1}} b^K_{\eps}(\sigma\cdot \hat \xi)\frac12
\left(  \abs{\hat \xi^-}^4 + \abs{\hat \xi^-}^2 \right)d\sigma \\
&= 
\int_{\Sp^{d-1}} b^K_{\eps}(\sigma\cdot \hat \xi)\frac{1}{2}
\left(  1 -\sigma\cdot\hat \xi   \right)d\sigma 
-\int_{\Sp^{d-1}} b^K_{\eps}(\sigma\cdot \hat \xi)\frac{1}{8}
\left(  1 -(\sigma\cdot\hat \xi)^2   \right)d\sigma \\
&\xrightarrow[K\to\infty]{} \bar\Lambda_\eps :=\frac{\Lambda_\eps}{2} - \frac{\bar\lambda_\eps}{4}   \in (0, \infty)\\
&\xrightarrow[\eps\to 0]{} \bar\Lambda := \frac{\Lambda}{2} - \frac{\bar\lambda}{4} \in (0, \infty) ,
\end{aligned}
$$
 and we have the same estimate for $T_3$.

After these computations we obtain
\begin{equation*}\label{}
\begin{aligned}
\frac{d}{dt}\frac{\abs{\Psi^\eps(\xi)}}{\abs{\xi}^6} + K \frac{\abs{\Psi^\eps(\xi)}}{\abs{\xi}^6}
\leq \alpha_K \sup_{\xi\in\R^d} \frac{\abs{\Psi^\eps(\xi)}}{\abs{\xi}^6} +
 2\beta_K e^{-(1-\eta)\lambda t} \,  M_4( f_0 + g_0)^{1-\eta} \,{\abs{d_0}}^{2+\eta}_2  + r_\eps^K
\end{aligned}
\end{equation*}
and by Gronwall's lemma
\begin{equation*}\label{}
\begin{aligned}
\sup_{\xi\in\R^d}\frac{\abs{\hat \psi_t(\xi)}}{\abs{\xi}^6} \leq
2\beta_K  \,  M_4( f_0 + g_0)^{1-\eta} \,  {\abs{d_0}}^{2+\eta}_2  \left( \frac{e^{-(1-\eta)\lambda t} - e^{(\alpha_K-K)t}}{K-\alpha_K - (1-\eta)\lambda}   \right) 
+ r_\eps^K\left( \frac{1 - e^{(\alpha_K-K)t}}{K-\alpha_K}\right).
\end{aligned}
\end{equation*}
We conclude by relaxing the cut-off parameter $K\to\infty$ and choosing $(1-\eta)\lambda \in (0,\bar\alpha_\eps)$, therefore we have
\begin{equation}\label{eq:psieps}
{\abs{\psi_t^\eps}}_6 \leq C_\eta\, \bar\Lambda_\eps \, e^{-(1-\eta)\lambda t} \,  M_4( f_0 + g_0)^{1-\eta} \,{\abs{d_0}}^{2+\eta}_2 .
\end{equation}

\medskip
\noindent
{\it Step 2.} 
Consider the solutions $f$, $g$, $h$ and $r$ of \eqref{equations-A4-0-Landau} .  

Let us choose $\lambda_1$ such that $0<(1-\eta)\lambda_1<  \bar \alpha$, where $\bar\alpha_\eps \to \bar\alpha\in(0,\infty)$ as $\eps\to 0$. Then we recover \eqref{eq:psieps} with the exponent $\lambda_1$ with does not depend on $\eps$. Hence, passing to the limit $\eps\to 0$, we have $g^\eps-f^\eps-h^\eps-u^\eps \to g-f-h-u$ (grazing collisions limit \cite{Vi2}), and in the right-hand side of \eqref{eq:psieps} we have $\bar\Lambda_\eps \to \bar\Lambda$. Then
$$
|g_t - f_t - h_t - r_t|_6 \le C_\eta \, \bar\Lambda \, e^{-(1-\eta)\lambda_1 t}\, M_4( f_0 + g_0) ^{1/3} \,{\abs{g_0-f_0}}^{1+\eta}_2.
$$

\end{proof}

Therefore the semigroup of the Landau equation 
$$
S^{\infty}_t \in C^{2,\eta}_{\Lambda_2} \cap C^{1,(1+2\eta)/3}_{\Lambda_2} \cap C^{0,(2+\eta)/3}_{\Lambda_2}( \PPP_{\GG_1,{\bf r}}; \PPP_{\GG_2}),
$$
where $\PPP_{\GG_2}$ is defined as $\PPP_{\GG_1}$ but endowed with the distance associated to the norm $\|\cdot\|_{\GG_2} = \Nt \cdot \Nt_{6}$ (see Definitions~\ref{def:toscani} 
and \ref{def:general-toscani}), with $\Lambda_2(f) : = M_4(f)^{1/3} = \Lambda_1(f)^{1/3}$. Moreover there exists a constant $C_4 >0$ such that one has
\beqn\label{eq:Sinfty}
\sup_{{\bf r} \in \RRR_{\GG_1}} \int_0^\infty \left( 
\left[ S^{\infty}_t \right]_{C^{2,0}_{\Lambda_2}}
+ \left[ S^{\infty}_t \right]_{C^{1,0}_{\Lambda_2}}^2   
 \right) \,dt \leq C_{4},
\eeqn
which proves Assumption {\bf (A4)}.

\begin{rem}In fact, we can deduce that 
$$
\sup_{{\bf r} \in \RRR_{\GG_1}} \int_0^\infty \left( 
\left[ S^{\infty}_t \right]_{C^{1,(1+2\eta)/3}_{\Lambda_2}}
+ \left[ S^{\infty}_t \right]_{C^{0,(2+\eta)/3}_{\Lambda_2}}^2+
\left[ S^{\infty}_t \right]_{C^{2,0}_{\Lambda_2}}
+ \left[ S^{\infty}_t \right]_{C^{1,0}_{\Lambda_2}}^2   
 \right) \,dt \leq C_{4}.
$$
However, coming back to the proof of Theorem~\ref{thm:abstract} and from the proof of {\bf (A3)} in Lemma~\ref{lem:A3}, we observe that we only need $[\Phi]_{C^{2,0}}$ instead of $[\Phi]_{C^{1,\eta}} + [\Phi]_{C^{2,0}}$, so that $\eqref{eq:Sinfty}$ is sufficient.
\end{rem}

\subsection{Proof of assumption (A5)}

We define the space of probability measures $\PPP_{\GG_3} := \PPP_2(\R^d) = \{ f\in \PPP(\R^d);\; M_2(f) < \infty \}$ endowed with the distance $\dist_{\GG_3}=W_2$, and the constraints associated to the momentum and energy: ${\bf m}_{\GG_3}(v) = (|v|^2,v)$ and $\RRR_{\GG_3} = \{(r,\bar r)\in \R_+ \times \R^d ;\; r = |\bar r|^2\}$, so that $\PPP_{\GG_3, \mathbf r} = \{ f \in \PPP_2(\R^d); \; \langle f , |v|^2 \rangle = r, \, \langle f , v_\alpha \rangle = \bar r_\alpha \text{ for } \alpha= 1, \dots, d \}$ for any $\mathbf r \in \mathbf R_{\GG_3}$. 
The following lemma proves {\bf (A5)} with $\FF_3 = \mathrm{Lip}(\R^d)$.

\begin{lemma}\label{lemma:A5}
Let $f_0,g_0$ have the same momentum and energy, and consider $f_t=S_t^{\infty} (f_0)$, $g_t=S_t^{\infty} (g_0)$ the respective solutions to the Landau equation with Maxwellian molecules. Then
\begin{equation}
\sup_{t\geq 0} W_2(f_t,g_t)\leq W_2(f_0,g_0).
\end{equation}
\end{lemma}

\begin{proof}[Proof of Lemma \ref{lemma:A5}]

Consider $f^\eps_t,g^\eps_t$ the solutions of the Boltzmann equation with kernel $b_\eps$ satisfying the grazing collisions \eqref{eq:grazing} and with initial data $f_0$ and $g_0$, respectively. We know from \cite{tanaka} that
$$
\sup_{t\geq 0} W_2(f^\eps_t,g^\eps_t)\leq W_2(f_0,g_0)
$$
We know also from \cite{Vi2} that $f^{\varepsilon}_t$ converges weakly in $L^1$ to a weak solution $f_t$ of the Landau equation (grazing collisions limit).
Moreover, both equations conserve energy so we have, for all $\varepsilon >0$,
$$
\int |v|^2 f^{\varepsilon}_t(v)\,dv = \int |v|^2 f_t(v)\,dv = \int |v|^2f_0(v)\,dv .
$$
Using the fact that the Wasserstein distance $W_2$ is equivalent to the weak convergence in $\PPP(\R^d)$ plus the convergence of the second order moment (see \cite{VillaniOTO&N}), and writing
$$
\begin{aligned}
W_2(f_t,g_t) &\leq W_2(f_t,f_t^{\varepsilon}) + W_2(f_t^{\varepsilon},g_t^{\varepsilon}) + W_2(g_t,g_t^{\varepsilon}) ,
\end{aligned}
$$
we obtain passing to the limit $\varepsilon\to 0$ that $W_2(f_t,f_t^{\varepsilon})\to 0$, 
$W_2(g_t,g_t^{\varepsilon})\to 0$ and hence
$$
\sup_{t\geq 0} W_2(f_t,g_t)\leq W_2(f_0,g_0).
$$
\end{proof}


\subsection{Proof of Theorem \ref{thm:Wchaos}}\label{ssec:Wchaos}
The proof is a consequence of Theorem~\ref{thm:chaos}, some results on different forms of measuring chaos from \cite{HaurayMischler} and quantitative estimates on the chaoticity of initial data from \cite{KC}.

\begin{proof}[Proof of Theorem~\ref{thm:Wchaos} (1)]
Thanks to Theorem~\ref{thm:chaos}, taking $\ell=2$, we have for all $\phi=\phi_1\otimes\phi_2 \in \FF^{\otimes 2}$ that
$$
\sup_{t\ge 0} \frac{|\left\langle \Pi_2(F^N_t) - f^{\otimes 2}_t ,  \phi \right\rangle |}{\|\phi\|_\FF}
\le C \left( \WW_{W_2}\left(\pi^N_P F^N_0,\delta_{f_0} \right) + \frac1N \right),
$$
where we recall that $\|\phi\|_\FF = \int (1+|\xi|^6)|\hat \phi(\xi)|$. Then we observe that, for $r>0$, applying Cauchy-Schwarz inequality,
$$
\bal
\|\phi_1\|_\FF &= \int (1+|\xi|^6)(1+|\xi|^2)^{r/2} \, |\hat \phi_1(\xi)| (1+|\xi|^2)^{-r/2} \, d\xi \\
&\le C \left( \int (1+|\xi|^2)^{6+r} \,|\hat \phi_1(\xi)|^2  \right)^{1/2}
      \left( \int (1+|\xi|^2)^{-r}   \right)^{1/2}.      
\eal
$$
The first integral in the right-rand side is the norm $\|\phi_1\|_{H^{6+r}}$ and the second one is finite if $2r>d$. We have then $H^s \subset \FF$ for $s>6+d/2$ which implies 
\begin{equation}\label{eq:chaosHs}
\begin{aligned}
\sup_{t\ge 0} {\Norm{\Pi_2(F^N_t) - f^{\otimes 2}_t}}_{H^{-s}} 
&\le C \left( \WW_{W_2}\left(\pi^N_P F^N_0,\delta_{f_0} \right) + \frac1N \right).
\end{aligned}
\end{equation}
Let us denote $M_k = M_k(\Pi_2(F^N_t)) + M_k(f^{\otimes 2}_t)$.
Thanks to \cite{HaurayMischler}, for any $0<\alpha<k(dk+d+k)^{-1}$ there exists $C:=C(\alpha,d,s,M_k)$ such that 
\begin{equation*}
\frac{W_1(F^N_t,f^{\otimes N}_t)}{N} \le C\left( {\Norm{\Pi_2(F^N_t) - f^{\otimes 2}_t}}_{H^{-s}}^{\frac{\alpha k}{d+ks}} + N^{-\frac{\alpha}{2}} \right),
\end{equation*}
which implies with \eqref{eq:chaosHs}
\begin{equation}\label{eq:W1}
\frac{W_1(F^N_t,f^{\otimes N}_t)}{N} \le C\left( \WW_{1,W_2}(\pi^N_p F^N_0, f^{\otimes N}_0)^{\frac{\alpha k}{d+ks}} + N^{-\frac{\alpha}{2}} \right).
\end{equation}
Now, we have just to estimate the first term of the right-hand side of \eqref{eq:chaosHs}.

We have from \cite[Proof of Theorem 8]{KC} that for any $0< \beta < (7d+6)^{-1}$ there exists $C=C(\beta)$ such that
$$
\WW_{1,W_2}(\pi^N_P F^N_0, \delta_{f_0}) \le C\, N^{-\beta}.
$$
We assumed that $M_6(f_0)$ is finite, which implies by construction that $M_6(\Pi_2(F^N_0))$ is also finite. Then, for all $t\ge 0$ we have $M_6(f_t)$ finite (see \cite{Vi1}) and $M_6(\Pi_2(F^N_t))$ also finite (see Lemma~\ref{lem:A1ii}). We can conclude gathering the last equation with \eqref{eq:W1}, $k=6$ and $s>6+d/2$.
\end{proof}

Using this result, we can prove the second part of the theorem following \cite{MMchaos}.

\begin{proof}[Proof of Theorem~\ref{thm:Wchaos} (2)]
We split the expression into
$$
\frac{W_1(F^N_t, \gamma^N)}{N} = \frac{W_1(F^N_t, f_t^{\otimes N})}{N} + \frac{W_1(\gamma^{\otimes N}, \gamma^N)}{N} + W_1(f_t,\gamma),
$$
where $\gamma$ is the equilibrium Gaussian probability with zero momentum and energy $\EE = \int |v|^2 d\gamma$. For the first term we have from point {\it (1)} that for all $\epsilon < 9[(7d+6)^2(d+9)]^{-1}$ there exists $C_\epsilon$ such that
$$
\frac{W_1(F^N_t, f_t^{\otimes N})}{N} \le C_\epsilon \, N^{-\epsilon}
$$
The second term can be estimated by \cite[Theorem 18]{KC}
$$
\frac{W_1(\gamma^{\otimes N}, \gamma^N)}{N} \le C N^{-\theta},
$$
for some $\theta > \eps$.
For the third term, thanks to \cite[Theorem 6]{Vi1} we have 
$$
W_1(f_t,\gamma) \le \| (f_t - \gamma)\langle v \rangle \|_{L^1} \le C e^{-\lambda t}.
$$
for contants $C>0$ and $\lambda>0$.
Finally, putting together these estimates it follows
\beqn\label{eq:W1a}
\frac{W_1(F^N_t, \gamma^N)}{N} \le C'_\eps ( N^{-\epsilon} + e^{-\lambda t}).
\eeqn

Moreover, consider $h^N_t$ the Radon-Nikodym derivative of $F^N_t$ with respect to $\gamma^N$, i.e. $h^N_t = dF^N_t/ d\gamma^N$. Thanks to \cite{KL}, for all $N\in \N^*$ and $t\ge 0$, it holds
$$
{\| h^N_t - 1 \|}_{L^2(\SS^N(\EE),d\gamma^N)} \le e^{-\lambda_1 t}{\| h^N_0 - 1 \|}_{L^2(\SS^N(\EE),d\gamma^N)},
$$
where $\lambda_1>0$. Since $F^N_0 = [f^{\otimes N}_0]_{\SS^N(\EE)}$ and $f_0 \in \PPP_6(\R^d)$, it is possible to bound the right-hand side by
$$
{\| h^N_0 - 1 \|}_{L^2(\SS^N(\E),d\gamma^N)} \le A^N,
$$
with $A>1$ that depends on $f_0$. Hence we deduce, with $\phi: \R^{dN}\to \R$,
$$
\bal
W_1(F^N_t, \gamma^N) &= \sup_{\| \phi \|_{C^{0,1}} \le 1} \int_{\R^{dN}} \phi (dF^N - d\gamma^N) \\
&\le   \int_{\R^{dN}}  \sum_{j=1}^N |v_j| \left|dF^N - d\gamma^N\right|\\
&\le N \EE^{1/2} {\| h^N_t - 1 \|}_{L^1(\SS^N(\EE),d\gamma^N)} \\
&\le N \EE^{1/2} {\| h^N_t - 1 \|}_{L^2(\SS^N(\EE),d\gamma^N)},
\eal
$$
which implies
\beqn\label{eq:W1b}
\frac{W_1(F^N_t, \gamma^N)}{N} \le A^N e^{-\lambda_1 t}.
\eeqn
Define $N(t)$ by $N(t) := \lambda_1 t\, (2\log A)^{-1}$ for some $\delta>0$. Then, choosing \eqref{eq:W1a} for $N> N(t)$ and \eqref{eq:W1b} for $N\le N(t)$ it yields, for all $N\in \N^*$ and $t\ge 0$,
$$
\frac{W_1(F^N_t,\gamma^N)}{N} \le p(t) := \min\left\{ C'_\eps \left(N(t)^{-\eps} + e^{-\lambda t} \right)  , e^{-\frac{\lambda_1}{2} t}    \right\},
$$
with a polynomial function $p(t)\to 0$ as $t\to\infty$.

\end{proof}

%
%
%


\section{Entropic chaos}\label{sec:ent}

We can define the master equation \eqref{Lmaster2} on $\R^{dN}$ or $\SS^N(\EE)$ thanks to the conservation of momentum and energy, hence for $g^N \in \Psym(\R^{dN})$ and $f^N \in \Psym(\SS^N(\EE))$ we have
\bear
\partial_t \left\langle g^N , \psi \right\rangle &=& \left\langle g^N , G^N\psi \right\rangle , \qquad \forall\, \psi\in C^2_b(\R^{dN}) \label{eq:R} \\
\partial_t \left\langle f^N , \phi \right\rangle &=& \left\langle f^N , G^N\phi \right\rangle , \qquad \forall\, \phi\in C^2_b(\SS^N(\EE)) \label{eq:S},
\eear
where $G^N$ is given by \eqref{Lmaster2}.

Suppose that $g^N$ is absolutely continuous with respect to the Lebesgue measure 
(and we still denote by $g^N$ its Radon-Nikodym derivative). Taking $\psi = \log g^N$ in \eqref{eq:R}, we obtain an equation for the entropy of $g^N$, i.e. $H(g^N) := \int_{\R^{dN}} g^N \log g^N \, dV$,
\beqn\label{eq:entR}
\begin{aligned}
&\frac{d}{dt}\int_{\R^{dN}} g^N \log g^N \, dV \\
&\qquad= -\frac{1}{2N} \, \sum_{i,j} \int_{\R^{dN}} a(v_i-v_j) \left(\frac{\nabla_i g^N}{g^N} - \frac{\nabla_j g^N}{g^N}\right)\cdot\left(\frac{\nabla_i g^N}{g^N} - \frac{\nabla_j g^N}{g^N} \right) g^N \,dV \le 0,
\end{aligned}
\eeqn
since $a$ is nonnegative.

Considering now $f^N$ absolutely continuous with respect to $\gamma^N$, the uniform probability measure on $\SS^N(\EE)$, and denoting by $h^N := df^N/d\gamma^N$ its derivative, we want to obtain the equation satisfied by the relative entropy of $f^N$ with respect to $\gamma^N$, given by
\beqn\label{rel-ent}
H(f^N | \gamma^N) := \int_{\SS^N} h^N \log h^N \, d\gamma^N.
\eeqn
For this purpose we could take $\phi = \log h^N$ in \eqref{eq:S}, but we have to give a meaning to
$\nabla_i h^N$ for a function $h^N$ defined on $\SS^N$.

Let us consider $h$ a function on $\SS^N(\EE)$ and we define $\widetilde h$ on $\R^{dN}$ by
\begin{equation}\label{eq:tildeh}
\widetilde h (V) = \rho(\EE(V))\, \eta(\MM(V)) \, h\left(\EE\, \frac{V-\MM(V)}{\EE(V)}\right),\qquad \forall\, V \in \R^{dN}.
\end{equation}
where $\EE(V) = N^{-1}\sum_{i=1}^N |v_i-\MM(V)|^2$, $\MM(V) = N^{-1}\sum_{i=1}^N v_i$ and the functions $\rho$ and $\eta$ are smooth.

Denoting by $\nabla_{\SS^N}$ the gradient with respect to $\SS^N(\EE)$ and by $\nabla_{\perp}$ the gradient with respect to its orthogonal space $(\SS^N)^{\perp}$, we can decompose the 
gradient on $\R^{dN}$ 
\beqn\label{grad}
\nabla_{\R^{dN}} \widetilde h = \nabla_{\perp} \widetilde h + \nabla_{\SS^N} \widetilde h
=(\nabla_{\perp} \rho\eta)\, h+  \rho\eta\, \nabla_{\SS^N}  h = (\nabla_{\perp} \log(\rho\eta))\,\widetilde  h+  \rho\eta\, \nabla_{\SS^N}  h .
\eeqn
For $\widetilde h$ we can compute $\nabla_i \widetilde h \in \R^d$, for $1\le i\le N$, as 
$$
\nabla_i \widetilde h =\left( \partial_{v_{i,\alpha}} \widetilde h    \right)_{1\le \alpha\le d} = \left( \nabla_{\R^{dN}} \widetilde h \cdot e_{i,\alpha}   \right)_{1\le \alpha\le d},
$$
where $(e_{i,\alpha})_{j,\beta} = \delta_{ij}\delta_{\alpha\beta} \in \R^{dN}$. Hence by \eqref{grad}, for all $1\leq i\leq N$ and all $1\le \alpha \le d$,
$$
\partial_{v_{i,\alpha}} \widetilde h  = (\nabla_{\perp} \log(\rho\eta) \cdot e_{i,\alpha} )\,\widetilde  h+  \rho\eta\, (\nabla_{\SS^N}  h\cdot e_{i,\alpha}).
$$
Now, observing that $(\nabla_{\perp} \log(\rho\eta) \cdot (e_{i,\alpha}-e_{j,\alpha}) )_{1\le \alpha\le d}$ is proportional to $(v_i-v_j)$ and using that $a(z)z=0$ for all $z\in\R^d$, we can evaluate the expression
$$
a(v_i-v_j) \left({\nabla_i \widetilde h} - {\nabla_j \widetilde h}\right)\cdot\left({\nabla_i \widetilde h} - {\nabla_j \widetilde h} \right) = 
(\rho\eta)^2 \, a(v_i-v_j) (\nabla_{\SS^N_i}  h - \nabla_{\SS^N_j}  h)\cdot(\nabla_{\SS^N_i}  h - \nabla_{\SS^N_j}  h ),
$$
where we define
\beqn\label{eq:gradSi}
\nabla_{\SS^N_i}  h = \left( \nabla_{\SS^N} h \cdot e_{i,\alpha}  \right)_{1\le \alpha\le d}.
\eeqn

Since we have the following Fubini-like theorem for Boltzmann's spheres (see \cite{Einav,KC})
\beqn\label{fubbini}
\bal
&\int_{\R^{dN}} \rho(\EE(V)) \,\eta(\MM(V)) \, A\left(h\left(\EE\, \frac{V-\MM(V)}{\EE(V)}\right)\right) \, dV \\
&\qquad= 
\left(\int_{\R^+\times \R^d} B\left(\rho(\EE),\eta(\MM)\right)\, d\EE\, d\MM \right) \left(\int_{\SS^N(\EE)} A(h) \, d\gamma^N\right),
\eal
\eeqn
for some functions $A$ and $B$, thanks to \eqref{eq:entR} with $h = h^N$ and $\widetilde h = g^N$, we obtain the equation for the relative entropy $H(f^N | \gamma^N)$,
\beqn\label{eq:RelativeEntropy}
\begin{aligned}
&\frac{d}{dt}\int_{\SS^N(\EE)} h^N \log h^N \, d\gamma^N  \\
&\qquad= -\frac{1}{2N} \, \sum_{i,j} \int_{\SS^N(\EE)} a(v_i-v_j) \left(\frac{\nabla_{\SS^N_i} h^N}{h^N} - \frac{\nabla_{\SS^N_j} h^N}{h^N}\right)\cdot\left(\frac{\nabla_{\SS^N_i} h^N}{h^N} - \frac{\nabla_{\SS^N_j} h^N}{h^N} \right) h^N \,d\gamma^N \\
&\qquad =: - D^N(F^N)\le 0,
\end{aligned}
\eeqn
and $D^N$ is called the entropy-production functional. 
This implies
\begin{equation}\label{ent1}
\frac1N \, H(F^N_t | \gamma^N) + \int_0^t \frac1N\, D^N(F^N_s)\,ds 
=\frac1N \, H(F^N_0 | \gamma^N) .
\end{equation}
Moreover for the limit equation we have \cite{Vi1}
$$
\begin{aligned}
\frac{d}{dt} H(f) &:= \frac{d}{dt}\int  f \log f \, dv =: D(f) \\
&= 
-\frac{1}{2}\int ff_* \,a(v-v_*)
\left( \frac{\nabla f}{f} -  \frac{\nabla_{\!*} f_*}{f_*} \right)
\cdot \left( \frac{\nabla f}{f} -  \frac{\nabla_{\!*} f_*}{f_*} \right) \,dv \,dv_* ,
\end{aligned}
$$
and then for the relative entropy $H(f|\gamma) = \int (f/\gamma)\log(f/\gamma)\,\gamma(dv)$ we obtain
\begin{equation}\label{ent2}
H(f_t|\gamma) + \int_0^t D(f_s) \, ds = H(f_0|\gamma).
\end{equation}

We are able now to prove the following result, which will be useful in the sequel.

\begin{lemma}\label{lem:semicont}
If $F^N$ is $f$-chaotic, then 
$$
H(f|\gamma) \le \liminf_{N\to\infty} \frac1N\, H(F^N | \gamma^N) 
\quad \text{and}\quad
D(f) \le \liminf_{N\to\infty} \frac1N\, D^N(F^N).
$$
\end{lemma}

\begin{proof}[Proof of Lemma \ref{lem:semicont}]
The lower semicontinuity property of the relative entropy is proved in \cite[Theorem 21]{KC}, thus we prove only the second inequality.

Let us denote $\nabla_{12} = \nabla_1 - \nabla_2$, $\nabla_{\SS^N_{12}} = \nabla_{\SS^N_1}-\nabla_{\SS^N_2}$, and for all $x,y,z \in \R^d$ we denote $a(z)xy = (a(z)x) \cdot y$. 
Since a is nonnegative, considering a function $\varphi : \R^{2d}\to\R^d$, we have
$$
a(v_1-v_2)\left(\nabla_{12}\log f_1f_2-\frac{\varphi}{2}\right)\left(\nabla_{12}\log f_1f_2-\frac{\varphi}{2}\right)\ge 0,
$$
which gives the following representation for $D(f)$,
$$
\bal
D(f) &= \frac12 \sup_{\varphi : \R^{2d}\to\R^d} \iint a(v_1-v_2) \left[ (\nabla_{12}\log f_1f_2)\varphi - \frac{\varphi\varphi}{4}\right] \, f_1 f_2\, dv_1\, dv_2 \\
&= \frac12 \sup_{\varphi : \R^{2d}\to\R^d} \iint  \left\{ -\nabla_{12}\cdot(a(v_1-v_2)\varphi) - a(v_1-v_2)\,\frac{\varphi\varphi}{4}\right\} \, f_1 f_2\, dv_1\, dv_2
\eal
$$
where $f_1 = f(v_1)$ and $f_2=f(v_2)$.
Let $\eps>0$ and choose $\varphi = \varphi(v_1,v_2) : \R^{2d}\to\R^d$ such that
$$
D(f) - \eps \le \frac12\iint  \left\{ - \nabla_{12}  (a(v_1-v_2) \varphi) - a(v_1-v_2)\,\frac{\varphi\varphi}{4} \right\} \, f_1 f_2\, dv_1\,dv_2.
$$
For the $N$-particle entropy-production $D^N$ defined in \eqref{eq:RelativeEntropy}, we have by symmetry
$$
\bal
&\frac1N D^N(F^N) \\
&\qquad
= \frac{N(N-1)}{N^2} \, \frac12\int_{\SS^N}a(v_1-v_2) 
\left( \frac{\nabla_{\SS^N_1} h^N}{h^N} - \frac{\nabla_{\SS^N_2} h^N}{h^N} \right) \cdot 
\left( \frac{\nabla_{\SS^N_1} h^N}{h^N} - \frac{\nabla_{\SS^N_2} h^N}{h^N} \right) h^N \, d\gamma^N \\
&\qquad
=: \frac{N(N-1)}{N^2} D^N_{12}(F^N),
\eal
$$
and then $\liminf_{N\to\infty} N^{-1} D^N(F^N) \ge  \liminf_{N\to\infty}D^N_{12}(F^N)$.
For $\Phi : \R^{dN}\to \R^{d}$, $\Phi\in C^1_b$, we have, with $F^N = h^N \gamma^N$,
$$
\begin{aligned}
D^N_{12}(F^N) &= \frac12\int_{\SS^N} a(v_1-v_2) \nabla_{\SS^N_{12}} \log h^N \cdot \nabla_{\SS^N_{12}} \log h^N \, h^N\,d\gamma^N\\
&= \frac12\sup_{\Phi : \R^{dN}\to \R^{d}}  \int_{\SS^N} a(v_1-v_2) \left( \nabla_{\SS^N_{12}} \log h^N  \Phi - \Phi \Phi /4   \right) \, h^N\,d\gamma^N \\
&= \frac12\sup_{\Phi}  \left\{ \int_{\SS^N} a(v_1-v_2) \nabla_{\SS^N_{12}}  h^N  \Phi \, d\gamma^N - \int_{\SS^N} a(v_1-v_2) \, \frac{\Phi \Phi}{4} \, h^N\,d\gamma^N \right\}.
\end{aligned}
$$

Choosing  $\Phi(V) = \varphi(v_1,v_2)$ we obtain, using \eqref{eq:gradSi},
\beqn\label{DN}
\begin{aligned}
D^N_{12}(F^N) 
&\ge \frac12\int_{\SS^N} a(v_1-v_2) \nabla_{\SS^N_{12}}  h^N \, \varphi \, d\gamma^N 
-\frac12 \int_{\SS^N} a(v_1-v_2) \, \frac{\varphi \varphi}{4} \, h^N\,d\gamma^N\\
&\ge \frac12\sum_{\alpha,\beta=1}^d \int_{\SS^N}  \nabla_{\SS^N}  h^N \cdot \left[(e_{1,\alpha}-e_{2,\alpha})\, a_{\alpha\beta}(v_1-v_2) \,\varphi_{\beta} \right]\, d\gamma^N \\
&\quad
-\frac12 \int_{\SS^N} a(v_1-v_2) \, \frac{\varphi \varphi}{4} \, h^N\,d\gamma^N.
\end{aligned}
\eeqn
We need an integration by parts formula for the first term on the right-hand side, thanks to \cite[Lemma 22]{KC}, for a function $A$ and a vector field $\Psi$, we have
$$
\int_{\SS^N} \left\{\nabla_{\SS^N} A(V) \cdot \Psi(V) + A(V)\Div_{\SS^N} \Psi(V) - \frac{d(N-1)-1}{dN}\, A(V)\, \Psi(V)\cdot V \right\}\, d\gamma^N(V)=0,
$$
with
\beqn\label{div}
\Div_{\SS^N} \Psi(V) = \Div\Psi(V) - \frac1N \sum_{i,j=1}^N\sum_{\beta=1}^d \partial_{v_{i,\beta}} \Psi_{j,\beta}(V) - \sum_{j=1}^N\sum_{\beta=1}^d V\cdot \nabla\Psi_{j,\beta}\, \frac{v_{j,\beta}}{|V|^2}.
\eeqn
Taking $\Psi(V) = (e_{1,\alpha}-e_{2,\alpha})\, a_{\alpha\beta}(v_1-v_2) \,\varphi_{\beta}$ we obtain
$$
\bal
&\int_{\SS^N}  \nabla_{\SS^N}  h^N \cdot (e_{1,\alpha}-e_{2,\alpha})\, a_{\alpha\beta}(v_1-v_2) \,\varphi_{\beta} \, d\gamma^N \\
&\qquad= 
-\int_{\SS^N}   h^N \Div_{\SS^N}\left[ (e_{1,\alpha}-e_{2,\alpha})\, a_{\alpha\beta} \,\varphi_{\beta} \right] d\gamma^N
+\frac{d(N-1)-1}{dN} \int_{\SS^N} h^N\, a_{\alpha\beta}\, \varphi_\beta \, (e_{1,\alpha}-e_{2,\alpha})\cdot V \, d\gamma^N.
\eal
$$
Since $(e_{1,\alpha}-e_{2,\alpha})\cdot V = (v_{1,\alpha} - v_{2,\alpha})$, when performing the summation $\alpha,\beta = 1$ to $d$ in the second term of the right-hand side of last equation, we obtain
$$
\int_{\SS^N} h^N a(v_1-v_2) (v_1 - v_2)\, \varphi\, d\gamma^N = 0.
$$
For the first term, thanks to \eqref{div},
$$
\bal
&\sum_{\alpha,\beta=1}^{d}\int_{\SS^N}   h^N \Div_{\SS^N}\left[ (e_{1,\alpha}-e_{2,\alpha})\, a_{\alpha\beta} \,\varphi_{\beta} \right] d\gamma^N \\
&\qquad
= \int_{\SS^N} \nabla_{12}\cdot (a(v_1-v_2)\,\varphi)\,  h^N \, d\gamma^N \\
&\qquad
-\sum_{\alpha,\beta=1}^{d}\frac{1}{|V|^2} \int_{\SS^N} \left\{ v_1\cdot \nabla_1(a_{\alpha\beta}\, \varphi_{\beta}) + v_2\cdot \nabla_2(a_{\alpha\beta}\, \varphi_{\beta})  \right\}
(v_{1,\alpha}-v_{2,\alpha})\,  h^N \, d\gamma^N 
\eal
$$

Getting back to \eqref{DN} with last expression, we split the integral over $(v_1,v_2)$ and $\SS^N(v_1,v_2) := \{ (v_3,\dots,v_N)\in \R^{d(N-2)};\; V \in \SS^N\}$, use that $|V|^2 = \EE N$ and $\int_{\SS^N(v_1,v_2)} h^N \, d\gamma^N = F^N_2$, which yields
\beqn\label{DN2}
\begin{aligned}
D^N_{12}(F^N) 
&\ge -\frac12    \iint \nabla_{12}\cdot (a(v_1-v_2)\,\varphi)\, F^N_2(v_1,v_2)\,dv_1\, dv_2
-\frac12 \iint a(v_1-v_2) \, \frac{\varphi \varphi}{4} \, F^N_2(v_1,v_2)\,dv_1\, dv_2 \\
&+O\left(\frac1N\right)\sum_{\alpha,\beta=1}^{d} \iint \left\{ v_1\cdot \nabla_1(a_{\alpha\beta}\, \varphi_{\beta}) + v_2\cdot \nabla_2(a_{\alpha\beta}\, \varphi_{\beta})  \right\}
(v_{1,\alpha}-v_{2,\alpha})\,  F^N_2(v_1,v_2)\,dv_1\, dv_2.
\end{aligned}
\eeqn
Passing to the limit $N\to\infty$, since $F^N_2 \wto f^{\otimes 2}$ we obtain
$$
\liminf_{N\to\infty}  D^N_{12}(F^N) \ge \frac12\int \left\{
- \nabla_{12}\cdot (a(v_1-v_2) \varphi)  -
a(v_1-v_2)\frac{\varphi\varphi}{4} \right\}\, f_1f_2\, dv_1\, dv_2 \ge D(f)-\eps
$$
and we conclude letting $\eps$ go to $0$.

\end{proof}

We define the Fisher information of $G\in \PPP(\R^{dN})$ that is absolutely continuous with respect to the Lebesgue measure by 
$$
I(G) := \int_{\R^{dN}} \frac{|\nabla_{\R^{dN}} G|^2}{G}\, dV.
$$
Moreover, for a probability measure $F\in \PPP(\SS^N(\EE))$ absolutely continuous with respect to $\gamma^N$, we define the relative Fisher's information by
\beqn\label{rel-fisher}
I(F | \gamma^N) := \int_{\SS^N(\EE)} \frac{|\nabla_{\SS^N} h|^2}{h}\, d\gamma^N,\quad
h = \frac{dF}{d\gamma^N},
\eeqn
where $\nabla_{\SS^N}$ stands for the gradient on $\SS^N(\EE)$.

We can now give the following result.

\begin{lemma}\label{lem:fisher}
Let $F^N_0\in \Psym(\SS^N(\EE))$  with finite relative Fisher information $I(F^N_0 | \gamma^N)$. For all $t>0$ consider the solution $F^N_t$ of the Landau master equation \eqref{Lmaster2}. Then we have
$$
I(F^N_t | \gamma^N) \le I(F^N_0 | \gamma^N ) .
$$
\end{lemma}

\begin{proof}[Proof of Lemma \ref{lem:fisher}]

Denote $h^N_0 := dF^N_0/ d\gamma^N$ and, for all $t\ge 0$, $h^N_t := dF^N_t/ d\gamma^N$.
Consider $\widetilde h^N_t$ defined on $\R^{dN}$ given by \eqref{eq:tildeh} and define then $\widetilde F^N_t = \widetilde h^N_t \mathscr L$ a solution of \eqref{eq:R}, where $\mathscr L$ is the Lebesgue measure on $\R^{dN}$. Following \cite[Lemma 7.4]{MMchaos}, we claim that is enough to prove that
$$
I(\widetilde F^N_t) := \int_{\R^{dN}} \frac{|\nabla_{\R^{dN}} \widetilde h^N_t|^2}{\widetilde h^N_t}\, dV \le
\int_{\R^{dN}} \frac{|\nabla_{\R^{dN}} \widetilde h^N_0|^2}{\widetilde h^N_0}\, dV =: I(\widetilde F^N_0).
$$ 
Indeed, this equation, \eqref{grad}, \eqref{fubbini} and the conservation of momentum and energy yield
$$
\bal
I(\widetilde F^N_t) &= |\nabla_{\perp} \log(\rho\eta)|^2 
+ \left(\int_{\R^+\times \R^d} B\left(\rho(\EE),\eta(\MM)\right)\, d\EE\, d\MM \right) \left(\int_{\SS^N(\EE)} 
\frac{|\nabla_{\SS^N} h^N_t|^2}{h^N_t} \, d\gamma^N\right) \\
&\le
 |\nabla_{\perp} \log(\rho\eta)|^2 
+ \left(\int_{\R^+\times \R^d} B\left(\rho(\EE),\eta(\MM)\right)\, d\EE\, d\MM \right) \left(\int_{\SS^N(\EE)} 
\frac{|\nabla_{\SS^N} h^N_0|^2}{h^N_0} \, d\gamma^N\right)
= I(\widetilde F^N_0 ),
\eal
$$
which implies, dropping the time independent terms,
$$
I(F^N_t | \gamma^N) \le I(F^N_0 | \gamma^N).
$$

Now, let $F^N_{t,\eps} \in \Psym(\SS^N(\EE))$ be the solution of the Boltzmann master equation \eqref{eq:MasterBoltzmann}-\eqref{eq:BoltzmannGenerator} with collision kernel $b_\eps$ satisfying the grazing collisions assumptions \eqref{eq:grazing} and initial datum $F^N_0 \in \Psym(\SS^N(\EE))$. Then we have from \cite[Lemma 7.4]{MMchaos}, for all $t\ge 0$,
$$
I(\widetilde F^N_{t,\eps} ) \le I( \widetilde F^N_0  ),
$$
where $\widetilde F^N_0, \widetilde F^N_{t,\eps} \in \Psym(\R^{dN})$ are constructed as before.

Since $\widetilde F^N_{t,\eps}$ weakly converges towards $\widetilde F^N_t$ when $\eps\to0$ and the Fisher information functional is weakly lower semicontinuous, we obtain
$$
I(\widetilde F^N_t) \le \liminf_{\eps\to 0} I(\widetilde F^N_{t,\eps}) \le I( \widetilde F^N_0  )
$$
and that concludes the proof.

\end{proof}

Now, with the definitions of relative entropy \eqref{rel-ent}, relative Fisher information \eqref{rel-fisher} and the notion of entropic chaos, described below, we are able to state our main theorem of this section, concerning the propagation of entropic chaos. 

Let $F^N$ be a sequence of probability measures $\SS^N(\EE)$ such that $F^N_1$ weakly converges to $f$ in measure sense, for some $f\in \PPP(\R^d)$. We say that $F^N$ is entropically $f$-chaotic if
\beqn\label{entropic-chaos}
\frac{H(F^N | \gamma^N)}{N} \xrightarrow[N\to\infty]{} H(f | \gamma).
\eeqn
For more information on entropic chaos we refer to \cite{CCLLV,HaurayMischler,KC}.

\begin{thm}\label{thm:PropEntChaos}
Let $f_0\in \PPP(\R^d)$ and $F^N_0\in \Psym(\SS^N(\EE))$ that is $f_0$-chaotic. Consider then, for all $t>0$, the solution $F^N_t$ of the Landau master equation 
for Maxwellian molecules \eqref{Lmaster2}
with initial condition $F^N_0$, and the solution $f_t$ of the limit Landau equation for Maxwellian molecules \eqref{eq:landau}-\eqref{eq:landau2} with initial data $f_0$.

Then we have:
\begin{enumerate}

\item If $F^N_0$ is entropically $f_0$-chaotic, then for all $t>0$ $F^N_t$ is entropically $f_t$-chaotic, more precisely
$$
\frac1N \, H(F^N_t | \gamma^N) \longrightarrow H(f_t|\gamma)\qquad\text{as}\qquad N\to\infty.
$$

\item Consider $f_0\in \PPP_6(\R^d)$ with $I(f_0 | \gamma)<\infty$ and $F^N_0=[f^{\otimes N}_0]_{\SS^N(\EE)} \in \Psym(\SS^N(\EE))$. Then, for all $t>0$, $F^N_t$ is entropically $f_t$-chaotic, more precisely, for any $0<\eps < 18 [5(7d+6)^2(d+9)]^{-1}$ there exists a constant $C:=C(\eps)>0$ such that
$$
\sup_{t\ge 0} \left| \frac{1}{N} \,  H(F^N_t | \gamma^N) - H(f_t|\gamma)\right| 
\le C N^{-\eps}.
$$


\item Consider $f_0\in \PPP_6(\R^d)$ with $I(f_0 | \gamma)<\infty$ and $F^N_0=[f^{\otimes N}_0]_{\SS^N(\EE)} \in \Psym(\SS^N(\EE))$. Then for all $N$ it holds
$$
\frac1N \, H(F^N_t | \gamma^N) \le p(t),
$$
for some polynomial function $p(t)\to 0$ as $t\to\infty$.

\end{enumerate}

\end{thm}

\begin{proof}[Proof of Theorem \ref{thm:PropEntChaos} (1)]
The idea is from \cite{MMchaos}. Using \eqref{ent1}, \eqref{ent2} and the entropic chaoticity at initial time, one has
$$
\bal
&\frac1N \, H(f^N_t | \gamma^N) + \int_0^t \frac1N\, D^N(F^N_s)\,ds 
=\frac1N \, H(f^N_0 | \gamma^N) \\
&\quad \xrightarrow[N\to\infty]{} 
H(f_0|\gamma) = H(f_t|\gamma) + \int_0^t D(f_s) \, ds.
\eal
$$
By Lemma~\ref{lem:semicont} one also has
$$
\liminf_{N\to\infty} \left( H(f^N_t | \gamma^N) + \int_0^t \frac1N\, D^N(F^N_s)\,ds \right)
\ge H(f_t|\gamma) + \int_0^t D(f_s) \, ds,
$$
and we can conclude with these two last equations together with Lemma~\ref{lem:semicont}.

\end{proof}

\begin{proof}[Proof of Theorem \ref{thm:PropEntChaos} (2)]

From Lemma~\ref{lem:fisher} we know that, for all $t\ge0$, $N^{-1}I(F^N_t | \gamma^N)\le N^{-1} I(F^N_0 | \gamma^N)$ and the later one is bounded by 
construction \cite{KC}, we deduce then that the normalised relative Fisher's information $N^{-1}I(F^N_t | \gamma^N)$ is bounded. 
Since the limit Landau equation for Maxwellian molecules propagates moments and the Fisher's information's bound \cite{Vi1,Vi3}, 
we have, for all $t>0$, $M_6(f_t)$ and $I(f_t | \gamma)$ bounded.

We can then apply \cite[Theorem 31]{KC} to $F^N_t$ and we obtain that for any 
$\beta< (7d+6)^{-1}$ there exists $C'=C'(\beta)>0$ such that
\begin{equation*}
\begin{aligned}
\left|  \frac{1}{N} \,  H(F^N_t | \gamma^N) - H(f_t | \gamma)  \right| 
&\le C' \left( \frac{W_2(F^N_t,f^{\otimes N}_t)}{\sqrt{N}} + N^{-\beta} \right).
\end{aligned}
\end{equation*}
We have then to estimate the first term of the right-hand side. From \cite{HaurayMischler}, the following estimation holds,
\beqn\label{eq:W2W1}
\frac{W_2(F^N_t,f^{\otimes N}_t)}{\sqrt{N}} \le C \left( \frac{M_6(F^N_t,f^{\otimes N}_t)}{N} \right)^{1/10} \left( \frac{W_1(F^N_t,f^{\otimes N}_t)}{N} \right)^{2/5}
\eeqn
where $M_6(F^N_t,f^{\otimes N}_t) = M_6(F^N_t) + M_6(f^{\otimes N}_t)$. We observe that $N^{-1} M_6(F^N_t,f^{\otimes N}_t)$ is bounded since $N^{-1}M_6(f^{\otimes N}_t) = M_6(f_t)$, $N^{-1} M_6(F^N_t) \le C \, N^{-1} M_6(F^N_0)$ by Lemma~\ref{lem:A1ii} and $N^{-1} M_6(F^N_0)$ is bounded by construction, thanks to the assumption $M_6(f_0)$ finite. 

Finally, Theorem~\ref{thm:Wchaos} and last equation \eqref{eq:W2W1} imply that for any $\epsilon < 9[(7d+6)^2(d+9)]^{-1}$ and any $\beta< (7d+6)^{-1}$ there exists a positive constant $C=C(\epsilon,\beta)$ such that
\begin{equation}\label{eq:HGNt}
\begin{aligned}
\left|  \frac{1}{N} \,  H(F^N_t | \gamma^N) - H(f_t | \gamma)  \right| 
&\le C \left( N^{-2\epsilon/5} + N^{-\beta} \right),
\end{aligned}
\end{equation}
which concludes the proof.

\end{proof}

\begin{proof}[Proof of Theorem \ref{thm:PropEntChaos} (3)]
By the HWI inequality \cite[Theorem 30.21]{VillaniOTO&N}, for all $t\ge0$, we have
$$
\frac{H(F^N_t | \gamma^N )}{N} \le \frac{\pi}{2}\,\sqrt{\frac{I(F^N_t | \gamma^N)}{N} } \, \frac{W_2(F^N_t,\gamma^N)}{\sqrt N}.
$$
From Lemma~\ref{lem:fisher} we have $N^{-1} I(F^N_t | \gamma^N) \le N^{-1} I(F^N_0 | \gamma^N)\le C$ for some constant $C>0$ independent of $N$, by construction. Moreover, thanks to Lemma~\ref{lem:A1ii} and \eqref{eq:W2W1} we deduce
$$
\frac{W_2(F^N_t,\gamma^N)}{\sqrt N} \le C \left( \frac{W_1(F^N_t,\gamma^N)}{N} \right)^{2/5}.
$$
Gathering these two estimates with point $(2)$ in Theorem~\ref{thm:Wchaos} it follows
$$
\frac{H(F^N_t | \gamma^N )}{N} \le C \left( \frac{W_1(F^N_t,\gamma^N)}{N} \right)^{2/5} \le p(t),
$$
for a polynomial function $p(t)\to 0$ as $t\to\infty$.

\end{proof}


\bibliographystyle{acm}
\bibliography{bib-LandauChaos}

\end{document}